\documentclass{amsart}
\usepackage[utf8]{inputenc}
\usepackage{amsfonts,latexsym,amssymb,amsmath,amsthm,color}
\usepackage{enumerate}
\usepackage{captdef}
\usepackage[dvips]{graphicx} % LaTeX
\usepackage{enumerate,graphicx,graphpap}
\usepackage[matrix]{xy}
\usepackage{cases}
\usepackage{mathrsfs}
\usepackage{mathtools}
\input xy

\theoremstyle{plain}
\newtheorem{teor}{Theorem}[section]
\newtheorem{lema}[teor]{Lemma}
\newtheorem{coro}[teor]{Corollary}
\newtheorem{prop}[teor]{Proposition}

\theoremstyle{definition}
\newtheorem{defi}{Definition}[section]
\newtheorem{eje}{Example}[section]
\newtheorem{nota}[teor]{Remark}
\newtheorem*{gracias}{Acknowledgments}

\newtheoremstyle{teoremacita}% name of the style to be used
{3pt}% measure of space to leave above the theorem. E.g.: 3pt
{3pt}% measure of space to leave below the theorem. E.g.: 3pt
{\itshape}% name of font to use in the body of the theorem
{}% measure of space to indent
{\bfseries}% name of head font
{}% punctuation between head and body
{ }% space after theorem head; " " = normal interword space
{\thmname{#1}\thmnumber{ #2'}\thmnote{ #3}.}
\theoremstyle{teoremacita} \newtheorem*{teor*}{}

\newcommand{\be}{\begin{enumerate}}
\newcommand{\ee}{\end{enumerate}}

\newcommand{\bi}{\begin{itemize}}
\newcommand{\ei}{\end{itemize}}

\def\N{\mathbb N}
\def\Z{\mathbb Z}

\def\R{\mathbb R}
\def\C{\mathbb C}

\def\a{\alpha}
\def\b{\beta}

\def\xx{{\boldsymbol x}}
\def\yy{{\boldsymbol y}}
\def\YY{{\boldsymbol Y}}
\def\II{{\boldsymbol I}}
\def\xxi{{\boldsymbol \xi}}
\def\aa{{\boldsymbol \alpha}}
\def\bb{{\boldsymbol \beta}}
\def\gg{{\boldsymbol \gamma}}
\def\ee{{\boldsymbol \varepsilon}}
\def\00{{\boldsymbol 0}}
\def\11{{\boldsymbol 1}}

\def\d{\partial}

\begin{document}

\title[Summability in a monomial for some singularly perturbed PDEs]{Summability in a monomial for some classes of singularly perturbed partial differential equations}

%Autor: YO
\author{Sergio A. Carrillo}
\address{Sergio A. Carrillo: 
	Fakult\"at f\"ur Mathematik, Universit\"at Wien, 
	Oskar-Morgenstern-Platz~1, A-1090 Wien, Austria.
	Escuela de Ciencias Exactas e Ingenier\'{i}a, Universidad Sergio Arboleda, Calle 74, $\#$ 14-14, Bogot\'{a} ,Colombia.}

\email{sergio.carrillo@univie.ac.at, sergio.carrillo@usa.edu.co}

\thanks{Supported by the Austrian FWF-Project P 26735-N25. Partially supported by the Ministerio de Econom\'{\i}a y Competitividad from Spain, under the Project ``\'{A}lgebra y Geometr\'{\i}a en Din\'{a}mica Real y Compleja III" (Ref.: MTM2013-46337-C2-1-P)}

%\thanks{}

%\keywords{Palabras clave}
\date{\today}
%\dedicatory{Dedico este art\'{\i}culo a mi mam\'{a}}

\begin{abstract}The aim of this paper is to continue the study of asymptotic expansions and summability in a monomial in any number of variables, as introduced in \cite{Monomial summ,Sum wrt germs}. In particular, we characterize these expansions in terms of bounded derivatives and we develop Tauberian theorems for the summability processes involved. Furthermore, we develop and apply the Borel--Laplace analysis in this framework to prove the monomial summability of solutions of a specific class of singularly perturbed PDEs.	
\end{abstract}

\maketitle

\section{Introduction}\label{Introduction}

The theory of asymptotic expansions in one complex variable is a well established and widely used branch of Analysis. It provides the adequate setting to treat solutions of analytic problems at singular points and opens naturally a way to study (Borel) summability of formal solutions, for instance, power series or exponential series. It finds its applications in some classes of ordinary and partial differential equations, analytic classification of formal objects and some other classes of functional equations. In this framework, we have at our disposal two main tools to approach such problems: the Ramis--Sibuya theorem and the Borel--Laplace analysis.

There are different notions of asymptotic expansions in several variables available in the literature. We can mention Majima's strong asymptotic expansions which allows to work with each variable independently, although the problem of identifying singular directions for summability persists, see e.g., \cite{Sanz}. In the Gevrey case,  this notion can also be approached through a Borel--Laplace analysis. Let us mention that in \cite{Costin1} this has been used as a regularization process to prove the existence and ``summability" of solutions for classes of nonlinear evolution partial differential equations.

Our goal in this paper is to follow the sketch we present in Section \ref{Asymptotic and summability in one variable} on the theory of asymptotic expansions in one variable to provide analogous results for the monomial case. Monomial asymptotic expansions lies in between the theory of one variable and the one of Majima. They were introduced and developed systematically in \cite{Monomial summ} for the case of two variables, and then extended in \cite{Sum wrt germs} for any number of them. These expansions are useful to treat, for instance, formal solutions of doubly singular equations where the monomial involved helps us to identify the possible singular directions for summability.

The main theoretical results we obtain here are: a characterization of having a monomial asymptotic expansion in terms of bounded derivatives (Theorem \ref{Bounded derivatives}), equivalent methods to establish monomial summability based on integral transformations (Theorem \ref{Monomial summability t Borel Laplace}) and Tauberian theorems comparing such summability processes (Theorem \ref{tauberian general case}). 

For the two-dimensional case the last two issues were treated in \cite{CM2,CM}. The key idea was to weigh the variables adequately to obtain integral transformations, as the ones introduced in \cite{Balser3}, and then to have at hand a direct Borel--Laplace analysis. One of the improvements we obtain here is that some of the weights can be zero. At a first look, it might seem that the zero weighted variables act as regular parameters. The main difference with the parametric case is that the domains where the asymptotic expansions take place also depend on them. In fact, the summability methods involved are incompatible (Theorem \ref{tauberian for one variable and monomials}). 

%by using blow-ups with center of codimension two.

As an application, we establish in Theorem \ref{Main result sum mon for PDE} the $\boldsymbol{x}^{\boldsymbol{\a}}\boldsymbol{\varepsilon}^{\boldsymbol{\a'}}$--$1$--summability of the unique formal power series solution of the partial differential equation \begin{equation}\label{PDEG}
\boldsymbol{x}^{\boldsymbol{\a}}\boldsymbol{\varepsilon}^{\boldsymbol{\a'}}\left( \mu_1 x_1 \frac{\d \boldsymbol{y}}{\d x_1}+\cdots+ \mu_n x_n \frac{\d \boldsymbol{y}}{\d x_n}\right)=\boldsymbol{G}(\boldsymbol{x},\boldsymbol{\varepsilon},\boldsymbol{y}),
\end{equation} where $\boldsymbol{x}=(x_1,\dots,x_n), \boldsymbol{\varepsilon}=(\varepsilon_1,\dots,\varepsilon_m)$  are complex variables, $\boldsymbol{\alpha}=(\alpha_1,\dots,\alpha_n)$,  $\boldsymbol{\a'}=(\a_1',\dots,\a_m')$ are tuples of positive integers, $(\mu_1,\dots,\mu_n)$ is, up to a non-zero multiple scalar, a $n$-tuple of positive real numbers,   $\boldsymbol{G}$ is a $\C^N$--valued holomorphic function in a neighborhood of the origin, and $\frac{\partial \boldsymbol{G}}{\partial \yy}(0,0,\boldsymbol{0})$ is an invertible matrix. In this way, we have generalized the results in \cite{Monomial summ,CM2} corresponding to the case $n=m=1$ by using directly the appropriate Borel--Laplace analysis.

Asymptotic expansions and summability have been recently generalized by J. Mozo and R. Sch\"{a}fke in \cite{Sum wrt germs} from monomials to germs of analytic functions. Integral transformations for the corresponding summability methods are still not available and it is an interesting problem to determine whether it is possible to extend our results to that setting. It is worth mentioning that after our current results, we have recently extended these Tauberian theorems for $k$--summability with respect to analytic germs, see \cite{CMS}.

The plan for the paper is as follows: in Section \ref{Asymptotic and summability in one variable} and \ref{Asymptotic expansions in a monomial} we recall the basic results on asymptotic expansions and summability in one variable and for monomials, respectively. Section \ref{Borel-Laplace analysis for monomial summability} is devoted to introduce and develop integral transformations to characterize monomial summability, and then in Section \ref{Tauberian properties for monomial summability} these tools are applied to prove Tauberian theorems for these summability methods. Finally, Section \ref{Monomial summability of a family of singular perturbed PDEs} contains the proof of the monomial summability of the formal solution of the singularly perturbed partial differential equation mentioned above.

\begin{gracias}I want to thank professors Jorge Mozo--Fern\'{a}ndez, Reinhard Sch\"{a}fke and Armin Rainer for fruitful discussions. I also want to thank Universidad de Valladolid (Spain) for the hospitality during my visits while preparing this article. 
\end{gracias}

\section{Asymptotics and summability in one variable}\label{Asymptotic and summability in one variable}

We start by introducing some notation: Let $\N$ denote the set of natural numbers including $0$ and $\N^+=\N\setminus\{0\}$. Domains in the complex plane $\C$ where holomorphic maps admit an asymptotic expansion are sectors with vertex at some fixed point, e.g., the origin. In this paper we denote them by $S=S(\theta,b-a,r)=V(a,b,r)=\{x\in\C \,|\, 0<|x|<r, a<\text{arg}(x)<b\}$ emphasizing on its \textit{bisecting direction} $\theta=(b+a)/2$, \textit{opening} $b-a>0$, and \textit{radius} $r>0$. For unbounded sectors we simply write $S=S(\theta,b-a)$. For subsectors $S'=S(\theta',b'-a',r')$, $a<a'<b'<b$, $0<r'<r$, we write $S'\Subset S$. We also denote by $D_r$ the disc of radius $r$ centered at the origin.

Let $(E,\|\cdot\|)$ be a complex Banach space. In most applications $E$ is $\C^d$, for some $d\geq1$, or a suitable space of functions. We will use the notation $\mathcal{C}(U, E)$ (resp. $\mathcal{O}(U, E)$, $\mathcal{O}_b(U, E)$) for the space of continuous (resp. holomorphic, holomorphic and bounded) $E$--valued maps defined on an open set $U\subseteq \C^d$. If $E=\C$, we will simply write $\mathcal{O}(U)$. We also denote by $E[[x]]$ (resp. $E\{x\}$) the space of formal (resp. convergent) power series in the variable $x$ with coefficients in $E$.

Consider $f\in\mathcal{O}(S, E)$ and assume it has  $\hat{f}=\sum_{n=0}^{\infty} a_n x^n\in E[[x]]$ as \textit{asymptotic expansion} at the origin on $S$ (denoted by $f\sim \hat{f}$ on $S$), i.e.,  for each $S'\Subset S$   and $N\in\N$, there exists $C_N(S')>0$ such that \begin{equation}\label{def asym classic}
\left\|f(x)-\sum_{n=0}^{N-1} a_n x^n\right\|\leq C_N(S')|x|^N,\quad \text{ on } S'.
\end{equation}

To check that $f\sim\hat{f}$ on $S$, it is actually sufficient to have inequalities (\ref{def asym classic}) only for the values $N=Mp$, where $p\in \N^+$ is fixed. The asymptotic expansion also holds if instead of the partial sums of $\hat{f}$ we consider a sequence $(f_N)_{N\in\N}\subset\mathcal{O}_b(D_R, E)$ satisfying that for each $S'\Subset S$ and $N\in\N$, there are constants $A_N(S')>0$ such that \begin{equation*}\label{def asym classic with sequence of function}\|f(x)-f_N(x)\|\leq A_N(S')|x|^N,\quad\text{ on }S'\cap D_R.\end{equation*}

%and there are constants $B,D$ with $\sup_{|x|\leq R}\|f_n(x)\|\leq DB^n M_n$, for all $n\in\N$.

The series $\hat{f}$ is completely determined by $f$ since $a_n=\lim_{S'\ni x\rightarrow 0} \frac{f^{(n)}(x)}{n!}$, for any $S'\Subset S$. The series $\hat{f}$ is also given by the limit of the Taylor series at the origin of the $f_n$, in the $\mathfrak{m}$--topology of $E[[x]]$, $\mathfrak{m}=(x)$. On the other hand, a map has an asymptotic expansion on $S$ if and only if it has bounded derivatives of all orders on every $S'\Subset S$, a fact that follows by using Taylor's formula.

When no restrictions on the coefficients $C_N(S')$ or on the sector $S$ are imposed, the map $f\mapsto \hat{f}$ is not injective. In applications to differential equations the types of asymptotic expansions that appear are of $s$--Gevrey kind  (denoted by $f\sim_{s} \hat{f}$ on $S$), for some $s>0$. This means that we can choose $C_n(S')=C(S')A(S')^n n!^s$, for some $C(S'), A(S')>0$ independent of $n$.  It follows from (\ref{def asym classic}) that $\|a_n\|\leq C_n(S')$ for all $n\in\N$. Then, in the $s$--Gevrey case, we conclude that $\hat{f}$ is a $s$--\textit{Gevrey series}. The space of such series will be denote by $E[[x]]_s$. The cornerstone to define $k$--summability is Watson's lemma: if $f\sim_s 0$ on $S(\theta,b-a,r)$ and $b-a> s\pi$, then $f\equiv 0$.

Given $\hat{f}\in E[[x]]$, $k>0$ and a direction $\theta\in\R$, we say that: \begin{enumerate} \item The series $\hat{f}$ is \textit{$k$--summable on $S=S(\theta,b-a,r)$ with sum $f\in\mathcal{O}(S,E)$} if $b-a>\pi/k$ and $f\sim_{1/k} \hat{f}$ on $S$. We also say that $\hat{f}$ is \textit{$k$--summable in the direction $\theta$}. The corresponding space of such series is denoted by $E\{x\}_{1/k,\theta}$.
	
\item The series $\hat{f}$ is \textit{$k$--summable} if it is $k$--summable in all directions, up to a finite number of them mod. $2\pi$ (the singular directions). The corresponding space is denoted by  $E\{x\}_{1/k}$.\end{enumerate}
	
Due to Watson's lemma, the $k$--sum of a $k$--summable series is unique. We have at our disposal integral transformations to compute these sums. Among the kernels of order $k$ for moment summability, see e.g., \cite[Sec. 6.5]{Balser2}, it is common to consider:

\begin{enumerate}\item \textit{The $k$--Borel transform}, defined by $\mathcal{B}_kf(\xi)=\frac{k}{2\pi i}\int_{\gamma} f(x)e^{(\xi/x)^k} \frac{dx}{x^{k+1}},$ where $f\in\mathcal{O}_b(S,E)$, $S=S(\theta,\pi/k+2\epsilon, R_0)$, $0<2\epsilon<\pi/k$, and $\gamma$ is the boundary, oriented positively, of a subsector of $S$ of opening larger than $\pi/k$. Its formal counterpart, $\widehat{\mathcal{B}}_k$, acts on monomials by the formula  $\widehat{\mathcal{B}}_k(x^\lambda)(\xi)=\frac{\xi^{\lambda-k}}{\Gamma(\lambda/k)}$, $\lambda\in\C$.
	
\item \textit{The $k$--Laplace transform in direction $\theta$}, defined by $\mathcal{L}_{k,\theta}(g)(x)=\int_0^{e^{i\theta}\infty} g(\xi)e^{-(\xi/x)^k}d\xi^k,$ where $g$ is continuous and has exponential growth of order at most $k$ on the domain of integration. If $g$ is defined on an unbounded sector, we obtain a map $\mathcal{L}_{k}(g)$, through analytic continuation by moving $\theta$.\end{enumerate}

Using these transformations, and due to their compatibility with asymptotic expansions, a $1/k$--Gevrey series $\hat{f}=\sum_{n=0}^\infty a_n x^n$ is called $k$--\textit{Borel summable in direction $\theta$} if $\widehat{\mathcal{B}}_k\left(\hat{f}-\sum_{n\leq k} a_n x_n\right)$ can be analytically continued, say as $\varphi$,  to an unbounded sector $S'$ containing $\theta$, and having exponential growth of order at most $k$ in $S'$, i.e., we can find constants $L,M>0$ such that $\|\varphi(\xi)\|\leq Le^{M|\xi|^k}$, for all $\xi\in S'$. The \textit{$k$--Borel sum} of $\hat{f}$ is defined by $f(x)=\sum_{n\leq k} a_nx^n+\mathcal{L}_{k}(\varphi)(x)$. It is well--known that a power series $\hat{f}\in E[[x]]$ is $k$--Borel summable in a direction $\theta$ if and only if it is $k$--summable in the direction $\theta$, and both sums coincide, see e.g., \cite{Ramis1}. This equivalence is useful also to prove Tauberian theorems for $k$--summability. In particular, we know that if a series is $k$--summable for two different values of the parameter $k$, then it is convergent.

The Borel--Laplace analysis has been applied as a regularization process in differential equations to prove the summability of formal solutions in generic situations. It exploits the isomorphism between the following structures \begin{equation}\label{Iso structres 1 variable}\left(E[[x]]_{1/k},+,\,\times\, ,x^{k+1}d/dx\right)\xrightarrow{\widehat{\mathcal{B}}_k} \left(\xi^{-k}E\{\xi\},+,\,\ast_k\, ,k\xi^k(\cdot)\right),
\end{equation} where $\times$ denotes the usual product and $\ast_k$ stands for the $k$--convolution product. For holomorphic maps, it is given by $(f\ast_k g)(\xi)=\xi^k \int_{0}^{1} f(\xi\tau^{1/k}), g(\xi(1-\tau)^{1/k})d\tau$. For more morphisms of this nature, see e.g., \cite{Mar-R}.

\section{Asymptotic expansions in a monomial}\label{Asymptotic expansions in a monomial}

In this section we recall the concepts of asymptotic expansions and $k$--summability in a monomial, and their main properties. In particular, we prove Theorem \ref{Bounded derivatives} that characterizes maps admitting a monomial asymptotic expansion in terms of bounds on their derivatives.

We introduce the remaining notation we will use along the text. For a fixed $d\in \N^+$, we will write $[1,d]$ for the set $\{1,2,\dots,d\}$,  $\boldsymbol{e}_1,\dots,\boldsymbol{e}_d$ will denote the canonical basis of $\C^d$, $\sigma_d=\{(t_1,\dots,t_d)\in\R_{>0}^{d} \,|\, t_1+\cdots+t_d=1\}$ will be the standard $d$--simplex they generate, and $\overline{\sigma_d}$ will denote its topological closure. We will also write $\left<\boldsymbol{\lambda},\boldsymbol{\mu}\right>=\lambda_1\mu_1+\cdots+\lambda_d\mu_d$, for all $\boldsymbol{\lambda}, \boldsymbol{\mu}\in \C^d$. 

We use complex coordinates $\boldsymbol{x}=(x_1,\dots,x_d)\in\C^d$. If  $\boldsymbol{\beta}=(\beta_1,\dots,\beta_d)\in\N^d$ and $\boldsymbol{s}=(s_1,\dots,s_d)\in \R^d_{\geq0}$, we use the multi-index notation $|\boldsymbol{\beta}|=\beta_1+\cdots+\beta_d$, $\boldsymbol{\beta}!^{\boldsymbol{s}}=\beta_1!^{s_1}\cdots\beta_d!^{s_d}$, $\boldsymbol{x}^{\boldsymbol{\beta}}=x_1^{\beta_1}\cdots x_d^{\beta_d}$, and $\frac{\d^{\boldsymbol{\beta}}}{\d \boldsymbol{x}^{\boldsymbol{\beta}}}={\d^{|\boldsymbol{\beta}|}}/{\d x_1^{\beta_1}\cdots \d x_d^{\beta_d}}$. If $J\subseteq [1,d]$, we denote by $J^c=\{i\in[1,d] \,|\, i\not\in J\}$ its complement, $\# J$ its cardinal, and we write $\N^J=\{(\beta_j)_{j\in J} \,|\, \beta_j\in\N, j\in J\}$, $\xx_J=(x_j)_{j\in J}$, $\bb_J=(\beta_j)_{j\in J}$, and  $\xx_J^{\bb_J}=\prod_{j\in J} x_j^{\beta_j}$. Along the text we work with the partial order on $\N^d$ defined by $\bb\leq \aa$ if and only if $\beta_j\leq \a_j$, for all $j\in [1,d]$. We will also write $\bb<\aa$ if $\beta_j<\a_j$, for all $j\in[1,d]$. Note that $\bb\not\leq \aa$ if and only if there is $j\in [1,d]$ such that $\beta_j>\a_j$.

Given a complex Banach space $(E,\|\cdot\|)$, $E[[\xx]]$ (resp. $E\{\xx\}$) will denote the space of formal (resp. convergent) power series in the variables $\xx$ with coefficients in $E$. If $\boldsymbol{s}\in \R_{\geq0}^d$, we denote by $E[[\boldsymbol{x}]]_{\boldsymbol{s}}$ the space of $\boldsymbol{s}$--Gevrey series in the variable $\xx$, i.e., $\sum_{\bb\in\N^d} a_{\boldsymbol{\beta}} \boldsymbol{x}^{\boldsymbol{\beta}}$ is $\boldsymbol{s}$--Gevrey if there exist constants $C,A>0$ such that $\|a_{\boldsymbol{\beta}}\|\leq CA^{|\boldsymbol{\beta}|}\boldsymbol{\beta}!^{\boldsymbol{s}}$, for all $\boldsymbol{\beta}\in\N^d$.

Given any $\hat{f}=\sum_{\bb\in\N^d} a_\bb \xx^\bb\in E[[\xx]]$, we can write it uniquely for every nonempty subset $J\subsetneq [1,d]$ as \begin{equation}\label{Decomposition f J}
\hat{f}=\sum_{\bb_J\in\N^J} \hat{f}_{J,\bb_J}(\xx_{J^c})\xx_J^{\bb_J},\quad \hat{f}_{J,\bb_J}(\xx_{J^c})=\sum_{\bb_{J^c}\in\N^{J^c}} a_{\bb_{J\cup J^c}} \xx_{J^c}^{\bb_{J^c}}.
\end{equation} Furthermore, if $\boldsymbol{\a}\in(\N^+)^{d}$ is given and we consider the monomial $\boldsymbol{x}^{\mathbf{\a}}$, $\hat{f}$ can be also written uniquely as \begin{equation}\label{para definir Tpq}
\hat{f}=\sum_{n=0}^\infty \hat{f}_{\aa,n}(\boldsymbol{x})\boldsymbol{x}^{n\boldsymbol{\a}},\quad \hat{f}_{\aa,n}=\sum_{\aa\not\leq\bb}a_{n\boldsymbol{\a}+\boldsymbol{\beta}}\boldsymbol{x}^{\boldsymbol{\beta}}.
\end{equation}

To ensure that each $\hat{f}_{\aa,n}=f_{\aa,n}$ gives rise to a holomorphic map, defined in a common polydisc at the origin for all $n\in\N$, it is necessary and sufficient that $\hat{f}\in \hat{\mathcal{O}}'_d(E):=\bigcup_{r>0} \hat{\mathcal{O}}'_d(r,E)$, where $\hat{\mathcal{O}}'_d(r,E):=\bigcap_{j=1}^d \mathcal{O}_b(D_r^{d-1},E)[[x_j]]$. If this is the case,  then $\hat{f}_{J,\bb_J}=f_{J,\bb_J}\in E\{\xx_{J^c}\}$, for all $\bb_J\in \N^{J}$ and $J\subsetneq [1,d]$, and they are defined in a common polydisc at the origin. Besides $f_{\aa,n}\in \mathcal{E}^{\boldsymbol{\a}}:=\bigcup_{r>0}\mathcal{E}^{\boldsymbol{\a}}_r$, where $\mathcal{E}^{\boldsymbol{\a}}_r$ is the space of holomorphic maps $g\in\mathcal{O}_b(D_r^d,E)$ such that $\frac{\d^{\bb}g}{\d x^{\bb}}(\boldsymbol{0})=0$, for all $\aa\not\leq\bb$. Also each $\mathcal{E}^{\boldsymbol{\a}}_r$ becomes a Banach space with the norm $\|g\|_r:=\sup_{|x_1|,\dots,|x_d|\leq r} \|g(\boldsymbol{x})\|$.

For any $\hat{f}\in\hat{\mathcal{O}}'_d(E)$ and $\gg\in\N^d$, we will use the notation $$\text{App}_\gg (\hat{f})(\xx)=\sum_{\boldsymbol{
	\gamma}\not\leq \bb}a_{\bb}\boldsymbol{x}^{\bb}\in E\{\xx\},$$ for the \textit{formal approximate of $\hat{f}$ of order $\boldsymbol{\gamma}$}. In particular,  if $\gg=N\aa$, we have \begin{equation}\label{App Na}
\text{App}_{N\aa} (\hat{f})(\xx)=\sum_{n=0}^{N-1} f_{\aa,n}(\xx) \xx^{n\aa}.
\end{equation}

%$$\text{App}_\gg (\hat{f})(\xx)=\sum_{\emptyset\neq J\subseteq[1,d]}(-1)^{\# J+1}\sum_{{\bb_J\in \N^J}\atop {\bb_J\leq \gg_{J}-1}} f_{J,\bb_J}(\xx_{J^c})\xx_J^{\bb_J}=\sum_{{\beta_j<\gamma_j}\atop\\ {\text{for some }}j}a_{\bb}\boldsymbol{x}^{\bb}\in E\{\xx\}.$$

\begin{nota}\label{Bounds factorial min max}For further use, we remark the following bounds on the factorial. First, it is elementary to show that \begin{equation}\label{Eq bounds factorial} n!^k\leq (kn)!\leq k^{kn} n!^k,\quad n,k\in\N^+.	
	\end{equation} Let us fix $\aa\in (\N^+)^d$ and consider $\gg\in\N^d\setminus\{\00\}$. If $N=\min_{1\leq j\leq d}\lfloor \gamma_j/\a_j\rfloor=\lfloor \gamma_l/\a_l\rfloor$, where $\lfloor\cdot\rfloor$ denotes the floor function, then $\gamma_l/\a_l\leq N \leq \gamma_j/\a_j$, for all $j=1,\dots,d$. Then, using (\ref{Eq bounds factorial}), we see that $N!^{\a_j}\leq (\a_j N)!\leq \gamma_j!$, and 
	$\gamma_l!\leq (\a_lN)!\leq \a_l^{\a_l N} N!^{\a_l}$. Since $N\leq |\gg|$, we can conclude that 
	$$|\aa|^{-|\gg|} \min_{1\leq j\leq d} \gamma_j!^{1/\a_j}\leq N!\leq \min_{1\leq j\leq d} \gamma_j!^{1/\a_j}.$$ Analogously, if we consider $N=\max_{1\leq j\leq d}\lfloor\gamma_j/\a_j\rfloor+1=\lfloor\gamma_m/\a_m\rfloor+1$ instead, now we have $\gamma_j/\a_j<N\leq \gamma_m/\a_m+1$, for all $j=1,\dots,d$. Then $N!^{\a_m}\leq (\a_m N)!\leq (\gamma_m+\a_m)!\leq 2^{\gamma_m+\a_m} \gamma_m!\a_m!$. Using that $N\leq 2|\gg|$, and $\a_m!^{1/\a_m}\leq \a_m\leq |\aa|$, we can conclude as before that $$|\aa|^{-2|\gg|} \max_{1\leq j\leq d} \gamma_j!^{1/\a_j}\leq N!\leq |\aa|2^{2|\gg|} \max_{1\leq j\leq d} \gamma_j!^{1/\a_j}.$$

%it follows from Stirling's formula that $$\lim_{n\rightarrow\infty} \frac{(nk)!^{1/k}}{k^n n!}n^{\frac{1}{2}-\frac{1}{2k}}=\frac{(2\pi k)^{\frac{1}{2k}}}{\sqrt{2\pi}},\quad \text{ for all integer } k\geq 1.$$ 	

\end{nota}

For each $\aa\in(\N^+)^d$, we consider the map $\hat{T}_{\aa}:\hat{\mathcal{O}}'_d(E)\rightarrow \mathcal{E}^{\boldsymbol{\a}}[[t]]$ given by $\hat{T}_{\aa}(\hat{f})=\sum_{n=0}^\infty f_{\aa,n} t^n$, by using decomposition (\ref{para definir Tpq}). We will say $\hat{f}$ is a \textit{$s$--Gevrey series in the monomial $\boldsymbol{x}^{\boldsymbol{\a}}$} if for some $r>0$, $\hat{T}_{\aa}(\hat{f})\in\mathcal{E}_r^{\boldsymbol{\a}}[[t]]$, and it is a $s$--Gevrey series in $t$, i.e., there are constants $B,D>0$ such that $\|f_{\aa,n}\|_r\leq DB^n n!^s$, for all $n\in\N$. The space of $s$--Gevrey series in $\boldsymbol{x}^{\boldsymbol{\a}}$ will be denoted by $E[[\boldsymbol{x}]]_{s}^{\boldsymbol{\a}}$. Their elements admit another characterization, for which we need

\begin{lema}\label{Bounds for formal gevrey series}The following assertions are verified for a series $\hat{f}=\sum a_{\boldsymbol{\beta}}\boldsymbol{x}^{\boldsymbol{\beta}}\in E[[\xx]]$:\begin{enumerate}
\item  $\hat{f}\in E[[\boldsymbol{x}]]_{s}^{\boldsymbol{\a}}$ if and only if there are constants $C,A>0$ satisfying $$\|a_{\boldsymbol{\beta}}\|\leq CA^{|\boldsymbol{\beta}|}\min\left\{\beta_1!^{s/\a_1},\dots,\beta_d!^{s/\a_d}\right\},\quad \boldsymbol{\beta}\in\N^d.$$ %If this is the case then $\|f_{J,\bb_J}(\xx_{J^c})\|\leq \frac{CA^{|\bb_J|}}{(1-rA)^{\# J^c}}\min_{j\in J} M_{\beta_j}^{1/\a_j}$, for all $J\subseteq [1,d]$.

\item If $\hat{f}\in E[[\boldsymbol{x}]]_{s}^{\boldsymbol{\a}'}$, then $\hat{T}_{\aa}(\hat{f})$ is a $\max_{1\leq j\leq d}\{\a_j/\a_j'\}s$--Gevrey series, in some $\mathcal{E}^{\boldsymbol{\a}}_r$.
\end{enumerate}
\end{lema}

\begin{proof}(1) Assume there are constants $B,D>0$ such that $\|f_{\aa,n}\|_r\leq DB^nn!^s$, for all $n\in\N$. Given $\boldsymbol{\gamma}\in\N^d$, let $N=\min_{1\leq j\leq d}\lfloor \gamma_j/\a_j\rfloor$. Thus $\gg=N\aa+\bb$ with $\beta_l<\a_l$, for some $l$. Then, by Cauchy's inequalities, we see that $$\|a_{\gg}\|=\|a_{N\aa+\bb}\|=\left\|\frac{1}{\bb!}\frac{\d^{\bb}f_{\aa,N}}{\d \xx^{\bb}}(\boldsymbol{0})\right\|\leq\frac{DB^N}{r^{|\boldsymbol{\beta}|}}N!^s,$$ what yields one implication, with the aid of Remark \ref{Bounds factorial min max}. The converse follows by the same argument as in (2) below.

(2) If $\|a_{\boldsymbol{\beta}}\|\leq CA^{|\boldsymbol{\beta}|}\min_{1\leq j\leq d}\{\beta_j!^{s/\a_j'}\}$, for all $\boldsymbol{\beta}\in\N^d$, we can directly estimate the growth of the $f_{\aa,n}$ by means of equation (\ref{para definir Tpq}): if $|\xx|<r$,  and $rA<1$ we obtain\begin{align*}
\|f_{\aa,n}(\boldsymbol{x})\|=\left\|\sum_{\aa\not\leq\gg} a_{n\aa+\gg} \xx^\gg\right\|&\leq \sum_{j=1}^{d}\sum_{\beta_j=0}^{\a_j-1} \sum_{\gg\in \N^d, \gamma_j=\b_j} CA^{n|\aa|+|\gg|} r^{|\gg|} \min_{1\leq l\leq d}(n\a_l+\gamma_l)!^{s/\a_l'}\\
&\leq \frac{CA^{n|\boldsymbol{\a}|}}{(1-rA)^{d-1}}\sum_{j=1}^{d}\sum_{\beta_j=0}^{\a_j-1}   (n\a_j+\beta_j)!^{s/\a_j'}(rA)^{\beta_j}.
\end{align*} If we write $s'=\max_{1\leq j\leq d} \{\a_j /\a_j'\} s$, by using inequalities (\ref{Eq bounds factorial}), we find that $$(n\a_j+\beta_j)!^{s/\a_j'}\leq (\a_j(n+1))!^{s/\a_j'}\leq \a_j^{\a_j (n+1)}(n+1)!^{s \a_j /\a_j'}\leq \a_j^{\a_j (n+1)} (n+1)!^{s'},$$ for all $n\in\N$. Then it is clear that we can find constants $K,M>0$ such that $\|f_{\aa,n}(\boldsymbol{x})\|\leq KM^n n!^{s'}$, for all $|\xx|<r$ and all $n\in\N$, as we wanted to show.

\end{proof}

The previous lemma implies that $E[[\xx]]_{Ms}^{M\aa}=E[[\xx]]_s^{\aa}$, for all $M\in\N^+$. It also shows that $\hat{f}\in E\{\boldsymbol{x}\}$ if and only if $\hat{T}_{\aa}(\hat{f})\in \mathcal{E}^{\boldsymbol{\a}}_r\{t\}$, for some $r>0$, by taking $s=0$. Moreover, Lemma \ref{Bounds for formal gevrey series} (1) shows that \begin{equation}\label{segment of line}
E[[\boldsymbol{x}]]_{s}^{\boldsymbol{\a}}=\bigcap_{j=1}^d E[[\boldsymbol{x}]]_{\frac{s}{\a_j}\boldsymbol{e}_j}\subseteq E[[\boldsymbol{x}]]_{\boldsymbol{s}},
\end{equation} for any $\boldsymbol{s}$ in the convex hull of $\left\{s/\a_1\boldsymbol{e}_1,\dots,s/\a_d\boldsymbol{e}_d\right\}$. This inclusion follows from the first inequality in \begin{equation}\label{prop the segment}\min\{a_1,\dots,a_d\}\leq a_1^{t_1}\cdots a_d^{t_d} \leq \max\{a_1,\dots,a_d\},\end{equation} valid for any $a_1,\dots,a_d>0$ and $(t_1,\dots,t_d)\in \overline{\sigma_d}$.

In the analytic setting, we use \textit{sectors in the monomial $\boldsymbol{x}^{\boldsymbol{\a}}$}, i.e., sets of the form $$\Pi_{\boldsymbol{\a}}=\Pi_{\boldsymbol{\a}}(a,b,r)=\left\{\boldsymbol{x}\in \C^d \,|\, a<\text{arg}(\boldsymbol{x}^{\boldsymbol{\a}})<b,\, 0<|x_j|^{\a_j}<r,\, j\in[1,d]\right\}.$$ Here any convenient branch of arg may be used. The number $r>0$ denotes the \textit{radius}, $b-a>0$ the \textit{opening} and $\theta=(b+a)/2$ the \textit{bisecting direction} of the monomial sector. We will also use the notation $\Pi_{\boldsymbol{\a}}(a,b,r)=S_{\boldsymbol{\a}}(\theta,b-a,r)=S_{\boldsymbol{\a}}$. In the case $r=+\infty$ we will simply write $\Pi_{\boldsymbol{\a}}(a,b)=S_{\boldsymbol{\a}}(\theta,b-a)$, and we will refer to it as an  \textit{unbounded sector}. The definition of subsector in a monomial is clear.

\begin{nota}\label{Polysectors} Given two monomial sectors $\Pi_\aa'=\Pi_\aa(a',b',r)\subset \Pi_\aa''=\Pi_\aa(a'',b'',r)$, we can always cover the first one by polysectors, i.e., Cartesian product of sectors, of constant opening contained in the second one. In particular, we can check that $\Pi_\aa'\subset U \subset \Pi_\aa''$, where $U$ is given by
\begin{align*}U=\bigcup_{\mu_1,\dots,\mu_{d-1}\in\R} \prod_{j=1}^{d-1}& V\left(\mu_j,\mu_j+\frac{b'-a'}{(d-1)\a_j},r^{1/\a_j}\right)\\
&\times V\left(\frac{a''-\sum_{k=1}^{d-1} \a_{k}\mu_{k}}{\a_d},\frac{b''-(b'-a')-\sum_{k=1}^{d-1} \a_{k}\mu_{k}}{\a_d},r^{1/\a_d}\right).\end{align*}

Indeed, if $\xx_0=(x_{1,0},\dots, x_{d,0})\in \Pi_\aa'$,  then $\xx_0$ belongs, for instance, to the polysector with $\mu_j$ given by
$$\mu_j=\text{arg}(x_{j,0})-\frac{\phi}{\a_j},\quad \max\left\{0,\frac{\text{arg}(\xx_0^\aa)-b''+b'-a'}{d-1}\right\}<\phi<\min\left\{\frac{b'-a'}{d-1},\frac{\text{arg}(\xx_0^\aa)-a''}{d-1}\right\}.$$
\end{nota}

\

If $f\in\mathcal{O}_b(\Pi_{\aa},E)$, $\Pi_{\aa}=\Pi_{\aa}(a,b,r)$, then we can construct an operator $T_{\boldsymbol{\a}}(f)_\rho:V\rightarrow \mathcal{E}^{\boldsymbol{\a}}_{\rho}$, $V=V(a,b,\rho^d)$ and $0<\rho<r$, as it is done in the formal case, such that \begin{equation*}\label{Property T_a}T_{\boldsymbol{\a}}(f)_\rho(\boldsymbol{x}^{\boldsymbol{\a}})(\boldsymbol{x})=f(\boldsymbol{x}).\end{equation*} We recall this construction by following \cite{Sum wrt germs}. We start with the case $\boldsymbol{\a}=\11:=(1,\dots,1)$. Define the map $g(t,x_2,\dots,x_d):=f\left(\frac{t}{x_2\cdots x_d},x_2,\dots,x_d\right)$ for $|x_2|,\dots,|x_d|<r$ and $|t|/r<|x_2\cdots x_d|$. The map $g$ admits a Laurent series expansion in this domain,  $g(t,x_2,\dots,x_d)=\sum_{\boldsymbol{m}\in\Z^{d-1}} g_{\boldsymbol{m}}(t)\boldsymbol{x}'^{\boldsymbol{m}}$, where  $g_{\boldsymbol{m}}\in\mathcal{O}(V,E)$.

If $\boldsymbol{m}\in\Z^{d-1}$, we use the notation $\mu(\boldsymbol{m})=\min\{0,m_2,\dots,m_d\}\leq0$, and $\phi:\Z^{d-1}\rightarrow \mathcal{M}_d$ for the bijection  $\phi(\boldsymbol{m})=(-\mu(\boldsymbol{m}),m_2-\mu(\boldsymbol{m}),\dots,m_d-\mu(\boldsymbol{m}))$, where $\mathcal{M}_d\subset \N^d$ is the set of all $(n_1,\dots,n_d)\in\N^d$ such that at least one of the $n_j$ vanishes. Then, by definition \begin{equation*}\label{Def T_1,...,1}	 T_{\11}(f)_\rho(t)(\boldsymbol{x}):=\sum_{\boldsymbol{m}\in\Z^d} t^{\mu(\boldsymbol{m})}g_{\boldsymbol{m}}(t) \boldsymbol{x}^{\phi(\boldsymbol{m})}.
\end{equation*} With these considerations, we guarantee that all the exponents in $\xx$ are non-negative. To check that this expression is well-defined and satisfies what it is required, note that since $f$ is bounded, say by some constant $C$, Cauchy's inequalities yield  $\|g_{\boldsymbol{m}}(t)\|\leq C r_2^{-m_2}\cdots r_d^{-m_d}$. If $m_l=\mu(\boldsymbol{m})$, choose $r_j=r$ for all $j\neq l$, and $r_l$ such that $r_2\cdots r_d=|t|/r$, to deduce that $$\|g_{\boldsymbol{m}}(t)\|\leq C |t|^{-\mu(\boldsymbol{m})}r^{d\mu(\boldsymbol{m})-(m_2+\cdots +m_d)}.$$ Thus each
$ t^{\mu(\boldsymbol{m})}g_{\boldsymbol{m}}(t)$ is holomorphic and bounded on $V$. It is also clear that $\boldsymbol{x}^{\phi(\boldsymbol{m})}\in \mathcal{E}^{\11}_{\rho}$, since $\phi(\boldsymbol{m})\in \mathcal{M}_d$, and then the map defined through the previous series also belongs to the same space.

More generally, if there is a function $K:(0,r^d)\rightarrow \R$ such that $\|f(\boldsymbol{x})\|\leq K(|x_1\cdots x_d|)$, $\boldsymbol{x}\in \Pi_{\11}$, then $\|t^{\mu(\boldsymbol{m})}g_{\boldsymbol{m}}(t)\|\leq K(|t|)r^{d\mu(\boldsymbol{m})-(m_2+\cdots +m_d)}$. Thus $T_{\11}(f)_\rho(t)(\boldsymbol{x})$ is bounded by
$$K(|t|) \sum_{\boldsymbol{m}\in \Z^{d-1}} \left(\frac{1}{r}\boldsymbol{x}\right)^{\phi(\boldsymbol{m})}=K(|t|)\sum_{\boldsymbol{n}\in\mathcal{M}_d} \left(\frac{1}{r}\boldsymbol{x}\right)^{\boldsymbol{n}}\leq  \frac{K(|t|)}{\prod_{j=1}^{d}\left(1-\frac{|x_j|}{r}\right)}.$$ We conclude that \begin{equation}\label{Bound for T_1,...,1}
\|T_{\11}(f)_\rho(t)(\boldsymbol{x})\|\leq \frac{K(|t|)}{\prod_{j=1}^{d}\left(1-\frac{|x_j|}{r}\right)},\quad t\in V(a,b,\rho^d),\, \xx\in \Pi_{\11}(a,b,r).
\end{equation}

In the general case, for an arbitrary monomial $\xx^{\aa}$, we can write \begin{equation}\label{Decomposition for ramifications}
f(\xx)=\sum_{\00\leq \bb<\aa} \xx^{\bb} f_{\bb}(x_1^{\a_1},\dots,x_d^{\a_d}),
\end{equation} with $f_{\boldsymbol{\beta}}\in\mathcal{O}(\Pi_{\11}(a,b,r),E)$. In fact, if for each $j$, $\omega_j$ is a $\a_j$-th primitive root of unity, then \begin{equation}
\label{Equation for f_beta} \boldsymbol{x}^{\boldsymbol{\beta}}f_{\boldsymbol{\beta}}(x_1^{\a_1},\dots,x_d^{\a_d})=\frac{1}{\a_1\cdots\a_d}\sum_{\00\leq \boldsymbol{\delta}<\aa} \omega_1^{-\delta_1\beta_1}\cdots\omega_d^{-\delta_d\beta_d}f(\omega_1^{\delta_1}x_1,\dots,\omega_d^{\delta_d}x_d).
\end{equation} Then, we define \begin{equation}\label{Definition T_a}
T_{\boldsymbol{\a}}(f)_\rho(t)(\boldsymbol{x}):=\sum_{\00\leq \bb<\aa} \boldsymbol{x}^{\boldsymbol{\beta}}T_{\11}(f_{\boldsymbol{\beta}})_\rho(t)(x_1^{\a_1},\dots,x_d^{\a_d}).
\end{equation}

To see that $T_{\boldsymbol{\a}}(f)_\rho$ is well-defined and thus holomorphic, we can actually show that if $f$ satisfies $\|f(\boldsymbol{x})\|\leq K(|\boldsymbol{x}^{\boldsymbol{\a}}|)$ on $\Pi_{\boldsymbol{\a}}$, for some function $K:(0,r^d)\rightarrow\R$, then \begin{equation}\label{inq Tpq}
\left\|T_{\boldsymbol{\a}}(f)_\rho(t)\right(\xx)\|\leq R_\aa(|x_1|^{\a_1},\dots,|x_d|^{\a_d},r)\frac{K(|t|)}{|t|}, \quad t\in V(a,b,\rho^d),\, \xx\in \Pi_\aa(a,b,r).
\end{equation} where $$R_\aa(\rho_1,\dots,\rho_d,r):=\frac{r^d}{\prod_{j=1}^d\left(1-\left(\frac{\rho_j}{r}\right)^{1/\a_j}\right)}, \quad 0\leq \rho_1,\dots,\rho_d<r.$$

Indeed, by using equation (\ref{Equation for f_beta}) we get $\|\boldsymbol{x}^{\boldsymbol{\a}}f_{\boldsymbol{\beta}}(x_1^{\a_1},\dots,x_d^{\a_d})\|\leq |x_1|^{\a_1-\beta_1}\cdots|x_d|^{\a_d-\beta_d} K(|\boldsymbol{x}^{\a}|)$. In the variables $u_j=x_j^{\a_j}$ these inequalities take the form \begin{equation*}\label{Inq for f_beta} \|f_{\bb}(\boldsymbol{u})\|\leq r^{1-\beta_1/\a_1}\cdots r^{1-\beta_d/\a_d}\frac{K(|u_1\cdots u_d|)}{|u_1\cdots u_d|}.
\end{equation*} We can thus apply inequality (\ref{Bound for T_1,...,1}) to each summand in (\ref{Definition T_a}) to finally obtain $$\left\|T_{\boldsymbol{\a}}(f)_\rho(t)(\boldsymbol{x})\right\|\leq \sum_{\00\leq \bb<\aa} \left(\frac{|x_1|}{r^{1/\a_1}}\right)^{\beta_1}\cdots \left(\frac{|x_d|}{r^{1/\a_d}}\right)^{\beta_d} \frac{r^dK(|t|)}{\prod_{j=1}^{d}\left(1-\frac{|x_j|^{\a_j}}{r}\right)|t|}.$$ Then (\ref{inq Tpq}) follows by noticing that $|x_j|^{\a_j}<\rho<r$, and also from the identity $\sum_{\00\leq\bb<\aa} a_1^{\beta_1}\cdots a_d^{\beta_d}=\frac{(1-a_1^{\a_1})\cdots (1-a_d^{\a_d})}{(1-a_1)\cdots(1-a_d)}$.

In the case $d=2$ and $\aa=\11=(1,1)$, if $f\in\mathcal{O}(\Pi_{\11}(a,b,r))$, the previous construction consists in writing $f\left(t/x_2,x_2\right)=\sum_{m\in\Z} f_m(t)x_2^m$, as a Laurent series on $|t|/r<|x_2|<r$, and setting  $$T_{\11}(f)_\rho(t)(x_1,x_2)=\sum_{m=0}^\infty \frac{f_{-m}(t)}{t^m}x_1^m+\sum_{m=1}^{\infty} f_m(t)x_2^m.$$

As an application of this construction, we can prove the following proposition on the dependence and growing of a holomorphic map in a monomial. %*Poner articulo de Jorge sobre crecimiento exponencial*

\begin{prop}\label{Dependence of only a monomial} Let $d\geq2$ and $f\in\mathcal{O}(\Pi_{\boldsymbol{\a}},E)$ be holomorphic, where $\Pi_{\boldsymbol{\a}}=\Pi_{\boldsymbol{\a}}(a,b)$. If $\|f(\boldsymbol{x})\|\leq K(|\boldsymbol{x}^{\boldsymbol{\a}}|)$ on $\Pi_{\boldsymbol{\a}}$, for some function $K:(0,+\infty)\rightarrow (0,+\infty)$, then there is $g\in\mathcal{O}(V,E)$, $V=V(a,b)$, such that $f(\boldsymbol{x})=g(\boldsymbol{x}^{\boldsymbol{\a}})$.
\end{prop}

\begin{proof}We proceed by induction on $d$. The case $d=2$ is proved in \cite{CM2}, but we repeat the proof here for the sake of completeness. First, note that we can assume $\text{g.c.d.}(\a_1,\a_2)=1$ by changing $K$ adequately. Given $\aa=(\a_1,\a_2)$ and $f\in\mathcal{O}(\Pi_{\aa},E)$, we write $f$ as in equation (\ref{Decomposition for ramifications}), and if $f_{\beta_1,\beta_2}\left(t/u_2,u_2\right)=\sum_{m\in\Z} f_{\beta_1,\beta_2,m}(t)u_2^m$ for $|t|/r<|u_2|<r$, equation (\ref{Equation for f_beta}) shows that $\|f_{\beta_1,\beta_2}(u_1,u_2)\|\leq |u_1|^{-\beta_1/\a_1}|u_2|^{-\beta_2/\a_2}K(|u_1u_2|)$, $(u_1,u_2)\in \Pi_\11$. Using Cauchy's formulas we obtain $$\|f_{\beta_1,\beta_2,m}(t)\|\leq \frac{|t|^{-\beta_1/\a_1}K(|t|)}{r^{m+\beta_2/\a_2-\beta_1/\a_1}}, \|f_{\beta_1,\beta_2,-m}(t)\|\leq \frac{|t|^{m-\beta_2/\a_2}K(|t|)}{r^{m+\beta_1/\a_1-\beta_2/\a_2}},\quad m\in\N, 0\leq \beta_j<\a_j, j=1,2.$$

If $m\geq 1$, the exponents of $r$ in the previous inequalities are positive. Since $r$ was arbitrary we can take $r\rightarrow+\infty$ and conclude that $f_{\beta_1,\beta_2.m}\equiv 0$. If $m=0$ and $(\beta_1,\beta_2)\neq(0,0)$, by considering one of the preceding inequalities, according to $\beta_2/\a_2>\beta_1/\a_1$ or $\beta_2/\a_2<\beta_1/\a_1$, the same conclusion follows. Thus $f(x_1,x_2)=f_{0,0,0}(x_1^{\a_1}x_2^{\a_2})$.

Now assume that the result is valid for some $d$ and let us prove it for $d+1$. To simplify notations we name our coordinates $(\xx,z)\in\C^d\times\C$ and $(\aa,p)\in(\N^+)^d\times \N^+$. If we decompose $f$ as $$f(\xx,z)=\sum_{j=0}^{p-1} z^j f_j(\xx,z^p),\quad f_j\in\mathcal{O}(\Pi_{\aa,1},E),$$ the bounds for $f$ show that $\|f_j(\xx,\zeta)\|\leq |\zeta|^{-j/p} K(|\xx^\aa \zeta|)$, $j=0,\dots,p-1$. For a fixed $\zeta\in\C^*$ let us write  $\Pi_\aa^\zeta=\{\xx\in(\C^*)^d \,|\, a<\text{arg}(\xx^\aa)+\text{arg}(\zeta)<b\}$ and $V^\zeta=\{\xi\in\C^* \,|\, a<\text{arg}(\xi)+\text{arg}(\zeta)<b\}$. Applying the induction hypothesis to each $f_j(\cdot,\zeta)\in\mathcal{O}(\Pi_\aa^\zeta,E)$, we can conclude that there are maps $g_j(\cdot,\zeta)\in \mathcal{O}(V^\zeta,E)$ such that $f_j(\xx,\zeta)=g_j(\xx^\aa,\zeta)$.

We can now define $g_j$ on $\Pi_{(1,1)}=\Pi_{(1,1)}(a,b)$ in such a way that $g_j\in \mathcal{O}(\Pi_{(1,1)},E)$. Indeed, if $(\xi,\zeta)\in \Pi_{(1,1)}$, then $\xi\in V^\zeta$ and $g_j(\xi,\zeta)$ is already defined. To show that $g_j$ is holomorphic, by using Hartog's theorem, see e.g.,  \cite[p. 28]{Shabat}, it is sufficient to show that it is holomporphic at any point $(\xi_0,\zeta_0)\in\Pi_{(1,1)}$ w.r.t. each of the variables. It only remains to prove this for the second one: choosing $\xx_0\in \Pi_\aa^{\zeta_0}$ such that $\xx_0^\aa=\xi_0$, we know that $g_j(\xi_0,\zeta_0)=f_j(\xx_0^\aa,\zeta_0)$ that depends holomorphically on the second variable. The functions $g_j$ satisfy $\|g_j(\xi,\zeta)\|\leq |\zeta|^{-j/p} K(|\xi\zeta|)$ for $(\xi,\zeta)\in \Pi_{(1,1)}$. The same argument used in the case $d=2$ shows that $g_j\equiv 0$ for $j\neq0$, and $g_0(\xi,\zeta)=g_{0,0}(\xi\zeta)$ for some $g_{0,0}\in\mathcal{O}(V,E)$. In conclusion $f(\xx,z)=g_{0,0}(\xx^\aa z^p)$ as we wanted to show. The induction principle allows us to conclude the proof.
\end{proof}

\begin{defi}\label{Def asym exp monomial}Let $f\in\mathcal{O}(\Pi_{\boldsymbol{\a}},E)$, $\Pi_{\boldsymbol{\a}}=\Pi_{\boldsymbol{\a}}(a,b,r)$ and $\hat{f}\in \hat{\mathcal{O}}'_d(E)$ be given. We will say that \textit{$f$ has $\hat{f}$ as asymptotic expansion at the origin in $\boldsymbol{x}^{\boldsymbol{\a}}$} (denoted by $f\sim^{\boldsymbol{\a}} \hat{f}$ on $\Pi_{\boldsymbol{\a}}$) if there is $0<r'\leq r$ such that $\hat{T}_{\aa}(\hat{f})=\sum f_{\aa,n}t^n\in \mathcal{E}^{\boldsymbol{\a}}_{r'}[[t]]$, and for every proper subsector  $\Pi_{\boldsymbol{\a}}'=\Pi_{\boldsymbol{\a}}(a',b',\rho)$, $0<\rho<r'$, and $N\in\N$, there exists $C_N(\Pi_{\boldsymbol{\a}}')>0$ such that \begin{equation}\label{formula def asym x^pe^q}
\left\|f(\boldsymbol{x})-\sum_{n=0}^{N-1}f_{\aa,n}(\boldsymbol{x})\boldsymbol{x}^{n\boldsymbol{\a}} \right\|\leq C_N(\Pi_{\boldsymbol{\a}}')|\boldsymbol{x}^{N\boldsymbol{\a}}|,\quad \text{ on } \Pi_{\boldsymbol{\a}}'.
\end{equation}
	
The asymptotic expansion is said to be of \textit{$s$--Gevrey type} (denoted by $f\sim^{\boldsymbol{\a}}_{s} \hat{f}$ on $\Pi_{\boldsymbol{\a}}$) if it is possible to choose $C_n(\Pi_{\boldsymbol{\a}}')=C(\Pi_{\boldsymbol{\a}}')A(\Pi_{\boldsymbol{\a}}')^n n!^s$, for some  $C(\Pi_{\boldsymbol{\a}}')$, $A(\Pi_{\boldsymbol{\a}}')>0$ independent of $n$.
\end{defi}

From the very definition of $f\sim^{\boldsymbol{\a}}\hat{f}=\sum a_{\boldsymbol{\beta}} \boldsymbol{x}^{\boldsymbol{\beta}}$ on $\Pi_{\boldsymbol{\a}}$ we can deduce, by using (\ref{formula def asym x^pe^q}) for $N=1$, that \begin{equation}\label{limits of x-0}
a_{\boldsymbol{0}}=\lim_{\Pi'_{\boldsymbol{\a}}\ni \boldsymbol{x}\rightarrow\boldsymbol{0}} f(\boldsymbol{x}),\quad f_{J,\00_J}(\xx_{J^c})=\lim_{{\xx_J\rightarrow0}\atop {\boldsymbol{x}\in \Pi'_{\boldsymbol{\a}}}} f(\boldsymbol{x}),\quad J\subsetneq [1,d],\quad \Pi'_{\boldsymbol{\a}}\subset \Pi_{\boldsymbol{\a}}.
\end{equation} 

Monomial asymptotic expansions can be reduced to the case of one variable by using the operators $T_\aa$ and $\hat{T}_\aa$. Indeed, direct estimates using (\ref{inq Tpq}) show that if $f\in\mathcal{O}(\Pi_{\aa}(a,b,r),E)$, and $\hat{f}\in \hat{\mathcal{O}}'_d(r',E)$, $r'\leq r$, then $f\sim^\aa\hat{f}$ on $\Pi_{\aa}(a,b,r)$ if and only if for every $0<\rho<r'$, $T_\aa(f)_\rho\sim \hat{T}_\aa(\hat{f})$ on $V(a,b,\rho^d)$. The same statement is valid in the Gevrey case, and in this case it follows that $\hat{f}\in E[[\boldsymbol{x}]]^{\boldsymbol{\a}}_{s}$, see \cite[Prop. 3.11]{Sum wrt germs} for details.

Another characterization of monomial asymptotic expansions is obtained by approximating by holomorphic functions.

\begin{prop}\label{equiv f iff f_n for x^pe^q}Let $f\in\mathcal{O}(\Pi_{\boldsymbol{\a}},E)$, $\Pi_{\boldsymbol{\a}}=\Pi_{\boldsymbol{\a}}(a,b,r)$ and $\hat{f}\in \hat{\mathcal{O}}'_d(r',E)$, $r'\leq r$, be given. The following assertions are equivalent:
\begin{enumerate}
\item $f\sim^{\boldsymbol{\a}} \hat{f}$ on $\Pi_{\boldsymbol{\a}}$,

\item There is $R>0$ and a sequence $(F_N)_{N\in\N}\subset \mathcal{O}_b(D_R^d,E)$, $F_0=0$, such that for each  subsector $\Pi_{\boldsymbol{\a}}'$ of $\Pi_{\boldsymbol{\a}}$ and $N\in\N$, there is $A_N(\Pi_{\boldsymbol{\a}}')>0$ such that \begin{equation}\label{f-f_N}\|f(\boldsymbol{x})-F_N(\boldsymbol{x})\|\leq A_N(\Pi_{\boldsymbol{\a}}')|\boldsymbol{x}^{N\boldsymbol{\a}}|,\quad \text{ on }\Pi_{\boldsymbol{\a}}'\cap D_R^d.\end{equation}
\end{enumerate}
If $s>0$, $f\sim^{\boldsymbol{\a}}_s \hat{f}$ on $\Pi_{\boldsymbol{\a}}$ if and only if inequality (\ref{f-f_N}) is satisfied with $A_n(\Pi_{\boldsymbol{\a}}')=C(\Pi_{\boldsymbol{\a}}')A(\Pi_{\boldsymbol{\a}}')^n n!^s$ for some $C(\Pi_{\boldsymbol{\a}}'), A(\Pi_{\boldsymbol{\a}}')>0$ independent of $n$, and there are $B,D>0$ such that $\|F_n\|_R\leq DB^n n!^s$, for all $n\in\N$. In any case, $\hat{f}$ is given by the limit of the Taylor series at the origin of the $F_n$, in the $\mathfrak{m}$--topology of $E[[\boldsymbol{x}]]$, $\mathfrak{m}=(\xx)$.
\end{prop}

\begin{proof}If $f\sim^{\boldsymbol{\a}} \hat{f}$ on $\Pi_{\boldsymbol{\a}}$ and $\hat{T}_{\aa}(\hat{f})(t)=\sum f_{\aa,n} t^n$, then $F_N(\boldsymbol{x})=\sum_{n<N} f_{\aa,n}(\boldsymbol{x})\boldsymbol{x}^{n\boldsymbol{\a}}$ satisfies the requirements. Conversely, suppose we have such a sequence $(F_N)_{N\in\N}$. Note that each $T_{\boldsymbol{\a}}(F_N)_R$ is holomorphic on $D_R$ and has $\hat{T}_{\aa}(F_N)$ as Taylor series at the origin. Let $g_{N+1}=T_{\boldsymbol{\a}}(F_N)_R$. Applying inequalities (\ref{inq Tpq}) for $N+1$ to inequality (\ref{f-f_N}) with $K(u)=u^{N+1}$, it follows that
$$\|T_{\boldsymbol{\a}}(f)_{\rho'}(t)-g_{N+1}\|_{\rho'}\leq R_\aa(\rho',\dots,\rho',\rho)A_{N+1}(\Pi_{\boldsymbol{\a}}')|t|^N,$$ in the corresponding sector, where $0<|t|<\rho'<\min\{\rho,R\}$. Thus we obtain that $T_{\boldsymbol{\a}}(f)_{\rho'}$ has $\hat{T}_{\aa}(\hat{f})$ as asymptotic expansion. But $\hat{T}_{\aa}(\hat{f})$ is given by the limit of the series $\hat{T}_{\aa}(F_N)$ in the $\mathfrak{m}$--topology of $\mathcal{E}^{\boldsymbol{\a}}_r[[t]]$, $\mathfrak{m}=(t)$, thus $f\sim^{\boldsymbol{\a}}\hat{f}$ on $\Pi_{\boldsymbol{\a}}$ as we wanted to show. The $s$--Gevrey case also follows since the sequence $(g_{N+1})_{N\in\N}$ has bounds of $s$--Gevrey type  if the sequence $(F_N)_{N\in\N}$ does.
\end{proof}

These characterizations allow us to prove that monomial asymptotic expansions are stable under sums, products and partial derivatives.
In particular, it follows from the relations (\ref{limits of x-0}) that if $f\sim^\aa\hat{f}$ on $\Pi_\aa$, then \begin{equation}\label{limits x-0 general}
a_{\boldsymbol{\beta}}=\lim_{\Pi'_{\boldsymbol{\a}}\ni \boldsymbol{x}\rightarrow\boldsymbol{0}} \frac{1}{\boldsymbol{\beta}!}\frac{\d^{\boldsymbol{\beta}}f}{\d \boldsymbol{x}^{\boldsymbol{\beta}}}(\boldsymbol{x}), \quad  f_{J,\bb_J}(\xx_{J^c})=\lim_{{\xx_J\rightarrow0}\atop {\boldsymbol{x}\in \Pi'_{\boldsymbol{\a}}}} \frac{1}{\bb_J!}\frac{\d^{\bb_J}f}{\d \xx_J^{\bb_J}}(\boldsymbol{x}),\quad \boldsymbol{\beta}\in\N^d, J\subsetneq [1,d],
\end{equation} In particular, $\hat{f}$ is completely determined by $f$.

\begin{nota}\label{On the other inqs for mon asyp}Assume that $f\sim_s^\aa \hat{f}=\sum a_\bb \xx^\bb$ on $\Pi_\aa$, and take $D,B>0$ such that $\|a_\bb\|\leq DB^{|\bb|}\min_{1\leq j\leq d}\{\beta_j!^{s/\a_j}\}$, for all $\bb\in\N^d$. We can also consider how $\hat{f}$ approximates $f$ for any index $\gg$, other than $N\aa$ as in Definition \ref{Def asym exp monomial}. Indeed, given $\gg\in\N^d$, if we consider $N=\max_{1\leq j\leq d}\lfloor\gamma_j/\a_j\rfloor+1$, then  $\00\leq N\aa-\gamma$. Thus, using Remark \ref{Bounds factorial min max}, we find that in any subsector $\Pi_\aa'$ of radius $\rho<\min_{1\leq j\leq d}1/B^{\a_j}$, $$
\|f(\xx)-\text{App}_\gg(\hat{f})(\xx)\|\leq  CA^N N!^s|\xx^{N\aa}|+\|\text{App}_{N\aa}(\hat{f})(\xx)-\text{App}_\gg(\hat{f})(\xx)\|,$$ but the second term is bounded by \begin{align*}
\sum_{N\aa-\gg\not\leq \boldsymbol{\delta}} \|a_{\gg+\boldsymbol{\delta}}\||\xx^{\gg+\boldsymbol{\delta}}|
\leq & DB^{|\gg|}|\xx ^{\gg}|\sum_{j=1}^{d}\sum_{\beta_j=0}^{N\a_j-\gamma_j-1} \sum_{\boldsymbol{\delta}\in \N^d, \delta_j=\b_j} B^{|\boldsymbol{\delta}|} |\xx^{\boldsymbol{\delta}}| \min_{1\leq l\leq d}(\gamma_l+\delta_l)!^{s/\a_l}\\
\leq & DB^{|\gg|}\left[\sum_{j=1}^{d} \sum_{\beta_j=0}^{N\a_j-\gamma_j-1} \frac{(\gamma_j+\beta_j)!^{s/\a_j}}{\prod_{k\neq j}(1-B\rho^{1/\a_k})}\right] |\xx^\gg|.
\end{align*}  Since $(\gamma_j+\beta_j)!<(\a_j N)!$, by taking into account Remark \ref{Bounds factorial min max} we can conclude that $$\|f(\xx)-\text{App}_\gg(\hat{f})(\xx)\| \leq \widetilde{C}\widetilde{A}^{|\gg|} \left(\max_{1\leq j\leq d} \gamma_j!^{s/\a_j}\right)|\xx^\gg|,$$ for adequate constants $\widetilde{C}, \widetilde{A}>0$ depending only on $\Pi_\aa'$.
\end{nota}

We can finally give a new, but expected, characterization of monomial asymptotic expansions in terms of bounded derivatives. We follow the proof of Proposition 3 in \cite{Haraoka}.

\begin{teor}\label{Bounded derivatives}Let $f\in\mathcal{O}(\Pi_{\boldsymbol{\a}},E)$, $\Pi_{\boldsymbol{\a}}=\Pi_{\boldsymbol{\a}}(a,b,r)$ and $\hat{f}\in \hat{\mathcal{O}}'_d(r',E)$, $r'\leq r$. Then $f\sim^{\aa}\hat{f}$ on $\Pi_{\aa}$ if and only if for each $\Pi'_{\aa}\subset \Pi_{\aa}$,  $\sup_{\xx\in\Pi'_{\aa}} \left\|\frac{1}{\bb!}\frac{\d^{\bb}f}{\d\xx^{\bb}}(\xx)\right\|$ is finite, for all $\bb\in\N^d$. More precisely, if $s>0$, the following assertions are equivalent:
\begin{enumerate}
\item $f\sim^{\aa}_s\hat{f}$ on $\Pi_{\aa}$,

\item For each $\Pi'_{\aa}\subset \Pi_{\aa}$, there are constants $C,A>0$ such that $$\sup_{\xx\in\Pi'_{\aa}} \left\|\frac{1}{(N\aa)!}\frac{\d^{N\aa}f}{\d\xx^{N\aa}}(\xx)\right\|\leq CA^{N} N!^s,\quad N\in\N.$$

\item For each $\Pi'_{\aa}\subset \Pi_{\aa}$, there are constants $C,A>0$ such that $$\sup_{\xx\in\Pi'_{\aa}} \left\|\frac{1}{\bb!}\frac{\d^{\bb}f}{\d\xx^{\bb}}(\xx)\right\|\leq CA^{|\bb|} \max\left\{\beta_1!^{s/\a_1},\dots,\beta_d!^{s/\a_d}\right\},\quad \bb\in\N^d.$$
\end{enumerate}
\end{teor}

\begin{proof}We only consider the statement for the $s$--Gevrey case. It is clear that (3) implies (2). To prove that (2) implies (1), we use Taylor's formula and equations (\ref{App Na}) and (\ref{limits x-0 general}) to write  \begin{align*}
\left\|f(\xx)-\sum_{n=0}^{N-1} f_{\aa,n}(\xx)\xx^{n\aa}\right\|=&\left\|f(\xx)-\text{App}_{N\aa}(\hat{f})(\xx)\right\|\\
=&\left\|\int_0^{x_1}\cdots\int_0^{x_d} \frac{(x_1-w_1)^{N\a_1-1}}{(N\a_1-1)!}\cdots\frac{(x_d-w_d)^{N\a_d-1}}{(N\a_d-1)!}\frac{\d^{N\aa}f}{\d\xx^{N\aa}}(\boldsymbol{w}) d\boldsymbol{w}\right\|\\
\leq&\sup_{\boldsymbol{w}\in\Pi'_{\aa}} \left\|\frac{1}{(N\aa)!}\frac{\d^{N\aa}f}{\d\xx^{N\aa}}(\boldsymbol{w})\right\||\xx^{N\aa}|\leq CA^N N!^s|\xx^{N\aa}|,
\end{align*} and then we can conclude the desired bounds. Note that the paths of integration are the segments $w_j=\rho_jx_j$, $0< \rho_j\leq 1$, and $\boldsymbol{w}=(\rho_1 x_1,\dots,\rho_d x_d)\in \Pi'_\aa$ if $\xx\in\Pi'_\aa$.

Finally, to show that (1) implies (3), take $\xx\in \Pi_\aa'\subset \Pi_\a''\subset\Pi_\a$, where   $\Pi_\aa'=\Pi_\aa'(a',b',\rho)$ and  $\Pi_\a''=\Pi_\a(a'',b'',\rho)$. Using Remark \ref{Polysectors}, we can find a polysector $P=\prod_{j=1}^d V_j$, with $V_j$ of opening $(b'-a')/(d-1)\a_j$, for $j=1,\dots,d-1$, and $b''-a''-(b'-a')$ for $j=d$, such that $\xx\in P\subset \Pi_\aa''$. Note that these openings are independent of the point $\xx\in \Pi_\aa'$. Thus we can find numbers $\sigma_1,\dots,\sigma_d>0$ depending only on $\Pi_\aa'$ and $\Pi_\aa''$ such that the circles $C_j$ given by $|w_j-x_j|=\sigma_j|x_j|$ are contained in $V_j$, for all $j=1,\dots,d$. Using Cauchy's formula and Remark \ref{On the other inqs for mon asyp} we find that \begin{align*}
\left\|\frac{1}{\bb!}\frac{\d^{\bb}f}{\d\xx^{\bb}}(\xx)\right\|&=\left\|\frac{1}{\bb!}\frac{\d^{\bb}}{\d\xx^{\bb}}\left(f-\text{App}_{\bb}(\hat{f})\right)(\xx)\right\|\\
&=\left\|\frac{1}{(2\pi i)^d}\int_{C_1}\cdots \int_{C_d} \frac{f(\boldsymbol{w})-\text{App}_{\bb}(\hat{f})(\boldsymbol{w})}{(w_1-x_1)^{\beta_1+1}\cdots (w_d-x_d)^{\beta_d+1}}dw_d\cdots dw_1\right\|\\
&\leq \left(\frac{\sigma_1+1}{\sigma_1}\right)^{\beta_1+1}\cdots\left(\frac{\sigma_d+1}{\sigma_d}\right)^{\beta_d+1} \widetilde{C}(\Pi_\aa'')\widetilde{A}(\Pi_\aa'')^{|\bb|}\cdot \max_{1\leq j\leq d} \beta_j!^{s/\a_j},
\end{align*} for some constants $\widetilde{C}(\Pi_\aa''), \widetilde{A}(\Pi_\aa'')>0$ independent of $\bb$. Then (3) follows.\end{proof}

%\|f(\xx)-\text{App}_\gg(\hat{f})(\xx)\| \widetilde{C}(\Pi_\aa'')\widetilde{A}(\Pi_\aa'')^{|\gg|} \left(\max_{1\leq j\leq d} M_{\gamma_j}^{1/\a_j}\right)|\xx^\gg|,

\begin{nota}The previous theorem shows that if $f\sim^{\aa}\hat{f}\in \mathcal{O}'_d(E)$ on $\Pi_{\aa}$,  then $f$ also has $\hat{f}$ as strong asymptotic expansion in Majima's sense, i.e., $f$ is \textit{strongly asymptotically developable as $\xx\rightarrow 0$} in any polysector properly contained in $\Pi_{\aa}$. Moreover, if  $f\sim^{\aa}_s\hat{f}$ on $\Pi_{\aa}$, then the strong asymptotic expansion is $(s/\a_1,\dots,s/\a_d)$--Gevrey, since $\max_{1\leq j\leq d }\, \beta_j!^{s/\a_j} \leq \beta_1!^{s/\a_1}\cdots\beta_d!^{s/\a_d}$, see e.g.,  \cite[Def. 3]{Haraoka}.	
\end{nota}

When we fix some of the variables in the monomial asymptotic expansion of a map, the expansion still holds in the remaining variables. We state this result when we only fix one variable. The proof is an immediate consequence of Proposition \ref{equiv f iff f_n for x^pe^q}.

\begin{prop}\label{summ fix a variable} Consider  $f\in\mathcal{O}(\Pi_{(\boldsymbol{\a},p)},E)$, $\Pi_{(\boldsymbol{\a},p)}=\Pi_{(\boldsymbol{\a},p)}(a,b,r)$, and $\hat{f}\in \hat{\mathcal{O}}_{d+1}'(E)$ such that $f\sim^{(\boldsymbol{\a},p)}_s\hat{f}$ on $\Pi_{(\boldsymbol{\a},p)}$. Then there is $r'>0$ such that, for all $z_0\in D_{r'}$, we have $f(\xx,z_0)\sim_s^\aa \hat{f}(\xx,z_0)$ on $\Pi_{\aa}(a', b',r')$, $a'=a-\text{arg}(z_0^p)$, $b'=b-\text{arg}(z_0^p)$.
%In particular, if $\hat{f}$ is $k$--summable in $\boldsymbol{x}^{\boldsymbol{\a}}$ in some direction $d$ then $\hat{f}(x_1,\dots,x_m,x_{m+1,0},\dots,x_{d,0})$ is $k$--summable in direction $d-\text{arg}(x_{m+1,0}^{\a_m}\cdots x_{d,0}^{\a_d})$.
\end{prop}

It is straightforward to characterize maps with null $s$--Gevrey asymptotic expansion in a monomial: $f\sim^\aa_s 0$ on $\Pi_\aa$ if and only if for every subsector $\Pi_\aa'\subset \Pi_\aa$, there are constants $C,A>0$ such that $$\|f(\xx)\|\leq C\exp(-1/A|\xx^\aa|^{1/s}),\quad \xx\in\Pi_\aa'.$$ In this framework Watson's lemma also holds: if $f\sim^\aa_s \hat{0}$ on $\Pi_\aa(a,b,r)$ and $b-a>s\pi$, then $f\equiv 0$. As in the case of one variable  it is natural to consider the following definition.

\begin{defi}Let $\hat{f}\in \hat{\mathcal{O}}_{d}'(E)$, $k>0$ and $\theta\in\R$ be a direction.
	\begin{enumerate}
		\item The series $\hat{f}$ is called \textit{$\xx^{\aa}$--$k$--summable  on $S_{\boldsymbol{\a}}=S_{\boldsymbol{\a}}(\theta,b-a,r)$} with sum $f\in\mathcal{O}(S_{\aa},E)$ if $b-a>\pi/k$ and $f\sim_{1/k}^{\boldsymbol{\a}} \hat{f}$ on $S_{\boldsymbol{\a}}$. We also say that $\hat{f}$ is \textit{$\xx^{\aa}$--$k$--summable in the direction $\theta$}. The space of $\xx^{\aa}$--$k$--summable series in the direction $\theta$ will be denoted by $E\{\boldsymbol{x}\}^{\boldsymbol{\a}}_{1/k,\theta}$.
		\item The series $\hat{f}$ is called \textit{$\xx^{\aa}$--$k$--summable} if it is $\xx^{\aa}$--$k$--summable in all directions, up to a finite number of them mod. $2\pi$ (the singular directions). The corresponding space is denoted by $E\{\boldsymbol{x}\}^{\boldsymbol{\a}}_{1/k}$.
\end{enumerate}
	
Note that both $E\{\boldsymbol{x}\}^{\boldsymbol{\a}}_{1/k,\theta}$ and $E\{\boldsymbol{x}\}^{\boldsymbol{\a}}_{1/k}$ are vector spaces, stable by partial derivatives, and they inherit naturally a structure of algebra when $E$ is a Banach algebra.
\end{defi}

\begin{nota}\label{On rank reduction} Given $\aa\in (\N^+)^d$, we note that formulas (\ref{Decomposition for ramifications}) and (\ref{Equation for f_beta}) can also be applied to formal power series. In particular, $\hat{f}\in \hat{\mathcal{O}}_{d}'(E)$ and $\hat{f}(\xx)=\sum_{\00\leq \bb<\aa} \xx^{\bb} \hat{f}_{\bb}(x_1^{\a_1},\dots,x_d^{\a_d})$, it is straightforward to show that $\hat{f}$ is $\xx^\aa$--$k$--summable in direction $\theta$ if and only if $\hat{f}_{\bb}$ is $\boldsymbol{z}^{\boldsymbol{1}}$--$k$--summable in direction $\theta$, for all $\00\leq \bb<\aa$, where ${\boldsymbol{1}}=(1,\dots,1)$ and $\boldsymbol{z}=(z_1,\dots,z_d)=(x_1^{\a_1},\dots,x_d^{\a_d})$.
\end{nota}

\section{Borel-Laplace analysis for monomial summability}\label{Borel-Laplace analysis for monomial summability}

The goal of this section is to generalize the Borel and Laplace transformations for monomial asymptotic expansions contained in \cite{CM2} to any number of variables, and  develop their main properties. We will prove that monomial summability is equivalent to Borel--summability in this framework. It is worth to remark that, in contrast with the approach in  \cite{CM2}, we have improved these results, since now we can allow some of the weights we use to be zero. This will be crucial in the application we present in Section \ref{Monomial summability of a family of singular perturbed PDEs}.

From now on, if $\boldsymbol{c}\in\R_{\geq0}^d$, we will write $J_{\boldsymbol{c}}:=\{j\in[1,d] \,|\, c_j\neq0\}$ for the set of indexes where $\boldsymbol{c}$ has nonzero entries.

\begin{defi}Consider $\aa\in (\N^+)^d$, $k>0$ and $\boldsymbol{s}\in\overline{\sigma_d}$. The \textit{$\xx^\aa$--$k$--$\boldsymbol{s}$--Borel transform} of a map $f$ is defined by the formula $$\mathcal{B}_{\boldsymbol{\lambda}}(f)(\boldsymbol{\xi})=\frac{(\boldsymbol{\xi}_{J_{\boldsymbol{s}}}^{k\boldsymbol{\a}_{J_{\boldsymbol{s}}}})^{-1}}{2\pi i}\int_\gamma f\left(\xi_1 u^{-\frac{s_1}{\a_1k}},\dots,\xi_d u^{-\frac{s_d}{\a_d k}}\right) e^u du,$$ \noindent where $\gamma$ denotes a Hankel path as we will explain below. Along this section we will write \begin{equation}\label{lambda}\boldsymbol{\lambda}=\left(\frac{s_1}{\a_1k},\dots,\frac{s_d}{\a_dk}\right),\qquad \boldsymbol{\lambda'}=\left(\frac{\a_1 k}{s_1},\dots,\frac{\a_d k}{s_d}\right).\end{equation} If $s_j=0$ for some $j$, we interpret the $j$th entry of $\boldsymbol{\lambda'}$ as $0$, and the variable $x_j$ remains unchanged. In this situation, to avoid cumbersome notation, we call it $\xi_j$ nevertheless. This convention will be used further on without explicit mention. Note that the factor outside the integral only includes the variables $\xi_j$ such that $s_j\neq0$. The ambient space $\C^d$ with coordinates $\xxi$ will be referred to as the \textit{$\xxi$--Borel plane}. 
\end{defi}

Since we admit that some weights can be zero, it is necessary to consider monomial sectors where some of the variables are bounded. Thus if
$f\in\mathcal{O}_b(S_{\boldsymbol{\a}},E)$,  $S_{\boldsymbol{\a}}=S_{\boldsymbol{\a}}(\theta, \pi/k+2\epsilon, R_0)$, $0<2\epsilon<\pi/k$, then $\mathcal{B}_{\boldsymbol{\lambda}}(f)$ will be defined and holomorphic on $$S_\aa^{\boldsymbol{s}}(\theta,2\epsilon,R_0):=S_\aa(\theta,2\epsilon)\cap \{\xxi\in \C^d \,|\, |\xi_j|^{\a_j}<R_0, j\not\in J_{\boldsymbol{s}}\}.$$ We will use the same notation for these sectors, for every $\boldsymbol{c}\in\R_{\geq0}^d$ other than $\boldsymbol{s}$. Note that if all entries of $\boldsymbol{s}$ are different from zero, then $S_\aa^{\boldsymbol{s}}(\theta,2\epsilon,R_0)$ is simply  $S_\aa(\theta,2\epsilon)$. 

If $\boldsymbol{\xi}\in S_{\aa}^{\boldsymbol{s}}(\theta, 2\epsilon',R_0)$, $0<\epsilon'<\epsilon$, we take $\gamma$ oriented positively and given by the arc of a circle centered at $0$ and radius $$R>\max_{j\in J_{\boldsymbol{s}}}\, (|\xi_j|^{\a_j}/R_0)^{\frac{k}{s_j}},$$ with endpoints on the directions $-\pi/2-k(\epsilon-\epsilon')$ and $\pi/2+k(\epsilon-\epsilon')$, and the half-lines with those directions from this arc to $\infty$. If $u$ goes along this path, the integrand is evaluated on $S_{\boldsymbol{\a}}$ and the integral converges absolutely, since $f$ is bounded and the exponential term tends to $0$ on those directions. The result is independent of $\epsilon'$ and $R$ due to Cauchy's theorem.

Using Hankel's formula for the Gamma function, we obtain the formula \begin{equation*}\label{borel lambda mu}
\mathcal{B}_{\boldsymbol{\lambda}}(\boldsymbol{x}^{\boldsymbol{\mu}})(\boldsymbol{\xi})=\frac{\boldsymbol{\xi}^{\boldsymbol{\mu}}\boldsymbol{\xi}_{J_{\boldsymbol{s}}}^{-k\boldsymbol{\a}_{J_{\boldsymbol{s}}}} }{\Gamma\left(\left<\boldsymbol{\mu,\boldsymbol{\lambda}}\right>\right)},\quad \boldsymbol{\mu}\in\C^d,
\end{equation*} and thus $\hat{\mathcal{B}}_{\boldsymbol{\lambda}}$, the \textit{formal $\xx^\aa$--$k$--$\boldsymbol{s}$--Borel transform}, can be defined for formal power series term by term.

By looking at the derivative with respect to $u$ of the integrand defining $\mathcal{B}_{\boldsymbol{\lambda}}$, it is natural to consider the vector field $X_{\boldsymbol{\lambda}}$ and its flow (at time $z$) $\phi^{\boldsymbol{\lambda}}_z$ given by $$X_{\boldsymbol{\lambda}}:=\frac{\boldsymbol{\xx}_{J_{\boldsymbol{s}}}^{k\boldsymbol{\a}_{J_{\boldsymbol{s}}}}}{k}\left(\frac{s_1}{\a_1}x_1\frac{\d}{\d x_1}+\cdots+\frac{s_d}{\a_d}x_d\frac{\d}{\d x_d}\right), \quad \varphi_z^{\boldsymbol{\lambda}}(\boldsymbol{x})=\sum_{j=1}^d \frac{x_j}{(1-z\boldsymbol{\xx}_{J_{\boldsymbol{s}}}^{k\boldsymbol{\a}_{J_{\boldsymbol{s}}}})^{s_j/\a_j k}}\boldsymbol{e}_j.$$

If $f\in\mathcal{O}_b(S_{\boldsymbol{\a}},E)$ is as before, it follows that
\begin{equation}
\label{Borel s1 s2 k y derivadas}
\mathcal{B}_{\boldsymbol{\lambda}}(X_{\boldsymbol{\lambda}}(f))(\boldsymbol{\xi})=\boldsymbol{\xxi}_{J_{\boldsymbol{s}}}^{k\boldsymbol{\a}_{J_{\boldsymbol{s}}}}\mathcal{B}_{\boldsymbol{\lambda}}(f)(\boldsymbol{\xi}), \quad
\mathcal{B}_{\boldsymbol{\lambda}}(f\circ\phi_z^{\boldsymbol{\lambda}})(\boldsymbol{\xi}):=\exp(z\boldsymbol{\xxi}_{J_{\boldsymbol{s}}}^{k\boldsymbol{\a}_{J_{\boldsymbol{s}}}})\mathcal{B}_{\boldsymbol{\lambda}}(f)(\boldsymbol{\xi}).
\end{equation}

Both formulas are naturally related since the first one is the linearization of the second one at $z=0$. In the variable $t=\xx^\aa$, the vector field $X_{\boldsymbol{\lambda}}$ reduces to $\frac{t^{k+1}}{k}\frac{\d}{\d t}$, and the first formula is just the one contained in the isomorphism (\ref{Iso structres 1 variable}).

We will say that $f$ has \textit{exponential growth of order at most $\boldsymbol{c}\in\R_{\geq0}^d$ on   $S_\aa^{\boldsymbol{c}}(\theta,b-a,R)$} if for every subsector  $S_\aa^{\boldsymbol{c}}(\theta',b'-a',R')$ ($a<a'<b'<b$, $R'<R$) there are  constants $C,M>0$ such that \begin{equation}\label{exp growth k, s_1 s_2}
\|f(\boldsymbol{\xi})\|\leq C\exp\left(M R_{\boldsymbol{c}}(\xxi)\right),\quad R_{\boldsymbol{c}}(\xxi):=\max_{j\in J_{\boldsymbol{c}}} |\xi_j|^{c_j},
\end{equation} on $S_\aa^{\boldsymbol{c}}(\theta',b'-a',R')$. Note we can also work with the term $\sum_{j\in J_{\boldsymbol{c}}} |\xi_j|^{c_j}$ in this bound, since $R_{\boldsymbol{c}}(\xxi)\leq \sum_{j\in J_{\boldsymbol{c}}} |\xi_j|^{c_j} \leq dR_{\boldsymbol{c}}(\xxi)$.

If $f$ is holomorphic on $\{\xxi\in\C^d \,|\, |\xi_j|<r_j, \, j\not\in J_{\boldsymbol{c}}\}$, for some fixed $r_j>0$ (resp. entire if $\boldsymbol{c}\in\R_{>0}^d$), and $\sum_{\bb\in\N^d} a_{\boldsymbol{\beta}}\boldsymbol{x}^{\boldsymbol{\beta}}$ is its Taylor series at the origin, condition  (\ref{exp growth k, s_1 s_2}) is equivalent to the existence of constants $D,B_1,\dots,B_d>0$ such that $$\|a_{\boldsymbol{\beta}}\|\leq \frac{D B_1^{\beta_1}\cdots B_d^{\beta_d}}{\Gamma\left(1+\sum_{j\in J_{\boldsymbol{c}}}\frac{\beta_j}{c_j}\right)},\quad  \text{ for all } \boldsymbol{\beta}\in\N^d.$$ This statement is standard and it can be deduced from Cauchy's integral formulas for the coefficients, Stirling's formula, and the inequalities \begin{equation}\label{inq Gamma}
\Gamma(1+a)\Gamma(1+b)\leq\Gamma(1+a+b)\leq 2^{a+b}\Gamma(1+a)\Gamma(1+b),\quad a,b>0,
\end{equation} satisfied by the Gamma function.

\begin{nota}\label{Borel an T_pq} Consider $\hat{f}\in \hat{\mathcal{O}}'_d(E)$, with $\hat{T}_{\aa}(\hat{f})=\sum f_{\aa,n} t^n$, and $\boldsymbol{\lambda}$ as in (\ref{lambda}). If we write  $\hat{\varphi}_{\boldsymbol{\lambda}}=\hat{\mathcal{B}}_{\boldsymbol{\lambda}}\left(\hat{f}-\sum_{k\aa_{J_{\boldsymbol{s}}}\not\leq \bb_{J_{\boldsymbol{s}}}} a_\bb \xx^\bb\right)$ and $\hat{T}_{\boldsymbol{\a}}(\boldsymbol{\xxi}_{J_{\boldsymbol{s}}}^{k\boldsymbol{\a}_{J_{\boldsymbol{s}}}}\hat{\varphi}_{\boldsymbol{\lambda}})=\sum \varphi_{\aa,n} \tau^n$, then \begin{equation}\label{varphi aa n}
\varphi_{\aa,n}(\xxi)=\sum_{\aa\not\leq \bb,\,  k\aa_{J_{\boldsymbol{s}}}\leq \bb_{J_{\boldsymbol{s}}}} \frac{a_{n\aa+\bb}}{\Gamma\left(\frac{n}{k}+\left<\bb,\boldsymbol{\lambda}\right>\right)} \xxi^{\bb}.
\end{equation} If $n\geq k$, then $k\a_l\leq n\a_l\leq n\a_l+\beta_l$, for all $l$. In particular, the condition $k\aa_{J_{\boldsymbol{s}}}\leq \bb_{J_{\boldsymbol{s}}}$ in the previous sum is satisfied and we can conclude that  $f_{\aa,n}$ and $\varphi_{\aa,n}$ are related by the formula $ \xxi_{J_{\boldsymbol{s}}}^{-k\aa_{J_{\boldsymbol{s}}}}\xxi^{n\aa}\varphi_{\aa,n}=\mathcal{B}_{\boldsymbol{\lambda}}(\boldsymbol{x}^{n\boldsymbol{\a}}f_{\aa,n})$, $n\geq k$. Since we can find $\rho>0$ such that $f_{\aa,n}\in \mathcal{O}(D_\rho^d,E)$, for all $n\in\N$, we see that the $\varphi_{\aa,n}$ are holomorphic maps on $ \{\xxi\in\C^d \,|\, |\xi_j|<\rho, j\not\in J_{\boldsymbol{s}}\}$, and a direct estimate using the expansion (\ref{varphi aa n}), shows that they satisfy bounds of type \begin{equation*}\label{growth of phi_m}
	\|\varphi_{\aa,n}(\boldsymbol{\xi})\|\leq \frac{L\|f_{\aa,n}\|_\rho}{\Gamma\left(1+\frac{n}{k}\right)}\exp\left(M R_{\boldsymbol{\lambda'}}(\xxi)\right),
	\end{equation*} where $L, M>0$ are some constants independent of $n$ but depending on $\rho$. 
\end{nota}

The behavior of the Borel transform with respect to monomial asymptotic expansions is presented in the next proposition. It is based on estimates included in \cite{Sanz} for the Borel transform in several variables.

\begin{prop}\label{exp growth of Borel k, s1 s2} Assume $f\sim_{s}^{\boldsymbol{\a}}\hat{f}$ on $S_{\boldsymbol{\a}}(\theta,\pi/k+2\epsilon, R_0)$, where  $0<2\epsilon<\pi/k$ and $\hat{f}=\sum_{k\aa_{J_{\boldsymbol{s}}}\leq \bb_{J_{\boldsymbol{s}}}} a_\bb \xx^\bb$. If $s>0$, $\boldsymbol{s}\in\overline{\sigma_d}$ and  $\boldsymbol{\lambda}, \boldsymbol{\lambda}'$ are given by (\ref{lambda}), the following statements are verified:

\begin{enumerate}
	\item If $g=\mathcal{B}_{\boldsymbol{\lambda}}(f)$ and  $\hat{g}=\hat{\mathcal{B}}_{\boldsymbol{\lambda}}(\hat{f})$, then  $\boldsymbol{\xxi}_{J_{\boldsymbol{s}}}^{k\boldsymbol{\a}_{J_{\boldsymbol{s}}}}g\sim^{\boldsymbol{\a}}_{s'} \boldsymbol{\xxi}_{J_{\boldsymbol{s}}}^{k\boldsymbol{\a}_{J_{\boldsymbol{s}}}}\hat{g}$ on $S_{\boldsymbol{\a}}^{\boldsymbol{s}}(\theta,2\epsilon,R_0)$, where $s'=\max\left\{s-\frac{1}{k},0\right\}$.
	
	\item For every unbounded subsector $S_{\boldsymbol{\a}}''$ of $S_{\boldsymbol{\a}}^{\boldsymbol{s}}(\theta,2\epsilon,R_0)$ there are $B,D,M>0$ such that $$\left\|g(\boldsymbol{\xi})-\sum_{n=0}^{N-1} g_{\aa,n}(\boldsymbol{\xi}) \boldsymbol{\xi}^{n\boldsymbol{\a}}\right\|\leq DB^N\Gamma(1+Ns')|\boldsymbol{\xi}^{N\boldsymbol{\a}}| |\boldsymbol{\xxi}_{J_{\boldsymbol{s}}}^{-k\boldsymbol{\a}_{J_{\boldsymbol{s}}}}| \exp\left(MR_{\boldsymbol{\lambda'}}(\xxi)\right)\text{ on }S_{\boldsymbol{\a}}'',$$ where $\hat{T}_{\boldsymbol{\a}}(\boldsymbol{\xxi}_{J_{\boldsymbol{s}}}^{k\boldsymbol{\a}_{J_{\boldsymbol{s}}}}\hat{g})=\sum (\boldsymbol{\xxi}_{J_{\boldsymbol{s}}}^{k\boldsymbol{\a}_{J_{\boldsymbol{s}}}}g_{\aa, n}) t^n$. If $N=0$, this inequality means that $\boldsymbol{\xxi}_{J_{\boldsymbol{s}}}^{k\boldsymbol{\a}_{J_{\boldsymbol{s}}}}g$ has exponential growth of order at most  $\boldsymbol{\lambda'}$ on $S^{\boldsymbol{s}}_{\boldsymbol{\a}}(\theta,2\epsilon,R_0)$.
\end{enumerate}
\end{prop}

\begin{proof}It is sufficient to prove (2). Thus, we have to establish those bounds for sectors of the form $S_{\aa}'=S_{\aa}^{\boldsymbol{s}}(\theta,2\epsilon',R_0)$ with $0<\epsilon'<\epsilon$. The proof relies on choosing adequately the radius of the arc of the path $\gamma$ in the definition. Write $\gamma=\gamma_1+\gamma_2-\gamma_3$ where $\gamma_1, \gamma_3$ denote the half-lines and $\gamma_2$ denotes the circular part, parameterized by $\gamma_2(\phi)=Re^{i\phi}$, $|\phi|\leq\pi/2+k(\epsilon''-\epsilon')/2$, where $0<\epsilon'<\epsilon''<\epsilon$ and $R$ will be chosen so that $\left(\xi_1 u^{-\frac{s_1}{\a_1k}},\dots,\xi_d u^{-\frac{s_d}{\a_d k}}\right)\in S_{\aa}=S_{\aa}(\theta,\pi/k+2\epsilon'', R_0/2)$, for all $u$ on $\gamma$ and $\xxi\in S_{\aa}'$.
	
We may assume that $\hat{T}_\aa(\hat{f})=\sum f_{\aa,n} t^n$, with $f_{\aa,n}\in \mathcal{O}_b(D_{R_0}^d,E)$ by reducing $R_0$ if necessary. By hypothesis, inequality (\ref{formula def asym x^pe^q}) holds with $C_N=CA^N\Gamma(1+Ns)$ on $S_{\boldsymbol{\a}}$. Setting  $a=\sin\left(\frac{k}{2}(\epsilon''-\epsilon')\right)>0$,  taking $R>1$ to be chosen, and by using the relation between $f_{\aa,n}$ and $g_{\aa,n}$ explained in Remark \ref{Borel an T_pq}, a direct estimate shows that \begin{align}
	\left\|g(\boldsymbol{\xi})-\sum_{n=0}^{N-1}g_{\aa,n}(\boldsymbol{\xi})\boldsymbol{\xi}^{n\boldsymbol{\a}}\right\|&\leq
	\frac{C}{a}A^N\Gamma(1+Ns)\frac{|\boldsymbol{\xi}^{N\boldsymbol{\a}}| |\boldsymbol{\xxi}_{J_{\boldsymbol{s}}}^{-k\boldsymbol{\a}_{J_{\boldsymbol{s}}}}|}{R^{N/k-1}}\left(\frac{e^{-aR}}{R}+e^R\right)\nonumber \\
	\label{proof Bks_1s_2 2}&\leq
	\frac{2C}{a}A^N\Gamma(1+Ns)|\boldsymbol{\xi}^{N\boldsymbol{\a}}| |\boldsymbol{\xxi}_{J_{\boldsymbol{s}}}^{-k\boldsymbol{\a}_{J_{\boldsymbol{s}}}}|\frac{e^R}{R^{N/k-1}}
	\qquad \text{ on } S_{\boldsymbol{\a}}',
	\end{align} for all $N\in\N$. For $N=0$ we are denoting $C=\sup_{\boldsymbol{x}\in S_{\boldsymbol{\a}}}\|f(\boldsymbol{x})\|$ (note that $f$ is bounded here, as we have reduced the radius and the opening of the sector).
	
	To prove (2) we divide our sector in two parts. First of all, consider $\boldsymbol{\xi}\in S_{\boldsymbol{\a}}(\theta,2\epsilon',r)$, with $r>R_0/2$ fixed. In this case choose $R\geq \max_{j\in J_{\boldsymbol{s}}}(2r/R_0)^{k/s_j}>1$. Since it is enough to establish the bounds for large $N$ we can suppose $N$ is large enough and take $R=N/k$. Then the bound follows using Stirling's formula, since asymptotically
	$$
	\frac{e^{R}}{R^{N/k-1}} = \frac{N/ke^{N/k}}{\left( N/k \right)^{N/k}} \sim \frac{\sqrt{2\pi} (N/k)^{3/2}}{\Gamma (1+N/k)},
	$$
	and also because $\Gamma(1+Ns)\leq 2^{sN}\Gamma (1+N/k)\Gamma\left(1+N(s-1/k)\right)$ if $s\geq 1/k$, due to the second inequality in (\ref{inq Gamma}).
	
	Now, we establish the bound in the complementary region, i.e., for $\boldsymbol{\xi}\in S_{\boldsymbol{\a}}'\setminus S_{\boldsymbol{\a}}(\theta,2\epsilon',r)$. For each $j\in J_{\boldsymbol{s}}$, choose $R_j<R_0/2$ and let us take $R=R(\xxi)=\max_{j\in J_{\boldsymbol{s}}}\left(|\xi_j|^{\a_j}/R_j\right)^{k/s_j}>1$. Note that $R(\boldsymbol{\xi})\leq \sum_{j\in J_{\boldsymbol{s}}} \left(|\xi_j|^{\a_j}/R_j\right)^{k/s_j}$ and
	$\frac{|\boldsymbol{\xi}_{J_{\boldsymbol{s}}}^{\boldsymbol{\a}_{J_{\boldsymbol{s}}}}|^k}{\prod_{j\in J_{\boldsymbol{s}}} R_j^k}\leq R(\boldsymbol{\xi})$ (second inequality of (\ref{prop the segment})). Then we can bound (\ref{proof Bks_1s_2 2}) by \begin{equation*}\label{last inequatily proof Borel}
	\frac{2C}{a}A^N\Gamma(1+Ns) \left|\xxi_{J^c_{\boldsymbol{s}}}^{N\aa_{J^c_{\boldsymbol{s}}}}\right| \prod_{j\in J_{\boldsymbol{s}}} R_j^{N-k}\exp\left(\left(|\xi_j|^{\a_j}/R_j\right)^{k/s_j}\right).
	\end{equation*}
	
Let $M=\max_{j\in J_{\boldsymbol{s}}}(4/R_0)^{k/s_j}>0$. For each $j\in J_{\boldsymbol{s}}$, we consider two cases: choose $R_j$$=|\xi_j|^{\a_j}/(s_jN/k)^{s_j/k}$ $< R_0/2$ as long as this inequality holds. Then $$	 R_j^{N-k}\exp\left(\left(|\xi_j|^{\a_j}/R_j\right)^{k/s_j}\right)=|\xi_j|^{(N-k)\a_j}\left(\frac{s_j N}{k}\right)^{s_j} \frac{e^{s_jN/k}}{(s_jN/k)^{s_jN/k}}.$$ In the second case, choose $R_j=R_0/4<R_0/2\leq |\xi_j|^{\a_j}/(s_jN/k)^{s_j/k}$ and then $$R_j^{N-k}\exp\left(\left(|\xi_j|^{\a_j}/R_j\right)^{k/s_j}\right)<|\xi_j|^{(N-k)\a_j}\left(\frac{s_j N}{k}\right)^{s_j} \frac{\exp\left(M |\xi_j|^{\a_j k/s_j}\right)}{(s_jN/k)^{s_jN/k}}.$$ In both cases, we conclude that $$R_j^{N-k}\exp\left(\left(|\xi_j|^{\a_j}/R_j\right)^{k/s_j}\right)\leq |\xi_j|^{(N-k)\a_j}\left(\frac{s_j N}{k}\right)^{s_j} \frac{e^{s_jN/k}}{(s_jN/k)^{s_jN/k}}\exp\left(M |\xi_j|^{\a_j k/s_j}\right).$$

Using Stirling's formula we conclude that there are constants $L,K>0$ such that $$\left\|g(\boldsymbol{\xi})-\sum_{n=0}^{N-1}g_{\aa,n}(\boldsymbol{\xi})\boldsymbol{\xi}^{n\boldsymbol{\a}}\right\|\leq LK^N \frac{\Gamma(1+Ns)}{\prod_{j\in J_{\boldsymbol{s}}}\Gamma\left(1+\frac{s_jN}{k}\right)}  |\boldsymbol{\xi}^{N\boldsymbol{\a}}|\,  |\boldsymbol{\xxi}_{J_{\boldsymbol{s}}}^{-k\boldsymbol{\a}_{J_{\boldsymbol{s}}}}| \exp\left(MR_{\boldsymbol{\lambda'}}(\xxi)\right),$$ on $S_{\boldsymbol{\a}}'\setminus S_{\boldsymbol{\a}}(\theta,2\epsilon',r)$. Finally, we can use the second inequality in (\ref{inq Gamma}) to conclude the proof.
\end{proof}

Now we move to the study of the Laplace transform, which will turn out be the inverse of the Borel transformation introduced above.

\begin{defi}Consider $\aa\in (\N^+)^d$, $k>0$ and $\boldsymbol{s}\in\overline{\sigma_d}$. The \textit{$\xx^\aa$--$k$--$\boldsymbol{s}$--Laplace transform  in direction $\phi$}, $|\phi|<\pi/2$, of a map $f$ is defined by the formula $$\mathcal{L}_{\boldsymbol{\lambda},\phi}(f)(\boldsymbol{x})=\boldsymbol{\xx}_{J_{\boldsymbol{s}}}^{k\boldsymbol{\a}_{J_{\boldsymbol{s}}}}\int_{0}^{e^{i\phi}\infty} f\left(x_1 u^{\frac{s_1}{\a_1k}},\dots, x_d u^{\frac{s_d}{\a_dk}}\right) e^{-u} du.$$
\end{defi}

As in the case of $\mathcal{B}_{\boldsymbol{\lambda}}$, if $s_j=0$ for some $j$, the variable $\xi_j$ is not affected by the transformation, although we will still call it $x_j$.

If the map $f\in\mathcal{O}(S_{\boldsymbol{\a}},E)$,  $S_{\boldsymbol{\a}}=S_{\boldsymbol{\a}}(\theta,b-a)$, had exponential growth of order $k$ in the monomial $\boldsymbol{\xi}^{\boldsymbol{\a}}$, i.e., $\|f(\boldsymbol{\xi})\|\leq C\exp(M|\boldsymbol{\xi}^{k\boldsymbol{\a}}|)$ on $S_{\boldsymbol{\a}}$, then by Proposition 3.3 $f$ would be a map depending only on $\boldsymbol{\xi}^{\boldsymbol{\a}}$. Instead, in view of Proposition \ref{exp growth of Borel k, s1 s2} we assume that $f$ is defined and has exponential growth of order at most $\boldsymbol{\lambda'}$ on $S_\aa^{\boldsymbol{s}}(\theta,b-a,R)$. If $f$ satisfies (\ref{exp growth k, s_1 s_2}) with  $\boldsymbol{c}=\boldsymbol{\lambda}'$, $\mathcal{L}_{\boldsymbol{\lambda},\phi}(f)$ converges if $\xx$ satisfies 
$$a-\phi/k<\text{arg}(\boldsymbol{x}^{\boldsymbol{\a}})<b-\phi/k,\quad MR_{\boldsymbol{\lambda'}}(\xx)<\cos\phi, \quad |x_j|^{\a_j}<R, j\not\in J_{\boldsymbol{s}}.$$ The domain in $\C^d$ defined by these conditions will be denoted by $D_{\aa}^{\boldsymbol{s}}(\theta-\phi/k,b-a;M,R)$,  indicating its bisecting direction $(b+a)/2-\phi/k=\theta-\phi/k$ and opening $b-a$. We will also denote by  $$D_{\aa}^{\boldsymbol{s}}(\theta,b-a+\pi/k;M,R):=\bigcup_{|\phi|<\pi/2} D_{\aa}^{\boldsymbol{s}}(\theta-\phi/k,b-a;M,R),$$ with bisecting direction $\theta$ and opening $b-a+\pi/k$. Note that given $a-\pi/2k<a'<b'<b+\pi/2k$, there is $r>0$ such that $S^{\boldsymbol{s}}_{\boldsymbol{\a}}(\theta',b'-a',r)\subset D_{\aa}^{\boldsymbol{s}}(\theta,b-a+\pi/k;M,R)$.

It follows that $\mathcal{L}_{\boldsymbol{\lambda},\phi}(f)$  is holomorphic on $D_{\aa}^{\boldsymbol{s}}(\theta-\phi/k,b-a;M,R)$. Furthermore, if we change direction $\phi$ by $\phi'$, we obtain an analytic continuation of $\mathcal{L}_{\boldsymbol{\lambda},\phi}(f)$ when $|\phi'-\phi|<k(b-a)$, a fact that follows directly from Cauchy's theorem. This process leads to a holomorphic map $\mathcal{L}_{\boldsymbol{\lambda}}(f)$ defined on $D_{\aa}^{\boldsymbol{s}}(\theta,b-a+\pi/k;M,R)$.

An adequate choice of the path $\gamma$ in the definition of $\mathcal{B}_{\boldsymbol{\lambda}}$, a limiting process and the residue theorem, as it is done for the case of one variable, see e.g. \cite[Thm 24, p. 82]{Balser2}, show that if $f\in\mathcal{O}_b(S_{\boldsymbol{\a}}^{\boldsymbol{s}}(\theta,\pi/k+2\epsilon,R_0),E)$, $0<2\epsilon<\pi/k$, then $$\mathcal{L}_{\boldsymbol{\lambda}}\mathcal{B}_{\boldsymbol{\lambda}}(f)=f, \text{ on the intersection of their domains.}$$

The operator  $\mathcal{L}_{\boldsymbol{\lambda}}$ is also injective as the usual Laplace transform. Thus, if $g$ is of exponential growth of order at most $\boldsymbol{\lambda'}$,  then $$\mathcal{B}_{\boldsymbol{\lambda}}\mathcal{L}_{\boldsymbol{\lambda}}(g)=g, \text{ on the intersection of their domains.}$$

We then define $\hat{\mathcal{L}}_{\boldsymbol{\lambda}}$, \textit{the formal $\xx^\aa$--$k$--$\boldsymbol{s}$--Laplace transform}, as the inverse of $\hat{\mathcal{B}}_{\boldsymbol{\lambda}}$. When we write a series as a series in a monomial, it is natural to ask what is the relation between its Laplace transform and the transform of its components. That is the content of next remark.

\begin{nota}\label{nota sobre Laplace formal y analitico}Let $\hat{f}=\sum_{\bb\in\N^d} a_{\boldsymbol{\beta}}\boldsymbol{\xi}^{\boldsymbol{\beta}}\boldsymbol{\xxi}_{J_{\boldsymbol{s}}}^{-k\boldsymbol{\a}_{J_{\boldsymbol{s}}}}$ be a formal power series and $\hat{T}_{\aa}(\boldsymbol{\xxi}_{J_{\boldsymbol{s}}}^{k\boldsymbol{\a}_{J_{\boldsymbol{s}}}}\hat{f})=\sum f_{\aa,n}\tau^n$. A necessary and sufficient condition for  $\hat{\mathcal{L}}_{\boldsymbol{\lambda}}(\hat{f})$ to be convergent is that $\boldsymbol{\xxi}_{J_{\boldsymbol{s}}}^{k\boldsymbol{\a}_{J_{\boldsymbol{s}}}}\hat{f}$ defines a holomorphic map $f$ of exponential growth of order at most $\boldsymbol{\lambda'}$ on $\{\xxi\in\C^d \,|\, |\xi_j|<r_j, j\not\in J_{\boldsymbol{s}}\}$, for some $r_j>0$ (resp. on $\C^d$ if $\boldsymbol{s}$ has all nonzero entries). Then $\mathcal{L}_{\boldsymbol{\lambda}}(f)$ is holomorphic in a polydisc at the origin and $\hat{\mathcal{L}}_{\boldsymbol{\lambda}}(\hat{f})$ is its Taylor series. 
	
Now assume that there are constants $s,B,D,M>0$ such that the family of maps $f_{\aa,n}$ are holomorphic and satisfy the bounds $$\|f_{\aa,n}(\boldsymbol{\xi})\|\leq DB^n\Gamma\left(1+ns\right)\exp\left(MR_{\boldsymbol{\lambda'}}(\xxi)\right),\quad\text{ on } \{\xxi\in\C^d \,|\, |\xi_j|<r_j, j\not\in J_{\boldsymbol{s}}\}.$$ In particular, $\boldsymbol{\xxi}_{J_{\boldsymbol{s}}}^{k\boldsymbol{\a}_{J_{\boldsymbol{s}}}}\hat{f}\in E[[\xxi]]_{s}^{\aa}$. Then all the maps $\mathcal{L}_{\boldsymbol{\lambda}}(\boldsymbol{\xxi}_{J_{\boldsymbol{s}}}^{-k\boldsymbol{\a}_{J_{\boldsymbol{s}}}}f_{\aa,n})$ are holomorphic in a common polydisc centered at the origin. Furthermore, if we write $\hat{T}_{\aa}(\hat{\mathcal{L}}_{\boldsymbol{\lambda}}(\hat{f}))=\sum h_{\aa,n} t^n$, then $f_{\aa,n}$ and $h_{\aa,n}$ are related by the formula  $h_{\aa,n}\xx^{n\aa}=\mathcal{L}_{\boldsymbol{\lambda}}(\xxi_{J_{\boldsymbol{s}}}^{-k\aa_{J_{\boldsymbol{s}}}}\xxi^{n\aa} f_{\aa,n})$. A direct estimate shows that $$\|h_{\aa,n}(\xx)\|\leq DB^n \Gamma(1+ns)\Gamma(n/k) \left(\cos\phi-MR_{\boldsymbol{\lambda'}}(\xx)\right)^{-1},$$ if $MR_{\boldsymbol{\lambda'}}(\xx)<\cos\phi$. Since $\Gamma(n/k)\leq n \Gamma(n/k)=k\Gamma(1+n/k)$, an application of the first inequality in (\ref{inq Gamma}) leads us to conclude that $\hat{\mathcal{L}}_{\boldsymbol{\lambda}}(\hat{f})\in E[[\boldsymbol{x}]]_{s+1/k}^{\boldsymbol{\a}}$. 
%$$\hat{\mathcal{L}}_{\boldsymbol{\lambda}}(\hat{f})=\sum_{n\geq 0} \mathcal{L}_{\boldsymbol{\lambda}}( \boldsymbol{\xxi}_{J_{\boldsymbol{s}}}^{-k\boldsymbol{\a}_{J_{\boldsymbol{s}}}} \boldsymbol{\xi}^{n\boldsymbol{\a}} f_{\aa,n}).$$
	
\end{nota}

The next technical proposition explains the behavior of the Laplace transform with respect to monomial asymptotic expansions. The hypotheses, although restrictive, are natural when compared to Proposition \ref{exp growth of Borel k, s1 s2}.

\begin{prop}\label{Laplace k s_1 s_2 properties} Consider $\boldsymbol{s}\in\overline{\sigma_d}$, $f\in \mathcal{O}(S_{\boldsymbol{\a}}^{\boldsymbol{s}},E)$, $S_{\boldsymbol{\a}}^{\boldsymbol{s}}=S_{\boldsymbol{\a}}^{\boldsymbol{s}}(\theta,b-a,R)$, $\boldsymbol{\xxi}_{J_{\boldsymbol{s}}}^{k\boldsymbol{\a}_{J_{\boldsymbol{s}}}}\hat{f}\in \hat{\mathcal{O}}'_d(E)$ with $\hat{T}_{\aa}(\boldsymbol{\xxi}_{J_{\boldsymbol{s}}}^{k\boldsymbol{\a}_{J_{\boldsymbol{s}}}}\hat{f})=\sum f_{\aa,n} t^n$ and $s\geq0$. Assume that: \begin{enumerate}
\item There are constants $B,D,K>0$ and $0<r<R$ such that $f_{\aa,n}$ are holomorphic and satisfy bounds of type $$\|f_{\aa,n}(\boldsymbol{\xi})\|\leq DB^n\Gamma\left(1+ns\right)\exp\left(K R_{\boldsymbol{\lambda'}}(\xxi)\right),\text{ on } \{\xxi\in\C^d \,|\, |\xi_j|^{\a_j}<r, j\not\in J_{\boldsymbol{s}}\}.$$

\item For every unbounded subsector $S_{\boldsymbol{\a}}'$ of $S_{\boldsymbol{\a}}^{\boldsymbol{s}}$ there are constants $C,A,M>0$ such that \begin{equation*}
		\left\|\boldsymbol{\xxi}_{J_{\boldsymbol{s}}}^{k\boldsymbol{\a}_{J_{\boldsymbol{s}}}}f(\boldsymbol{\xi})-\sum_{n=0}^{N-1} f_{\aa,n}(\boldsymbol{\xi}) \boldsymbol{\xi}^{n\boldsymbol{\a}}\right\|\leq CA^N\Gamma(1+Ns)|\boldsymbol{\xi}^{N\boldsymbol{\a}}|\exp\left(MR_{\boldsymbol{\lambda'}}(\xxi)\right)\text{ on }S_{\boldsymbol{\a}}', \end{equation*} for all $N\in\N$. In particular, $\boldsymbol{\xxi}_{J_{\boldsymbol{s}}}^{k\boldsymbol{\a}_{J_{\boldsymbol{s}}}} f\sim^{\aa}_{s} \boldsymbol{\xxi}_{J_{\boldsymbol{s}}}^{k\boldsymbol{\a}_{J_{\boldsymbol{s}}}} \hat{f}$ on $S_{\aa}^{\boldsymbol{s}}$.
\end{enumerate} Then,  $\mathcal{L}_{\boldsymbol{\lambda}}(f)\sim_{s'}^{\boldsymbol{\a}} \hat{\mathcal{L}}_{\boldsymbol{\lambda}}(\hat{f})$ on any sector in $\xx^\aa$ contained in  $D_{\aa}^{\boldsymbol{s}}(\theta,b-a+\pi/k;M,R)$, where $s'=s+1/k$.
\end{prop}

\begin{proof}If $N=0$, assertion (2) means that $f$ has exponential growth of order at most $\boldsymbol{\lambda'}$ on $S_{\boldsymbol{\a}}^{\boldsymbol{s}}$. Write  $h=\mathcal{L}_{\boldsymbol{\lambda}}(f)$ and $\hat{T}_{\aa}(\hat{\mathcal{L}}_{\boldsymbol{\lambda}}(\hat{f}))=\sum h_{\aa,n} t^n$ as in Remark \ref{nota sobre Laplace formal y analitico}. For a fixed $|\phi|<\pi/2$, it is enough to prove the estimates for subsectors $S_{\boldsymbol{\a}}''$ contained in $ D_{\aa}^{\boldsymbol{s}}(\theta-\phi/k,b-a;M,r)$. We can find $\delta>0$ small enough such that $MR_{\boldsymbol{\lambda'}}(\xx)<\cos\phi-\delta,$ on $S_{\boldsymbol{\a}}''$. Now, let $S_{\boldsymbol{\a}}'$ be a subsector of $S_{\boldsymbol{\a}}^{\boldsymbol{s}}$ such that $(x_1u^{\frac{s_1}{\a_1k}},\dots, x_du^{\frac{s_d}{\a_dk}})\in S_{\boldsymbol{\a}}'$ if $\boldsymbol{x}\in S_{\boldsymbol{\a}}''$ and $u$ is on the half-line from $0$ to $\infty$ in the direction $\phi$ . Applying the hypothesis (2) for $S_{\boldsymbol{\a}}'$ we see that $$\left\|h(\boldsymbol{x})-\sum_{n=0}^{N-1}h_{\aa,n}(\boldsymbol{x})\boldsymbol{x}^{n\boldsymbol{\a}}\right\|\leq\frac{CA^N}{\delta^{N/k}}\Gamma(1+Ns)\Gamma\left(N/k\right)|\boldsymbol{x}^{N\boldsymbol{\a}}|,\quad \text{ on }S_{\boldsymbol{\a}}'',$$ as we wanted to show. 
\end{proof}

Finally, we introduce a convolution product that shares similar properties with the classical one. Indeed, the \textit{$\xx^{\aa}$--k--$\boldsymbol{s}$--convolution between $f$ and $g$}, $\boldsymbol{s}\in\overline{\sigma_d}$, is defined by \begin{equation*}\label{convolution monomial}
	(f\ast_{\boldsymbol{\lambda}} g)(\boldsymbol{x})=\boldsymbol{\xx}_{J_{\boldsymbol{s}}}^{k\boldsymbol{\a}_{J_{\boldsymbol{s}}}}\int_0^1 f(x_1\tau^{\frac{s_1}{\a_1k}},\dots,x_d\tau^{\frac{s_d}{\a_dk}})g(x_1(1-\tau)^{\frac{s_1}{\a_1k}},\dots,x_d(1-\tau)^{\frac{s_d}{\a_dk}})d\tau,
\end{equation*} where $\boldsymbol{\lambda}$ is given by (\ref{lambda}). As an example we can compute, with the aid of the Beta function, the convolution between two monomials $$\frac{\boldsymbol{x}^{\boldsymbol{\mu}}}{\Gamma\left(\left<\boldsymbol{\mu},\boldsymbol{\lambda}\right>+1\right)}\ast_{\boldsymbol{\lambda}} \frac{\boldsymbol{x}^{\boldsymbol{\eta}}}{\Gamma\left(\left<\boldsymbol{\eta},\boldsymbol{\lambda}\right>+1\right)}=\frac{\boldsymbol{x}^{\boldsymbol{\mu}+\boldsymbol{\eta}}\boldsymbol{\xx}_{J_{\boldsymbol{s}}}^{k\boldsymbol{\a}_{J_{\boldsymbol{s}}}}}{\Gamma\left(\left<\boldsymbol{\mu}+\boldsymbol{\eta},\boldsymbol{\lambda}\right>+2\right)},$$ formula valid for  $\boldsymbol{\mu}, \boldsymbol{\eta}\in\C^d$ with entries of positive real part.

If $f$ and $g$ have exponential growth of order at most $\boldsymbol{\lambda'}$ (resp. belong to $\mathcal{O}_b(S_{\boldsymbol{\a}},E)$), then the same is valid for $f\ast_{\boldsymbol{\lambda}}g$ and \begin{equation}\label{conv. product}
\mathcal{L}_{\boldsymbol{\lambda}}(f\ast_{\boldsymbol{\lambda}} g)=\mathcal{L}_{\boldsymbol{\lambda}}(f)\mathcal{L}_{\boldsymbol{\lambda}}(g), \quad \text{(resp. }\mathcal{B}_{\boldsymbol{\lambda}}(fg)=\mathcal{B}_{\boldsymbol{\lambda}}(f)\ast_{\boldsymbol{\lambda}}\mathcal{B}_{\boldsymbol{\lambda}}(g)\text{)},
\end{equation} as in the classical case. This shows in particular that $\ast_{\boldsymbol{\lambda}}$ is a bilinear, commutative and associative binary operation, distributive  over addition.

\begin{nota}In analogy with the isomorphism explained in (\ref{Iso structres 1 variable}), for each $\aa\in(\N^+)^d$, $k>0$ and $\boldsymbol{s}\in \overline{\sigma_d}$, we have the following monomorphism between the structures  $$\left(E[[\xx]]_{1/k}^{\aa},+,\,\times\, ,X_{\boldsymbol{\lambda}}\right)\xhookrightarrow{\widehat{\mathcal{B}}_{\boldsymbol{\lambda}}} \left(\xxi^{-k\aa}E\{\xxi\},+,\,\ast_{\boldsymbol{\lambda}}\, ,\xxi^{k\aa}(\cdot)\right),$$  by taking into account (\ref{segment of line}) for the image and also (\ref{Borel s1 s2 k y derivadas})  and  (\ref{conv. product}).
\end{nota}

\begin{nota} We remark that all transformations introduced here reduce to their counterparts in one variable, for maps depending only on the corresponding monomial.
\end{nota}

At this point we are ready to define the summation methods based in the above Borel and Laplace transforms. We will see that they turn out to be equivalent to monomial summability.

\begin{defi}\label{def k-s_1,s_2-Borel sum} Consider  $\aa\in(\N^+)^d$, $k>0$, $\boldsymbol{s}\in\overline{\sigma_d}$, and let $\boldsymbol{\lambda}$, $\boldsymbol{\lambda'}$ be given by (\ref{lambda}). We will say that $\hat{f}=\sum_{\bb\in\N^d} a_{\bb}\xx^\bb\in$ $E[[\boldsymbol{x}]]_{1/k}^{\boldsymbol{\a}}$ is  $\xx^{\aa}$--$k$--$\boldsymbol{s}$--\textit{Borel summable in direction} $\theta$ if $\hat{\varphi}_{\boldsymbol{s}}=\hat{\mathcal{B}}_{\boldsymbol{\lambda}}(\hat{f}-\sum_{k\aa_{J_{\boldsymbol{s}}}\not\leq \bb_{J_{\boldsymbol{s}}} } a_\bb \xx^\bb)$ can be analytically continued, say as $\varphi_{\boldsymbol{s}}$, to a domain of the form $S_{\boldsymbol{\a}}^{\boldsymbol{s}}(\theta,2\epsilon,R)$ and having exponential growth of order at most $\boldsymbol{\lambda'}$ there. In this case, the $\xx^{\aa}$--$k$--$\boldsymbol{s}$--\textit{Borel sum of $\hat{f}$ in direction $\theta$} is defined by $$f(\boldsymbol{x})=\sum_{k\aa_{J_{\boldsymbol{s}}}\not\leq \bb_{J_{\boldsymbol{s}}} } a_\bb \xx^\bb+\mathcal{L}_{\boldsymbol{\lambda}}(\varphi_{\boldsymbol{s}})(\boldsymbol{x}).$$
\end{defi}

Note that the terms we have removed from $\hat{f}$ give an analytic map at the origin, since $\hat{f}\in \hat{\mathcal{O}}'_d(E)$.  Another way to avoid the use of power series with non-integer exponents is to consider  $\hat{\mathcal{B}}_{\boldsymbol{\lambda}}(\boldsymbol{\xx}_{J_{\boldsymbol{s}}}^{k\boldsymbol{\a}_{J_{\boldsymbol{s}}}}\hat{f})$, attempt analytic continuation with adequate exponential growth and to multiply by  $\boldsymbol{\xx}_{J_{\boldsymbol{s}}}^{-k\boldsymbol{\a}_{J_{\boldsymbol{s}}}}$ the corresponding Laplace transformation. We have not followed this equivalent approach since the introduction of such factor does not adapt well for the non-linear terms in the PDEs we consider in Section \ref{Monomial summability of a family of singular perturbed PDEs}.

\begin{teor}\label{Monomial summability t Borel Laplace} The following statements are equivalent for a series $\hat{f}\in E[[\xx]]_{1/k}^\aa$:
	\begin{enumerate}
		\item $\hat{f}\in E\{\boldsymbol{x}\}_{1/k,\theta}^{\boldsymbol{\a}}$, i.e., $\hat{f}$ is $\xx^{\aa}$--$k$--summable in the direction $\theta$.
		
		\item There is $\boldsymbol{s}\in\overline{\sigma_d}$ such that $\hat{f}$ is $\xx^{\aa}$--$k$--$\boldsymbol{s}$--Borel summable in direction $\theta$.
		
		\item For all $\boldsymbol{s}\in\overline{\sigma_d}$, $\hat{f}$ is $\xx^{\aa}$--$k$--$\boldsymbol{s}$--Borel summable in direction $\theta$.
	\end{enumerate} In all cases the corresponding sums coincide.
\end{teor}

\begin{proof}We may restrict out attention to the case $\hat{f}=\sum_{k\aa_{J_{\boldsymbol{s}}}\leq \bb_{J_{\boldsymbol{s}}} } a_\bb \xx^\bb.$ To show that (1) implies (3), assume $\hat{f}\in E\{\boldsymbol{x}\}_{1/k,\theta}^{\boldsymbol{\a}}$ and let  $f\in\mathcal{O}(S_{\boldsymbol{\a}}(\theta,\pi/k+2\epsilon, R_0),E)$ be its $\xx^{\aa}$--$k$--sum in direction $\theta$. For a fixed $\boldsymbol{s}\in\overline{\sigma_d}$ and $\boldsymbol{\lambda}$ as usual set  $\varphi_{\boldsymbol{s}}=\mathcal{B}_{\boldsymbol{\lambda}}(f)$ and $\hat{\varphi}_{\boldsymbol{s}}=\hat{\mathcal{B}}_{\boldsymbol{\lambda}}(\hat{f})$, convergent in some $D_r^d$. We can apply Proposition \ref{exp growth of Borel k, s1 s2} to $s=1/k$ to conclude that $\boldsymbol{\xxi}_{J_{\boldsymbol{s}}}^{k\aa_{J_{\boldsymbol{s}}}}\varphi_{\boldsymbol{s}}\sim_0^{\aa} \boldsymbol{\xxi}_{J_{\boldsymbol{s}}}^{k\aa_{J_{\boldsymbol{s}}}}\hat{\varphi}_{\boldsymbol{s}}$ on $S_{\aa}^{\boldsymbol{s}}(\theta,2\epsilon,R_0)$. These two properties imply that $\varphi_{\boldsymbol{s}}$ coincides with the sum of $\hat{\varphi}_{\boldsymbol{s}}$ on $S_{\boldsymbol{\a}}(d,2\epsilon)\cap D_r^d$. In other words,  $\hat{\varphi}_{\boldsymbol{s}}$ can be analytically continued and having exponential growth of order at most $\boldsymbol{\lambda'}$ on $S_{\boldsymbol{\a}}^{\boldsymbol{s}}(\theta,2\epsilon,R_0)$. Therefore, $\hat{f}$ is $\xx^{\aa}$--$k$--$\boldsymbol{s}$--Borel summable in direction $\theta$. Since $\mathcal{B}_{\boldsymbol{\lambda}}$ and $\mathcal{L}_{\boldsymbol{\lambda}}$ are inverses of each other, the $\xx^{\aa}$--$k$--$\boldsymbol{s}$--Borel sum of $\hat{f}$ is $f$.

The implication (3) to (2) is clear. Assuming (2), fix $\boldsymbol{s}\in\overline{\sigma_d}$ such that $\hat{f}$ is  $\xx^{\aa}$--$k$--$\boldsymbol{s}$--Borel summable in direction $\theta$, and let $\varphi_{\boldsymbol{s}}\in\mathcal{O}(S_{\boldsymbol{\a}}^{\boldsymbol{s}}(\theta,2\epsilon,R_0),E)$ and $\hat{\varphi}_{\boldsymbol{s}}$ be as in Definition \ref{def k-s_1,s_2-Borel sum}. Also write  $\hat{T}_{\boldsymbol{\a}}(\boldsymbol{\xxi}_{J_{\boldsymbol{s}}}^{k\boldsymbol{\a}_{J_{\boldsymbol{s}}}}\hat{\varphi}_{\boldsymbol{s}})=\sum \varphi_{\aa,n} \tau^n$. Then we can find constants $B,D>0$ such that \begin{equation}\label{Growth phi phi_n}\|\boldsymbol{\xxi}_{J_{\boldsymbol{s}}}^{k\boldsymbol{\a}_{J_{\boldsymbol{s}}}}\varphi_{\boldsymbol{s}}(\boldsymbol{\xi})\|\leq D\exp(MR_{\boldsymbol{\lambda'}}(\boldsymbol{\xi})),\quad \|\varphi_{\aa,n}(\boldsymbol{\xi})\|\leq DB^n \exp(MR_{\boldsymbol{\lambda'}}(\boldsymbol{\xi})),
\end{equation} the first on  $S_{\boldsymbol{\a}}^{\boldsymbol{s}}(d,2\epsilon',R')$, $0<\epsilon'<\epsilon$, $0<R'<R_0$ by our hypothesis, and the second on $\{\xxi\in\C^d \,|\, |\xi_j|^{\a_j}<\rho, j\not\in J_{\boldsymbol{s}}\}$ for some $0<\rho$, by using Remark \ref{Borel an T_pq}. We may assume $R_0<\rho$ by reducing $R_0$ if necessary.

To apply Proposition \ref{Laplace k s_1 s_2 properties} we have to show that there are constants $C,A>0$ such that, for all $N\in\N$, we have \begin{equation}\label{cond eq BL y MS}
\left\|\boldsymbol{\xxi}_{J_{\boldsymbol{s}}}^{k\boldsymbol{\a}_{J_{\boldsymbol{s}}}}\varphi_{\boldsymbol{s}}(\boldsymbol{\xi})-\sum_{n=0}^{N-1} \varphi_{\aa,n}(\boldsymbol{\xi}) \boldsymbol{\xi}^{n\boldsymbol{\a}}\right\|\leq CA^N|\boldsymbol{\xi}^{N\boldsymbol{\a}}|\exp(MR_{\boldsymbol{\lambda'}}(\boldsymbol{\xi})),\quad \text{ on }S_{\boldsymbol{\a}}^{\boldsymbol{s}}(\theta,2\epsilon',R').\end{equation} Since $\boldsymbol{\xxi}_{J_{\boldsymbol{s}}}^{k\boldsymbol{\a}_{J_{\boldsymbol{s}}}}\hat{\varphi}_{\boldsymbol{s}}$ is the convergent Taylor series of $\boldsymbol{\xxi}_{J_{\boldsymbol{s}}}^{k\boldsymbol{\a}_{J_{\boldsymbol{s}}}}\varphi_{\boldsymbol{s}}$ at the origin,  then (\ref{cond eq BL y MS}) is satisfied for all $|\xi_1|,\dots,|\xi_d|\leq R$, for some $R>0$. Due to (\ref{Growth phi phi_n}), the series of maps $\sum \varphi_{\aa,n}(\boldsymbol{\xi}) \boldsymbol{\xi}^{n\boldsymbol{\a}}$ converges absolutely in compact subsets of $\{\boldsymbol{\xi}\in \C^d \,|\, B|\boldsymbol{\xi}^{\boldsymbol{\a}}|<1\}$, and therefore $\boldsymbol{\xxi}_{J_{\boldsymbol{s}}}^{k\boldsymbol{\a}_{J_{\boldsymbol{s}}}}\varphi_{\boldsymbol{s}}$ can be analytically continued on this domain through this series. Thus, if $B|\boldsymbol{\xi}^{\boldsymbol{\a}}|<1/2$, inequality (\ref{cond eq BL y MS}) is also satisfied. On the other hand, using again (\ref{Growth phi phi_n}),  the left side of (\ref{cond eq BL y MS}) is bounded by  $$\left(D+\sum_{n=0}^{N-1} DB^n |\boldsymbol{\xi}^{n\boldsymbol{\a}}| \right) \exp(MR_{\boldsymbol{\lambda'}}(\xxi))\quad \text{ on } S_{\boldsymbol{\a}}^{\boldsymbol{s}}(\theta,2\epsilon',R').$$

\noindent If $1/2\leq B|\boldsymbol{\xi}^{\boldsymbol{\a}}|\leq 2$, the previous term is bounded by $D2^Ne^{MR_{\boldsymbol{\lambda'}}(\xxi)}\leq D(4B)^N|\boldsymbol{\xi}^{N\boldsymbol{\a}}|e^{MR_{\boldsymbol{\lambda'}}(\xxi)}$. If $B|\boldsymbol{\xi}^{\boldsymbol{\a}}|>2$, we can bound it by $$\left(D+D\frac{B^N|\boldsymbol{\xi}^{N\boldsymbol{\a}}|-1}{B|\boldsymbol{\xi}^{\boldsymbol{\a}}|-1}\right)e^{MR_{\boldsymbol{\lambda'}}(\xxi)}< D B^N|\boldsymbol{\xi}^{N\boldsymbol{\a}}|e^{MR_{\boldsymbol{\lambda'}}(\xxi)},$$ and thus (\ref{cond eq BL y MS}) is valid in all cases with $C,A$ large enough. By applying Proposition \ref{Laplace k s_1 s_2 properties} to $\varphi_{\boldsymbol{s}}$, $\hat{\varphi}_{\boldsymbol{s}}$ and $s=0$, we conclude that $f=\mathcal{L}_{\boldsymbol{\lambda}}(\varphi_{\boldsymbol{s}})\sim_{1/k}^{\boldsymbol{\a}} \hat{\mathcal{L}}_{\boldsymbol{\lambda}}(\hat{\varphi}_{\boldsymbol{s}})=\hat{f}$  on $D_{\aa}^{\boldsymbol{s}}(\theta,2\varepsilon+\pi/k;M,R_0)$.  In conclusion, $\hat{f}$ is $\xx^\aa$--$k$--summable in the direction $\theta$ and its sum can be found through the $\xx ^\aa$--$k$--$\boldsymbol{s}$--Laplace transform of $\varphi_{\boldsymbol{s}}$.
\end{proof}

As an immediate corollary we can relate monomial summability for different powers of a monomial. The proof follows from Theorem \ref{Monomial summability t Borel Laplace}, by noticing that $\boldsymbol{\lambda}=\left(\frac{s_1}{(N\a_1)\frac{k}{N}},\dots,\frac{s_d}{(N\a_d)\frac{k}{N}}\right)$, for all $N\in\N^+$.

\begin{coro}\label{p,q y Mp,Mq} Consider $\aa\in(\N^+)^d$, $k>0$ and $\theta\in\R$. Then $E\{\boldsymbol{x}\}_{1/k,\theta}^{\boldsymbol{\a}}=E\{\boldsymbol{x}\}_{N/k,N\theta}^{N\boldsymbol{\a}}$, and  $E\{\boldsymbol{x}\}_{1/k}^{\boldsymbol{\a}}=E\{\boldsymbol{x}\}_{N/k}^{N\boldsymbol{\a}}$, for all $N\in\N^+$.
\end{coro}

\section{Tauberian properties for monomial summability}\label{Tauberian properties for monomial summability}

The goal of this section is to recover Tauberian theorems for monomial summability. For instance, the relation between different levels of summability for distinct  monomials, and the comparison of summability in one variable with holomorphic coefficients in the remaining ones. The main tool we use to treat these situations are blow-ups with centers of codimesion two. We also establish  the behavior of Borel-Laplace transformations under these blow-ups.

In one variable we have the following two statements that provide Tauberian properties for $k$--summability.

\begin{teor}\label{Tauberian Ramis} The followings statements are true for $0<k<l$:\begin{enumerate}\item If $\hat{f}\in E\{x\}_{1/k}$ has no singular directions, then it is convergent.
\item $E[[x]]_{1/l}\cap E\{x\}_{1/k}=E\{x\}_{1/l}\cap E\{x\}_{1/k}=E\{x\}.$
\end{enumerate}
\end{teor}

Our goal is to extend this theorem for monomial summability. We know that a series $\hat{f}$ is $\xx^\aa$-$k$--summable in some direction $\theta$ if and only if 
there exist $r=r_\theta>0$ such that $\hat{T}_{\aa}(\hat{f}) $ is $k$--summable in direction $\theta$ in $\mathcal{E}^\aa_{r_\theta}$  in the classical sense. Unfortunately, $r_\theta$ might tend to 0 when $\theta$ tends to a singular direction. Therefore, $\xx^\aa$-$k$--summability of a series $\hat{f}$ does not imply that $\hat{T}_{\aa}(\hat{f})$ is $k$--summable in $\mathcal{E}^\aa_{r}$ for some fixed $r>0$. For a counterexample, see \cite{Monomial summ}, Section 6. However, we still can recover Theorem \ref{Tauberian Ramis}. We start with the following proposition.

%\begin{prop}\label{no sing directions then convergence for monomials} Let $0<k<l$ be positive numbers. Then the following statements are true:\begin{enumerate} \item $E\{\boldsymbol{x}\}^{\boldsymbol{\a}}_{1/k}\cap E\{\boldsymbol{x}\}^{\boldsymbol{\a}}_{1/l}=E\{\boldsymbol{x}\}^{\boldsymbol{\a}}_{1/k}\cap E[[\boldsymbol{x}]]^{\boldsymbol{\a}}_{1/l}=E\{\boldsymbol{x}\}$.	\end{enumerate}	\end{prop}

\begin{prop}\label{No sing direction implies convergence} If $\hat{f}\in E\{\boldsymbol{x}\}^{\boldsymbol{\a}}_{1/k}$ has no singular directions, then it is convergent.
\end{prop}

\begin{proof}Let us fix $\boldsymbol{s}\in\sigma_d$ and write $\hat{f}=\sum_{\bb\in\N^d} a_{\boldsymbol{\beta}}\boldsymbol{x}^{\boldsymbol{\beta}}$. We use $\xx^\aa$--$k$--$\boldsymbol{s}$--Borel summability as explained below Definition \ref{def k-s_1,s_2-Borel sum}, thus we consider $\varphi_{\boldsymbol{\lambda}}=\hat{\mathcal{B}}_{\boldsymbol{\lambda}}(\xx^{k\aa}\hat{f})$. If $\hat{f}$ has no singular directions for $\xx^\aa$--$k$--summability, Theorem \ref{Monomial summability t Borel Laplace} shows that for each direction $\theta\in [0,2\pi]$, there are constants $\delta_\theta, C_\theta, M_\theta>0$ such that $\|\varphi_{\boldsymbol{\lambda}}(\boldsymbol{\xi})\|\leq C_\theta\exp\left(M_\theta R_{\boldsymbol{\lambda'}}(\xxi) \right)$, for all $\xxi\in S_\aa(\theta,2\delta_\theta)$. Since the interval $[0,2\pi]$ is compact, we can choose a finite number of directions $\theta_1,\dots,\theta_n$ such that $[0,2\pi]\subseteq \cup_{j=1}^n (\theta_j-\delta_{\theta_j},\theta_j+\delta_{\theta_j})$. Then, the sectors $S_\aa(\theta_j,2\delta_{\theta_j})$, $j=1,\dots,n$, cover $\C^d\setminus\{x_1\cdots x_d=0\}$, and if we write $C=\max_{1\leq j\leq n} C_{\theta_j}$ and 
$M=\max_{1\leq j\leq n} M_{\theta_j}$, we find that  $$\|\varphi_{\boldsymbol{\lambda}}(\boldsymbol{\xi})\|\leq C\exp\left(M R_{\boldsymbol{\lambda'}}(\xxi) \right),\quad \text{ for all } \xxi \in \C^d\setminus\{x_1\cdots x_d=0\}.$$ Applying Cauchy's estimates we see that $$\frac{\|a_{\boldsymbol{\beta}}\|}{\Gamma\left(1+\sum_{j=1}^d\frac{\beta_js_j}{\a_jk}\right)}\leq C\prod_{j=1}^d\frac{e^{MR_j^{\a_jk/s_j}}}{R_j^{\beta_j}}, \quad \text{ for all } R_j>0.$$ Since the map $x\mapsto \exp(Mx^l)/x^n$, $l>0$, $n\in\N$, attains a minimum at $x=(n/Ml)^{1/l}$, if we choose $R_j=(\beta_js_j/M\a_jk)^{\frac{s_j}{\a_jk}}$, $j=1,\dots,d$, we conclude that $$\left\|a_{\boldsymbol{\beta}}\right\|\leq C\prod_{j=1}^d\left[\left(\frac{2^dMe\a_jk}{\beta_js_j}\right)^{\beta_js_j/\a_jk}\Gamma\left(1+\frac{\beta_js_j}{\a_jk}\right)\right].$$ \noindent Note we have used inequality (\ref{inq Gamma}) repeatedly. An application of Stirling's formula in each factor leads to the existence of constants $A,K>0$ such that $\|a_\bb\|\leq KA^{|\bb|}$, for all $\bb\in\N^d$, i.e., $\hat{f}$ is a convergent power series.
\end{proof}

To generalize Theorem \ref{Tauberian Ramis} (2), we consider the monomial transformations $$\pi_{ij}(x_1,\dots,x_d)=(x_1,\dots,\underbrace{x_ix_j}_{j \text{th entry}},\dots,x_d),$$ where $i,j=1,\dots,d$, $i\neq j$. Note that the maps $\pi_{ij}, \pi_{ji}$ correspond to the usual charts of the blow-up with the center of codimension two given by $\{x_i=x_j=0\}$. At the formal level we need the following lemma, whose proof is identical as the one of Lemma 3.6 of \cite{CM}.

\begin{lema}\label{f(xe,e) s-Gevrey in x^pe^(p+q)}Let $\hat{f}\in E[[\boldsymbol{x}]]$ be a formal power series. Then the following assertions are true:
	\begin{enumerate}
		\item $\hat{f}\in E\{\boldsymbol{x}\}$ if and only if $\hat{f}\circ\pi_{ij}\in E\{\boldsymbol{x}\}$ for some $i,j=1,\dots,d$.
		\item $\hat{f}\in E[[\boldsymbol{x}]]^{\boldsymbol{\a}}_{s}$ if and only if  there are  $i,j=1,...,d$, $i\neq j$ such that $\hat{f}\circ\pi_{ij}\in E[[\boldsymbol{x}]]^{\boldsymbol{\a}+\a_j\boldsymbol{e}_i}_{s}$ and $\hat{f}\circ\pi_{ji}\in E[[\boldsymbol{x}]]^{\boldsymbol{\a}+\a_i\boldsymbol{e}_j}_{s}$.
	\end{enumerate}
\end{lema}

%\begin{proof}We only prove the nontrivial implication in the second statement for some $\pi_{ij}$.  Let $\hat{f}=\sum a_{\boldsymbol{\beta}} \boldsymbol{x}^{\boldsymbol{\beta}}$ and write $\hat{f}\circ\pi_{ij}=\sum a'_{\boldsymbol{\beta}} \boldsymbol{x}^{\boldsymbol{\beta}}$ where $a'_{\boldsymbol{\beta}}=a_{\boldsymbol{\beta}-\beta_j\boldsymbol{e}_i}$ if $\beta_i\geq \beta_j$ and $0$ otherwise. If $\|a_{\boldsymbol{\beta}}\|\leq CA^{|\boldsymbol{\beta}|} \min_{1\leq m\leq d}\{\beta_m!^{s/\a_m}\}$ then in particular $\|a'_{\boldsymbol{\beta}}\|\leq CA^{|\boldsymbol{\beta}|-\beta_j} \beta_j!^{\lambda/\a_j}(\beta_i-\beta_j)!^{(1-\lambda)/\a_i},$ for all $\beta_i\geq \beta_j$ and $0\leq \lambda \leq 1$. The result follows taking $\lambda=\a_j/(\a_j+\a_i)$.\end{proof}

To establish Lemma \ref{f(xe,e) s-Gevrey in x^pe^(p+q)} (2) for summability we will use the following interesting relation between the monomial Borel and Laplace transformations and the blow-up maps $\pi_{ij}$. Fix $\aa\in(\N^+)^d$, $k>0$ and $\boldsymbol{s}\in \overline{\sigma_d}$. If for some indexes $i\neq j$ we have $s_j \a_i\geq s_i\a_j$, ($s_i=0$ if $s_j=0$), then $\boldsymbol{s}'=(s_1',\dots,s_d')\in \overline{\sigma_d}$, where $s_l'=s_l$, if $l\neq i, j$, $s_i'=s_i+\frac{\a_j}{\a_ i} s_i$ and $s_j'=s_j-\frac{\a_j}{\a_ i} s_i$. Furthermore, if $\boldsymbol{\lambda}=(\frac{s_1}{\a_1 k},\dots,\frac{s_d}{\a_d k})$, then $\boldsymbol{\lambda}-\frac{s_i}{\a_ik}\boldsymbol{e}_j=(\frac{s_1'}{\a_1' k},\dots,\frac{s_d'}{\a_d' k})$, where $\aa'=\aa+\a_j\boldsymbol{e}_i$, and the $\xx^\aa$--$k$--$\boldsymbol{s}$--Borel (resp. Laplace) and  $\xx^{\aa'}$--$k$--$\boldsymbol{s}'$--Borel (resp. Laplace) transformations are related by the formulas  \begin{align}\label{Borel blow up}\mathcal{B}_{\boldsymbol{\lambda}}(f)\circ \pi_{ij}(\boldsymbol{\xi})&=\mathcal{B}_{\boldsymbol{\lambda}-\frac{s_i}{\a_ik}\boldsymbol{e}_j}(f\circ \pi_{ij})(\boldsymbol{\xi}),\\ \label{Laplace blow up}\mathcal{L}_{\boldsymbol{\lambda}}(f)\circ \pi_{ij}(\boldsymbol{\xi})&=\mathcal{L}_{\boldsymbol{\lambda}-\frac{s_i}{\a_ik}\boldsymbol{e}_j}(f\circ \pi_{ij})(\boldsymbol{\xi}),\end{align} whenever the functions are defined. The same relations hold for their formal counterparts.

The next theorem corresponds to Theorem 7.24 in \cite{Sum wrt germs} for monomial summability. Although our approach follows the Borel--Laplace analysis developed in the previous section, the idea of proof is based on the same arguments.

\begin{teor}\label{monomial sum and blowups}$\hat{f}\in E\{\boldsymbol{x}\}^{\boldsymbol{\a}}_{1/k,\theta}$ if and only if there exist $i\neq j$ such that $\hat{f}\circ\pi_{ij}\in E\{\boldsymbol{x}\}^{\boldsymbol{\a}+\a_j\boldsymbol{e}_i}_{1/k,\theta}$ and $\hat{f}\circ\pi_{ji}\in E\{\boldsymbol{x}\}^{\boldsymbol{\a}+\a_i\boldsymbol{e}_j}_{1/k,\theta}$. %it has $k$--sum $f\circ\pi_{ij}$ in the direction $\theta$, for all $i,j=1,\dots,d$, $i\neq j$.
\end{teor}

\begin{proof}Using Proposition \ref{equiv f iff f_n for x^pe^q} we see that if $(f_N)_{N\in\N}$ is a family of bounded analytic functions that provide the monomial asymptotic expansion of $f$, then $(f_N\circ\pi_{ij})_{N\in\N}$ will provide the asymptotic expansion of $f\circ\pi_{ij}$, $i,j=1,\dots,d$, $i\neq j$.
	
Conversely, it is enough to do the proof for the case $\aa=\boldsymbol{1}=(1,1,\dots,1)$, see Remark \ref{On rank reduction}. To fix ideas, we take $i=1, j=2$. It is also not restrictive to assume that $\hat{f}=\sum_{k\leq \beta_1, \beta_2} a_\bb \xx^\bb$ . Applying Theorem \ref{Monomial summability t Borel Laplace}, we know by hypothesis that  $\hat{f}\circ\pi_{12}$ is $\xx^{\boldsymbol{1}+\boldsymbol{e}_1}$ --$k$--$\boldsymbol{e}_1$--summable and $\hat{f}\circ\pi_{21}$ is $\xx^{\boldsymbol{1}+\boldsymbol{e}_2}$--$k$--$\boldsymbol{e}_2$--summable, both in direction $\theta$. This means that we can find $\varepsilon>0$ and $R_0<1$ such that the maps  $$\varphi_1(\xxi)=\hat{\mathcal{B}}_{\frac{1}{2k}\boldsymbol{e}_1}(\hat{f}\circ \pi_{12}),\quad \varphi_2(\xxi)=\hat{\mathcal{B}}_{\frac{1}{2k}\boldsymbol{e}_2}(\hat{f}\circ \pi_{21}),$$ can be analytically continued to the domains  $S^{\boldsymbol{e}_1}_{\boldsymbol{1}+\boldsymbol{e}_1}(\theta,2\epsilon,R_0)$ and  $S^{\boldsymbol{e}_2}_{\boldsymbol{1}+\boldsymbol{e}_2}(\theta,2\epsilon,R_0)$, and furthermore, there are constants $C,M>0$ such that \begin{equation}\label{Exp growth p1 p2}
\|\varphi_1(\xxi)\|\leq C\exp(M|\xi_1|^{2k}),\quad \|\varphi_2(\xxi)\|\leq C\exp(M|\xi_2|^{2k}),	
\end{equation} on their respective domains.

We will prove that $\hat{f}\in E\{\boldsymbol{x}\}^{\boldsymbol{1}}_{1/k,\theta}$ by showing that $\hat{f}$ is $\xx^{\boldsymbol{1}}$--$k$--$\boldsymbol{s}$--summable in direction $\theta$, where $\boldsymbol{s}=\frac{1}{2}\left(\boldsymbol{e}_1+\boldsymbol{e}_2\right)$. In this case, $\boldsymbol{\lambda}$ and $\boldsymbol{\lambda}'$ in (\ref{lambda}) are given by  $\boldsymbol{\lambda}=\frac{1}{2k}(\boldsymbol{e}_1+\boldsymbol{e}_2)$ and $\boldsymbol{\lambda}'=2k(\boldsymbol{e}_1+\boldsymbol{e}_2)$. We know that $\varphi=\hat{\mathcal{B}}_{\boldsymbol{\lambda}}(\hat{f})$ is analytic in a polydisk at the origin. By reducing $R_0$ if necessary, we assume that $\varphi\in\mathcal{O}(D_{R_0}^d,E)$.

To finish the proof we show that $\varphi$ can be analytically continued to $S^{\boldsymbol{s}}_{\boldsymbol{1}}(\theta,2\epsilon,R_0)$ with exponential growth at most $\boldsymbol{\lambda}'$. First of all, we use formulas (\ref{Borel blow up}) to write $$\hat{\mathcal{B}}_{\boldsymbol{\lambda}}(\hat{f})(\xi_1,\xi_1\xi_2,\xxi'')=\hat{\mathcal{B}}_{\frac{1}{2k}\boldsymbol{e}_1}(\hat{f}\circ \pi_{12})(\boldsymbol{\xi}),\quad  \hat{\mathcal{B}}_{\boldsymbol{\lambda}}(\hat{f})(\xi_1\xi_2,\xi_2,\xxi'')=\hat{\mathcal{B}}_{\frac{1}{2k}\boldsymbol{e}_2}(\hat{f}\circ \pi_{21})(\boldsymbol{\xi}),$$ where $\xxi=(\xi_1,\xi_2,\xxi'')$. Then, we can extend $\varphi$ by the rules \begin{equation}\label{Def varphi}
\varphi(\xxi)=\left\{\begin{array}{ll}\varphi_1(\xi_1,\xi_2/\xi_1,\xxi''), &\text{ if }\quad |\xi_2|<R_0|\xi_1|,\\ \varphi_2(\xi_1/\xi_2,\xi_2,\xxi''), &\text{ if }\quad |\xi_1|<R_0|\xi_2|,\end{array}\right.	
\end{equation} as long as $\xxi$ satisfies $|\text{arg}(\xxi^{\boldsymbol{1}})-\theta|<\epsilon$, $|\xi_j|<R_0, j\neq 1,2$, and $1/R_0<|\xi_2/\xi_1|$ or $|\xi_2/\xi_1|<R_0$. Note that $\varphi$ has exponential growth at most $\boldsymbol{\lambda}'$ in this domain due to inequalities (\ref{Exp growth p1 p2}). Finally, to remove the restrictions on the norms of $\xi_1, \xi_2$, we use Cauchy's integral formula. To simplify notation, we use the auxiliary variable $w=\xi_3\cdots \xi_d$ (the case $d=2$ does not require $w$ or $\xxi''$ above). Since we are working with the monomial $\xxi^{\boldsymbol{1}}$, we also introduce the variable $\tau=\xxi^{\boldsymbol{1}}=\xi_1\xi_2 w$. Then we define the map \begin{equation}\label{Def G}
G(\xi_1,\tau,\xxi'')=\frac{1}{2\pi i} \left(\int_{|\zeta|=R}-\int_{|\zeta|=r}\right) \frac{\varphi\left(\zeta,\frac{\tau/w}{\zeta} ,w\right)}{\zeta-\xi_1}d\zeta,
\end{equation} where $\varphi$ is given by (\ref{Def varphi}), thus, we require that $R_0^{-1}r^2<|\tau/w|<R_0R^2$ and $0<r<R$. Note that the integral is independent of $r$ and $R$ as long as these constraints are satisfied. We will check that $G$ defines an holomorphic map on $\Omega= \C^\ast\times S(\theta,2\varepsilon)\times (D^{d-2}_{R_0}\setminus\{w=0\})=\C^\ast\times \Omega_1$. Then $G$ will provide the required extension to $S^{\boldsymbol{s}}_{\boldsymbol{1}}(\theta,2\epsilon,R_0)$, since $G(\xi_1,\xxi^{\boldsymbol{1}},\xxi'')=\varphi(\xxi)$ if $R_0^{-1}r^2<|\xi_1\xi_2|<R_0R^2$ and $r<|\xi_1|<R$, due to Cauchy's integral formula.

To prove the holomorphy of $G$, consider $U\subseteq K\subseteq \C^\ast$ and $U_1\subseteq K_1\subseteq \Omega_1$ where $U,U_1$ are open and $K,K_1$ are compact. Write  $L_1=\inf_{(\tau,\xxi'')\in K_1} |\tau/w|$ and $L_2=\sup_{(\tau,\xxi'')\in K_1} |\tau/w|$, both finite positive numbers, since $\tau$ and $w$ do not vanish on $\Omega_1$. Then choose $0<r<R$ such that $r^2<R_0L_1<R_0^{-1}L_2<R^2$ (recall that $R_0<1$) and $r<\inf_{\xi_1\in K} |\xi_1|\leq \sup_{\xi_1\in K} |\xi_1|<R$. Then $G$ is defined at all points of $U\times U_1$ and it is clearly holomorphic there. 

Finally, we need to show that the extension of $\varphi$ has exponential growth or order at most $\boldsymbol{\lambda}'$ on $S^{\boldsymbol{s}}_{\boldsymbol{1}}(\theta,2\epsilon,R_0)$, for $\xi_1 ,\xi_2$ such that $R_0\leq |\xi_2/\xi_1|\leq R_0^{-1}$. By calculating the values of $\varphi$ in (\ref{Def G}) for $|\zeta|=R$ using $\varphi_1$ and for $|\zeta|=r$ using $\varphi_2$ instead, we can employ inequalities (\ref{Exp growth p1 p2}) to find that  $$\|\varphi(\xxi)\|=\|G(\xi_1,\xxi^{\boldsymbol{1}},\xxi'')\|\leq \frac{CR\exp\left(MR^{2k}\right)}{\text{dist}(\xi_1,\d D_R)}+\frac{Cr\exp\left(M|\tau/rw|^{2k}\right)}{\text{dist}(\xi_1,\d D_r)}.$$ Since $\tau/w=\xi_1\xi_2$, by taking $R^2=4|\xi_1\xi_2|/R_0$ and $r^2=R_0|\xi_1\xi_2|/4$, we conclude that $$\|\varphi(\xxi)\|\leq 
\left(\frac{R}{\left|R-|\xi_1|\right|}+\frac{r}{\left|r-|\xi_1|\right|}\right)C\exp\left(M(4/R_0)^k |\xi_1\xi_2|^k\right).$$ Note that the denominators do not vanish and are uniformly bounded due to the restriction $R_0\leq |\xi_2/\xi_1|\leq R_0^{-1}$. The conclusion now follows by noting that $|\xi_1\xi_2|^k\leq \max\{|\xi_1|^{2k},|\xi_2|^{2k}\}$.

\end{proof}

%\begin{nota}It is also possible to consider the charts of the blow-up of the origin given by $\pi_j(x_1,\dots,x_d)=(x_1x_j,\dots,x_j,\dots,x_dx_j)$, $j=1,\dots,d$. Since $\pi_j=\pi_{j1}\circ\cdots\circ\pi_{j,j-1}\circ \pi_{j,j+1}\circ\cdots\circ\pi_{j,d}$, for all $j=1,\dots,d$, Proposition \ref{monomial sum and blowups} implies that if $\hat{f}\in E\{\boldsymbol{x}\}^{\boldsymbol{\a}}_{1/k,d}$ has $k-$sum $f$ in direction $d$ then $\hat{f}\circ\pi_j\in E\{\boldsymbol{x}\}^{\boldsymbol{\a}+(|\a|-\a_j)\boldsymbol{e}_j}_{1/k,d}$ and have $k-$sum $f\circ\pi_{j}$ in direction $d$, for all $j=1,\dots,d$.
%\end{nota}

We are ready to state and prove the third main result so far, comparing summable series in different monomials. This result was obtained in \cite{CM}. Although the content is correct, the proof given there is based on the false statement discussed above Proposition \ref{No sing direction implies convergence}. This is repaired here.

\begin{teor}\label{tauberian general case}Consider $\aa,\aa'\in(\N^+)^d$ and $k, k'>0$. The following statements hold:
	
	\begin{enumerate}
		\item If $\hat{f}\in E\{\boldsymbol{x}\}_{1/k}^{\boldsymbol{\a}}$ and $\hat{T}_\aa(\hat{f})$ is an $s$--Gevrey series with some $s<1/k$, then $\hat{f}$ is convergent. In particular, if $\max_{1\leq j\leq d}\{\a_j/\a_j'\}< k'/k$,  then $E\{\xx\}_{1/k}^{\aa}\cap  E[[\xx]]_{1/k'}^{\aa'}=E\{\xx\}$.
		
		\item $E\{\xx\}_{1/k}^{\aa}\cap E\{\xx\}_{1/k'}^{\aa'}=E\{\xx\}$, except in the case $k\aa=k'\aa'$, where $E\{\boldsymbol{x}\}_{1/k}^{\boldsymbol{\a}}=E\{\boldsymbol{x}\}_{1/k'}^{\boldsymbol{\a}'}$.
	\end{enumerate}
\end{teor}

\begin{proof} (1) The proof of the first statement is based in the proof of Theorem 3.8.2 in \cite{RS89}. 
	We write $\hat T_\aa(\hat f)(t)=\sum_{n=0}^\infty f_{\aa,n}t^n$, with $f_{\aa,n}\in \mathcal{E}_r^{\boldsymbol{\a}}$ and 
	use the
	$k$--Borel transform $g$ of $\hat T_\aa (\hat{f})$ in the form $g(\xx,\xi)=\sum_{n=0}^\infty \frac{f_{\aa,n}(\xx)}{\Gamma(1+n/k)}\xi^n$.
	Since $\hat T_\aa(\hat{f})$ is $s$--Gevrey with some $s<1/k$, we find constants $K,A>0$ such that $$\frac{\|f_{\aa,n}(\xx)\|}{\Gamma(1+n/k)}\leq K A^n n!^{-1/\mu},\quad \text{for all } \xx\in D_r^d, n\in\N,\quad  1/\mu:=1/k-s.$$ This implies that $g$ defines a holomorphic function on $D_r^d\times\C$ and we can find constants $L,B>0$ such that
	\begin{equation*}\label{expomu}\|g(\xx,\xi)\|\leq L \exp(B|\xi|^\mu),\quad \mbox{ for all }  (\xx,\xi)\in D_r^d\times \C.\end{equation*}
	
	We show now that $\hat f$ has no singular directions for $\xx^\aa$--$k$--summability, thus, it is convergent due to Proposition \ref{No sing direction implies convergence}. Indeed, arguing by contradiction, we assume $\theta$ is a singular direction of $\hat f$. We choose $0<\delta<\frac\pi{2\mu}$ such that $\hat{T}_\aa(\hat{f})$ is $k$--Borel--summable in the directions $\theta_{\pm}=\theta\pm\delta$, in some $\mathcal{E}^\aa_{r'}$, $0<r'<r$. Then, there exist $0<\rho<r'$ and $M,C>0$ such that $g$ satisfies
	\begin{equation*}\label{expok}\|g(\xx,\xi)\|\leq M\exp(C|\xi|^k),\quad \mbox{ for all } \xx\in D_{\rho}^d,\,  \arg(\xi)=\theta_{\pm}.\end{equation*}
	We can use the Phragm\'{e}n-Lindel\"{o}f principle, see e.g. \cite[Thm. 70, p. 235]{Balser2},  to show that $g$ also has exponential growth of order $k$ on the sector $S(\theta, 2\delta)$. Indeed, consider the map 
	$h(\xx,\xi)=g(\xx,\xi)\exp\left(-D(\xi e^{-{i\theta}})^k\right)$, where $D\cos(k\delta)=C$. Using the previous bounds, it follows that $\|h(\xx,\xi)\|\leq M$ if $\arg(\xi)=\theta_{\pm}$, and 
	$\|h(\xx,\xi)\|\leq L \exp(B|\xi|^\mu)$, if $\arg(\xi)\in[\theta_-,\theta_+]$, for all $\xx\in D_\rho^d$. Since the opening of $S(\theta,2\delta)$ is
	smaller than $\pi/\mu$, Phragm\'{e}n-Lindel\"{o}f principle yields that $h$ is bounded on the full sector. Thus we can find constants $\tilde M,\tilde C>0$ such that $$\|g(\xx,\xi)\|\leq \tilde M\exp(\tilde C|\xi|^k),\quad \mbox{ for all } (\xx,\xi)\in D_\rho^d\times S(\theta,2\delta).$$ This means that 
	$\hat{T}_\aa(\hat{f})$ is $k$--Borel--summable in direction $\theta$ in $\mathcal{E}^{\aa}_\rho$.  Thus, $\hat f$ is $\xx^\a$--$k$--summable in direction $\theta$, what contradicts the assumption.

The second statement in (1) follows from the first one, since Lemma \ref{Bounds for formal gevrey series} (2) implies that if $\hat{f}\in E[[\xx]]_{1/k'}^{\aa'}$, then $\hat{T}_{\aa}(\hat{f})$ is a $\max_{1\leq j\leq d}\{\a_j/\a_j'\}/k'$--Gevrey series in some $\mathcal{E}_r^{\boldsymbol{\a}}$. 
	
(2) Consider $\hat{f}\in E\{\xx\}_{1/k}^{\aa}\cap E\{\xx\}_{1/k'}^{\aa'}$. If $\a_j/\a_j'$ is independent of $j$, write this positive rational number as $a/b$, with $\text{g.c.d.}(a,b)=1$. Then $\a_j=m_j a$, $\a_j'=m_j b$ for some $m_j\in\N^+$, and by applying Corollary \ref{p,q y Mp,Mq}, we have that $E\{\xx\}_{1/k}^{\aa}=E\{\xx\}_{1/ak}^{\boldsymbol{m}}$ and $E\{\xx\}_{1/k'}^{\aa'}=E\{\xx\}_{1/bk'}^{\boldsymbol{m}}$ where $\boldsymbol{m}=(m_1,\dots,m_d)$. Then, this case follows from (1). Furthermore, the cases $\max_{1\leq j\leq d}\{\a_j/\a_j'\}< k'/k$ and $k'/k<\min_{1\leq j\leq d}\{\a_j/\a_j'\}$ also follow from (1).

Finally, the case $\min_{1\leq j\leq d}\, \a_j/\a_j'\leq k'/k\leq \max_{1\leq j\leq d}\, \a_j/\a_j'$ can be reduced to the previous situations by using monomial transformations. To fix ideas, assume that $\a_1/\a_1'\leq \a_2/\a_2'\leq\cdots\leq \a_d/\a_d'$ and $\a_1/\a_1'<k'/k$. If $j_0$ is the smallest index such that $k'/k\leq\a_{j_0}/\a_{j_0}'$, for each $j_0\leq j\leq d$ choose $N_j\in\N^+$  such that $$\frac{k\a_j-k'\a_j'}{k'\a_1'-k\a_1}<N_j.$$ Proposition \ref{monomial sum and blowups} shows that $\hat{f}_1=\hat{f}\circ\pi_{d1}^{N_d}\circ\cdots\circ\pi_{j_01}^{N_{j_0}}\in E\{\xx\}_{1/k}^{\boldsymbol{\a}+\a_1\sum_{j=j_0}^d N_j\boldsymbol{e}_j}\cap E\{\xx\}_{1/k'}^{\aa'+\a_1'\sum_{j=j_0}^d N_j\boldsymbol{e}_j}$, but the new monomials satisfy $\max_{{1\leq i<j_0}\atop{j_0\leq j\leq d}}\{\a_i/\a_i', (N_j\a_1+\a_j)/(N_j\a_1'+\a_j')\}<k'/k$. Thus $\hat{f}_1$ is convergent, and by Lemma \ref{f(xe,e) s-Gevrey in x^pe^(p+q)} (1), so is $\hat{f}$.
\end{proof}

The same idea of proof can be generalized to construct series which are not $\xx^\aa$--$k$--summable for any $\xx^\aa$ or $k>0$. The following theorem is a version of \cite[Thm. 3.9]{CM}, which is incorrect as it is stated there. What is actually proved there, by induction on $n$, is the following.

%so it will be omitted here. 

\begin{teor}\label{series non monomial sumable} Consider $\aa_1,\dots,\aa_n\in (\N^+)^d$ and $k_1,\dots,k_n>0$. For each $j=1,\dots,n$, take a series $\hat{f}_j\in E\{\xx\}^{\aa_j}_{1/k_j}$. If $k_i\aa_i\neq k_j\aa_j$, for all $i\neq j$ and  $\hat{f}_1+\cdots+\hat{f}_n=0$, then $\hat{f}_j\in E\{\xx\}$, for all $j=1,\dots,n$.
\end{teor}

Given $\hat{f}\in\mathcal{O}'_d(E)$, we can also consider the case where $\hat{f}$ is $k$--summable in a monomial in some variables with coefficients which are holomorphic maps in the remaining ones and compare the summability phenomena we have at our disposal. As a matter of fact, in this situation the methods are again incompatible and the proofs can be reduced to Theorem \ref{tauberian general case} using monomial transformations. The key statement is the following.

\begin{prop}\label{summ and blowups}Let  $J\subsetneq [1,d]$ be non-empty, $n=\#J$, and let $\hat{f}\in \mathcal{O}_b(D_R^{d-n},E)\{\xx_J\}^{\aa_J}_{1/k,\theta}$ with  $\xx_J^{\aa_J}$--$k$--sum $f$ in the direction $\theta$. Then $\hat{f}\circ\pi_{ij}$ belongs to $ \mathcal{O}_b(D_{R'}^{d-n},E)\{\xx_J\}^{\aa_J}_{1/k,\theta}$ if $j\not\in J$, and belongs to $\mathcal{O}_b(D_{R'}^{d-n-1},E)\{\xx_{J\cup\{i\}}\}^{\aa'_{J\cup\{i\}}}_{1/k,\theta}$ if $j\in J$, where $\xx_{J\cup\{i\}}^{\aa'_{J\cup\{i\}}}=x_i^{\a_j}\xx_J^{\aa_J}$, for small $0<R'\leq R$. In both cases the corresponding sum is given by $f\circ\pi_{i,j}$.
\end{prop}

\begin{proof}Let $\hat{f}=\sum_{\bb_J\in\N^J} f_{J,\bb_J}(\xx_{J^c})\xx_J^{\bb_J}$ as in equation (\ref{Decomposition f J}), where $f_{J,\bb_J}\in \mathcal{O}_b(D_R^{d-n},E)$ and $\hat{T}_{\aa_J}(\hat{f})=\sum_{n=0}^\infty  f_{\aa_J,n}(\xx) \xx_J^{n\aa_J}$. If $f\sim_{1/k}^{\aa_J}\hat{f}$ on $S_{\aa_J}(\theta,b-a,r)$, $b-a>\pi/k$, for every subsector $S'_{\aa_J}$ of $S_{\aa_J}$  we can find constants $C,A>0$ such that for every $N\in\N$ we have  $$\left\|f(\xx)-\sum_{n=0}^{N-1} f_{\aa_J,n}(\xx)\xx_J^{n\aa_J}\right\|\leq CA^NN!^{\frac{1}{k}}|\xx_J^{N\aa_J}|, \quad\text{ on } \{\xx\in\C^d \,|\, \xx_J\in S'_{\aa_J}, \xx_{J^c}\in D_R^{d-n}\}.$$ Note that $\xx_J^{\aa_J}\circ \pi_{i,j}=\xx_J^{\aa_J}$ if $j\not\in J$ and $\xx_J^{\aa_J}\circ \pi_{i,j}=x_i^{\a_j}\xx_J^{\aa_J}$ if $j\in J$. Then both statements follow with the aid of Proposition \ref{equiv f iff f_n for x^pe^q} by replacing $\xx$ by $\pi_{ij}(\xx)$ in the previous inequality as long as we choose the radii $r,R$ small enough such that $\pi_{ij}(\xx)\in \{\xx\in\C^d \,|\, \xx_J\in S'_{\aa_J}, \xx_{J^c}\in D_R^{d-n}\}$.
\end{proof}

%\begin{nota}If $J=\{j\}$, the previous proposition claims that if  $\hat{f}\in \mathcal{O}_b(D_R^{d-1},E)\{x_j\}_{1/k,\theta}$ then  $\hat{f}\circ\pi_{ij}\in \mathcal{O}_b(D_{R'}^{d-2},E)\{x_i, x_j\}^{(1,1)}_{1/k,\theta}$ for all $i\neq j$ and  $\hat{f}\circ\pi_{il}\in \mathcal{O}_b(D_{R'}^{d-1},E)\{x_j\}_{1/k,\theta}$ if $i,l\neq j$, for small $0<R'\leq R$.\end{nota}

\begin{teor}\label{tauberian for one variable and monomials} Let $I,J\subseteq [1,d]$ be non-empty sets, $n=\#J$, $m=\#I$, and consider $\aa_J\in(\N^+)^m$, $\aa_I'\in(\N^+)^n$, and $k,k'>0$. Then $$\mathcal{O}_b(D_R^{d-n},E)\{\xx_J\}^{\aa_J}_{1/k}\cap \mathcal{O}_b(D_R^{d-m},E)\{\xx_I\}^{\aa_I'}_{1/k'}=E\{\xx\},$$ except in the case $J=I$ and $k\aa_J=k'\aa_I'$, where the spaces coincide. %$\mathcal{O}_b(D_R^{d-n},E)\{\xx_J\}^{\aa_J}_{1/k}= \mathcal{O}_b(D_R^{d-m},E)\{\xx_I\}^{\bb_I}_{1/k'}$.
\end{teor}

\begin{proof}We divide the proof in several cases. First, if $J=I$, the result follows from Theorem \ref{tauberian general case} applied to the space $\mathcal{O}_b(D_R^{d-m},E)$ (resp. $E$ if $m=d$). Second, assume $J\subsetneq I$. Changing the order of coordinates if necessary, we can assume $J=\{1,\dots,n\}$ and $I=\{1,\dots,m\}$. Then $\hat{f}_1=\hat{f}\circ \pi_{m,1}\circ\cdots\circ \pi_{n+1,1}$ is $\xx_J^{\aa_J}(x_{n+1}\cdots x_m)^{\a_1}$--$k$--summable and $\xx_I^{\aa_I'}(x_{n+1}\cdots x_m)^{\beta_1}$--$k'$--summable, both with coefficients in $\mathcal{O}_b(D_R^{d-m},E)$ (resp. $E$ if $m=d$). We can apply the previous case to conclude that $\hat{f}_1$ and thus $\hat{f}$ are convergent since  $\a_1/(\a_m'+\a_1')=\a_1/\a_1'$ implies $\a_m'=0$, which is not the case. Finally, we can assume there are $j_0\in J\setminus I$ and $i_0\in I\setminus J$, and we consider the series $$\hat{f}_2=\hat{f}\circ \textstyle^\circ\prod_{i\in I\setminus J}\pi_{ij_0} \circ ^\circ\prod_{j\in J\setminus I} \pi_{ji_0}.$$ Here $^\circ\prod$ denotes the composition product, that in this case is independent of the order because the monomial transformations involved commute since $j_0\neq i_0$.  Then $\hat{f}_2$ is $\xx_I^{\aa_I'}\left(\prod_{j\in J\setminus I} x_j \right)^{\a_{i_0}'}$--$k'$--summable and $\xx_J^{\aa_J}\left(\prod_{i\in I\setminus J} x_i \cdot \prod_{j\in J\setminus I} x_j \right)^{\a_{j_0}}$ --$k$--summable with coefficients in $\mathcal{O}_b(D_R^{d-\#(I\cup J)},E)$ (resp. $E$ if $J\cup I=[1,d]$). Since the $i_0$-components (resp. $j_0$-components) of these monomials are $\a_{i_0}'$ and $\a_{j_0}$ (resp. $\a_{i_0}'$ and $2\a_{j_0}$) respectively, then $\a_{j_0}/\a_{i_0}'\neq 2\a_{j_0}/\a_{i_0}'$ and therefore $\hat{f}_2$ and $\hat{f}$ are convergent as we wanted to prove.
\end{proof}

Having this result at hand it is possible to formulate and prove a statement similar to the one in Theorem \ref{series non monomial sumable}. This provides more examples of series that cannot be summed with the methods we have studied along this work. In fact, this has been done recently in \cite{CMS} in the more general setting of $k$--summability in analytic germs and we refer the reader to this work for a complete proof of these facts.

\section{Monomial summability of a family of singular perturbed PDEs}\label{Monomial summability of a family of singular perturbed PDEs}

Summability in a monomial is useful to study formal solutions of doubly singular equations, i.e., singularly perturbed ordinary differential equations of the form  $$\varepsilon^q x^{p+1}\frac{\d \yy}{\d x}=\boldsymbol{F}(x,\varepsilon,\yy),$$ where $p,q\in\N^+$ and  $\boldsymbol{F}$ is a $\C^N$--valued holomorphic map in some neighborhood of $(0,0,\00)\in \C\times\C\times \C^N$. If $\frac{\d\boldsymbol{F}}{\d \yy}(0,0,\00)$ is an invertible matrix, this system has a unique formal power series solution and it is $x^p\varepsilon^q$--$1$--summable, see \cite[Thm. 5.2]{Monomial summ}. The technique employed to prove this result in the case $p=q=1$ is to apply the change of variables $t=x\varepsilon$, to obtain an equation involving $t$ and $\varepsilon$. The new equation is then solved on large sectors and the differences of such solutions, in their common domains, are studied in order to apply Ramis--Sibuya's theorem. The general case follows after rank reduction.

The goal of this section is to generalize the previous result by establishing the $\boldsymbol{x}^{\boldsymbol{\a}}\boldsymbol{\varepsilon}^{\boldsymbol{\a'}}$--$1$--summability  of the unique formal power series solution of the singularly perturbed partial differential equation (\ref{PDEG}) explained in the Introduction. By hypothesis, the $n$--tuple $\boldsymbol{\mu}=(\mu_1,\dots,\mu_n)$ has entries, up to a non-zero multiple scalar, positive real numbers. Then, dividing equation (\ref{PDEG}) by $\left<\boldsymbol{\mu},\aa\right>$, it is sufficient to study the singularly perturbed partial differential equation  \begin{equation}\label{SPPDE1}\boldsymbol{\varepsilon}^{\boldsymbol{\a'}}X_{\boldsymbol{\lambda}}(\yy)(\xx,\ee)= \boldsymbol{\varepsilon}^{\boldsymbol{\a'}}\xx^{\aa}\left( \frac{s_1}{\a_1} x_1 \frac{\d \boldsymbol{y}}{\d x_1}+\cdots+ \frac{s_n}{\a_n} x_n \frac{\d \boldsymbol{y}}{\d x_n}\right)=\boldsymbol{F}(\boldsymbol{x},\boldsymbol{\varepsilon},\boldsymbol{y}),\end{equation} where $\boldsymbol{x}=(x_1,\dots,x_n), \boldsymbol{\varepsilon}=(\varepsilon_1,\dots,\varepsilon_m)$  are tuples of complex variables, $\boldsymbol{\alpha}=(\alpha_1,\dots,\alpha_n)$,  $\boldsymbol{\a'}=(\a_1',\dots,\a_m')$ are tuples of positive integers, $\boldsymbol{\lambda}=\left(\frac{s_1}{\a_1},\dots,\frac{s_n}{\a_n}\right)$ where the   $s_j/\a_j=\mu_j/\left<\boldsymbol{\mu},\aa\right>>0$ satisfy $s_1+\cdots+s_n=1$,   $\boldsymbol{F}=\left<\boldsymbol{\mu},\aa\right>^{-1}\boldsymbol{G}$ is a $\C^N$--valued holomorphic map in some neighborhood of $(\00,\00,\00)\in\C^n\times\C^m\times \C^N$ and $A_0:=\frac{\partial \boldsymbol{F}}{\partial \yy}(\00,\00,\boldsymbol{0})$ is an invertible matrix.

We will apply directly the Borel-Laplace analysis developed in Section \ref{Borel-Laplace analysis for monomial summability}, based on the methods of one variable, see e.g., \cite{Costin2}. The existence of the unique formal solution is a straightforward result. To determine the Gevrey type of this solution we can use a variant of Nagumo norms, as the ones used in \cite{CDRSS}. For the summability, we will study the convolution equation obtained from (\ref{SPPDE1}) after applying the adequate Borel transformation. After introducing suitable spaces of analytic functions and norms, we will solve the convolution equation using the Banach fixed-point theorem.

\begin{teor}\label{Main result sum mon for PDE}Consider the singularly perturbed partial differential equation (\ref{PDEG}). If $\boldsymbol{G}(\00,\00,\00)=\00$ and $B_0=\frac{\partial \boldsymbol{G}}{\partial \yy}(\00,\00,\boldsymbol{0})$ is an invertible matrix, then (\ref{PDEG}) has a unique formal power series solution $\widehat{\yy}\in\C[[\xx,\boldsymbol{\varepsilon}]]^N$ and it is $\xx^{\boldsymbol{\a}}\boldsymbol{\varepsilon}^{\boldsymbol{\a'}}$--$1$--summable, with possible singular directions  determined by the equation $$\det\left(\left<\boldsymbol{\mu},\aa\right>\xxi^{\aa}\boldsymbol{\eta}^{\aa'}I_N-B_0\right)=0,$$ in the $(\xxi,\boldsymbol{\eta})$--Borel space. Here $I_N$ denotes  the identity matrix of size $N$.

\end{teor}

\begin{proof}We divide the proof in four main steps: existence and Gevrey type of the formal solution, establishment of the associated convolution equation, introduction of adequate Banach algebras and their properties, and finally, the application of the fixed point theorem.
	
\textbf{1. The formal solution.} We will consider the norms $|\yy|=\max_{1\leq j\leq N} |y_j|$ on $\C^N$ and $|A|=\max_{1\leq i\leq N} \sum_{j=1}^N |A_{ij}|$ on $\text{Mat}(N\times N,\C)$, where the notation should be clear from context. 

Let us write $$\boldsymbol{F}(\xx,\ee,\yy)=c(\xx,\ee)+A(\xx,\ee)\yy+\sum_{|\boldsymbol{I}|\geq 2}  A_{\boldsymbol{I}}(\xx,\ee)\yy^{\boldsymbol{I}},$$ as a power series in $\yy$, where $\boldsymbol{c}, A_{I}\in\mathcal{O}_b(D_r^n\times D_R^m,\C^N)$, $A \in\mathcal{O}_b(D_r^n\times D_R^m,\text{Mat}(N\times N,\C))$, for all $I=(i_1,\dots,i_N)\in\N^N$ and $\yy^{\II}=y_1^{i_1}\cdots y_N^{i_N}$. We also write $c(\xx,\ee)=\sum_{\bb\in\N^n} c_\bb(\ee) \xx^\bb=\sum_{\bb'\in\N^m} c_{\bb'}(\xx) \ee^{\bb'} $ as a convergent power series in $\xx$ (resp. in $\ee$) with coefficients $c_\bb\in \mathcal{O}_b(D^m_R,\C^N)$ (resp. $c_{\bb'}\in \mathcal{O}_b(D^n_r,\C^N)$). Since $\boldsymbol{F}$ is holomorphic we can find constants $K,\delta>0$ such that \begin{equation}\label{Gr AI}
\|A_\II\|_{r,R}:=\sup_{(\xx,\ee)\in D_r^n\times D_R^m}|A_{\boldsymbol{I}}(\xx,\ee)|\leq K\delta^{|I|}.
\end{equation} The notation $\|\cdot\|_{r,R}$ will be also used for matrix valued functions.

To find the formal solution set $\widehat{\yy}=\sum_{\bb\in\N^n} \yy_\bb(\ee) \xx^\bb$. Since $\boldsymbol{F}(\00,\00,\00)=0$ and $A_0$ is invertible we can apply the implicit function theorem to find a unique $\yy_\00(\ee)\in \mathcal{O}_b(D^m_R,\C^N)$, $\yy_\00(\00)=\00$ (with  $R$ small enough) solving the equation $\boldsymbol{F}(\00,\ee,\yy_\00(\ee))=0$. To determine the higher order terms we use the recurrence $$\left<\boldsymbol{\lambda},\bb-\aa\right> \ee^{\aa'}\yy_{\bb-\aa}(\ee)=c_\bb(\ee)+A(\00,\ee)\yy_\bb(\ee)+\text{known terms},$$ obtained by inserting $\widehat{\yy}$ in the differential equation and equating the coefficient of $\xx^\bb$. Since $A(\00,\00)=A_0$ is invertible, we may assume (by reducing $R$ again if necessary) that $A(\00,\ee)$ is invertible, for all $\ee\in D^m_R$. Thus this recurrence determines $\yy_\bb$ uniquely, and the uniqueness of $\widehat{\yy}$ follows. 

Similar computations as for the classical irregular singularities for ODEs show that $\widehat{\yy}$ is a $1$--Gevrey series in $\xx^\aa$. To determine the Gevrey order in $\ee$, we use the following variant of the Nagumo norms for higher dimensions: if $f\in\mathcal{O}(D_r^n\times D_R^m)$ and $l\in\N$, we define \begin{equation*}\label{Nagumo n} \|f\|_l:=\sup_{(\xx,\ee)\in D_r^n\times D_R^m} |f(\xx,\ee)|(r-|x_1|)^l\cdots(r-|x_n|)^l.
	\end{equation*} They satisfy the majorant inequalities $$\|f+g\|_l\leq \|f\|_l+\|g\|_l,\quad  \|fg\|_{l+k}\leq \|f\|_l\|g\|_k,\quad \left\|\frac{\d f}{\d x_j}\right\|_{l+1}\leq e(l+1)r^{n-1}\|f\|_l,$$ for all $l,k\in\N,  j=1,\dots,n$. The proof of the last inequality can be done in the same way as it is proved in \cite{CDRSS} for the Nagumo norms introduced there. Then, applying the usual majorant series technique to the recurrence PDEs $$\xx^{\aa}\sum_{j=1}^n \frac{s_j}{\a_j} x_j \frac{\d}{\d x_j}( \yy_{\bb'-\aa'})=c_{\bb'}(\xx)+A(\xx,\00)\yy_{\bb'}(\xx)+\text{known terms},$$ obtained by replacing $\widehat{\yy}=\sum_{\bb'\in\N^m} \yy_{\bb'}(\xx) \ee^{\bb'}$ into (\ref{SPPDE1}), we may conclude that $\widehat{\yy}$ is also $1$--Gevrey in $\ee^{\aa'}$. Here we have also reduced $r$ (if necessary) to guarantee that $A(\xx,\00)$ is invertible, for all $\xx\in D^n_r$. 

\textbf{2. The convolution equation.} To simplify notation we will write $\hat{\mathcal{B}}=\hat{\mathcal{B}}_{(\boldsymbol{\lambda},\00)}$ and $\ast=\ast_{(\boldsymbol{\lambda},\00)}$, where  $(\boldsymbol{\lambda},\00)\in\R_{>0}^n\times\R^m$. Also set $\boldsymbol{s}=(s_1,\dots,s_n)$ which by hypothesis belongs to $\sigma_n$.

Applying $\mathcal{B}$ to (\ref{SPPDE1}) we obtain a convolution equation, that written as a fixed point equation, it is given by \begin{align}\label{CE}
	(\xxi^\aa\boldsymbol{\eta}^{\aa'}I_N-A_0)\boldsymbol{Y}=&\mathcal{B}(c)+(\mathcal{B}(A-A(0,\boldsymbol{\eta})))\ast \boldsymbol{Y}+(A(\00,\boldsymbol{\eta})-A_0) \boldsymbol{Y}+\sum_{|\boldsymbol{I}|\geq 2} \mathcal{B}(A_{\boldsymbol{I}}-A_{\II}(\00,\boldsymbol{\eta}))\ast \boldsymbol{Y}^{\ast \boldsymbol{I}} \nonumber\\
	&+\sum_{|\boldsymbol{I}|\geq 2}  A_{\II}(\00,\boldsymbol{\eta})\boldsymbol{Y}^{\ast \boldsymbol{I}},
\end{align} in the $(\xxi,\boldsymbol{\eta})$--Borel plane. Here we write $\YY^{\ast\II}=Y_1^{\ast i_1}\ast\cdots\ast Y_N^{\ast i_N}$, and $Y_j^{\ast i_j}=Y_j\ast\cdots \ast Y_j$, $i_j$ times. 

Note that under the holomorphic change of variables $\boldsymbol{w}=\yy-\sum_{\aa\not<\bb} \yy_{\bb}(\ee)\xx^{\bb}$, we may assume that $c(\xx,\ee)=\sum_{\bb>\aa} c_\bb(\ee) \xx^\bb$, $\mathcal{B}(c)$ is holomorphic at the origin and $\mathcal{B}(c)(\00,\boldsymbol{\eta})=\00$, and so we will do it from now on.

Equation (\ref{CE}) has a unique analytic solution at the origin given by $\YY=\mathcal{B}(\widehat{\yy})$. To solve (\ref{CE}) in larger domains, we ask $\xxi^\aa\boldsymbol{\eta}^{\aa'}I_N-A_0$ to be invertible. Let $\nu_1=|\nu_1|e^{i\theta_1},\dots,\nu_N=|\nu_N|e^{i\theta_N}$ be the eigenvalues of $A_0$ repeated according to their multiplicity, all different from zero by hypothesis. We will work on domains contained in $$\Omega:=\{(\xxi,\boldsymbol{\eta})\in\C^n\times\C^m \,|\, \xxi^\aa\boldsymbol{\eta}^{\aa'}\neq \nu_j, \text{ for all } j=1,\dots,N\}.$$

\textbf{3. Some focusing spaces.} Consider  $$S:=S_{r,R'}=S_{(\aa,\aa')}^{(\boldsymbol{s},\00)}(\theta,2\epsilon,R')\cup (D_r^n\times \overline{D}_{R'^{1/\a_1'}}\times\cdots\times \overline{D}_{R'^{1/\a_m'}}),$$ where $\theta\neq \theta_j$, for all $j=1,\dots,N$, $\epsilon>0$ is chosen small such that $S\subset \Omega$, and $R'<R^{\a_j'}$, for all $j=1,\dots,m$, to ensure that $S\subset \C^n\times D_R^m$. Note that $S$ contains a polydisc around the origin.

If $\mu>0$, we introduce the spaces of holomorphic maps $$
\mathcal{A}^N_{\mu}(S):=\{\boldsymbol{f}=(f_1,\dots,f_N)\in \mathcal{O}(S,\C^N) \,|\, \boldsymbol{f}(\00,\cdot)=\00,\, \|\boldsymbol{f}\|_{N,\mu}:=\max_{1\leq j\leq N}\|f_j\|_{\mu}<+\infty \},$$
$$\|f\|_\mu:=\|f\|_{1,\mu}=M_0 \sup_{(\xxi,\boldsymbol{\eta})\in  S} |f(\xxi,\boldsymbol{\eta})|(1+R(\xxi)^2)e^{-\mu R(\xxi)},\quad f\in\mathcal{O}(S).$$ Here $R(\xxi)=\max_{1\leq j\leq n}\,|\xi_j|^{\a_j/s_j}$ and $M_0=\sup_{s>0} s(1+s^2)I(s)\approx 3.76$, where $$I(s):=\int_0^1 \frac{d\tau}{(1+s^2\tau^2 )(1+s^2(1-\tau)^2)}=\frac{2(\ln(1+s^2)+s\arctan(s))}{s^2(4+s^2)}.$$ This family of norms is an adaptation for monomials of the norms introduced in \cite[Def. 4.1]{Costin2} for one variable, useful to treat non-linear partial PDEs, see e.g., \cite{Zhang}. We refer to \cite{Costin1} for similar norms in higher dimensions. The properties we need are described in the following three technical lemmas.

\begin{lema}\label{Exp and convolucion}Let $\mu>0$, $R$ and $S$ as before. The following statements hold:\begin{enumerate}
\item  If $Q\in\mathcal{O}_b(D_R^m,\C^N)$,  then $$\|f Q\|_{N,\mu}\leq \|Q\|_{r,R}\|f\|_\mu, \quad f\in \mathcal{A}^1_{\mu}(S).$$

\item $(\mathcal{A}^1_{\mu}(S),\ast,\|\cdot\|_{\mu})$ is a Banach algebra. More precisely, if $f,g\in\mathcal{A}^1_{\mu}(S)$, then $f\ast g\in\mathcal{A}^1_{\mu}(S)$ and $$\|f\ast g\|_\mu\leq \|f\|_\mu \|g\|_\mu.$$

\item If $0<\mu_0<\mu$ and $\boldsymbol{f}\in\mathcal{A}^N_{\mu_0}(S)$, then $\boldsymbol{f}\in\mathcal{A}^N_{\mu}(S)$ and $\|\boldsymbol{f}\|_{N,\mu}\rightarrow 0$ as $\mu\rightarrow+\infty$. In other words, $\mathcal{A}^N(S):=\bigcup_{\mu>0} \mathcal{A}^N_{\mu}(S)$ is a focusing space, see e.g., \cite[p. 14]{Costin}.
\end{enumerate}
\end{lema}

\begin{proof}Inequality in (1) is an immediate consequence of the definition. To prove (2) note that $f\ast g$ is clearly analytic on $S$ as long as $f$ and $g$ are analytic there. To establish the desired bound, we use that $R(\xi_1 \tau^{\a_1/s_1},\dots, \xi_d \tau^{\a_n/s_n})$ $=\tau R(\xxi)$, for all $\tau>0$, $|\xxi^{\aa}|\leq R(\xxi)$ (second inequality in (\ref{prop the segment})) and also the definition of $M_0$ to get \begin{align*}
	|(f\ast g)(\xxi,\boldsymbol{\eta})|\leq |\xxi^{\aa}|\frac{e^{\mu R(\xxi)}}{M_0^2}\|f\|_\mu \|g\|_\mu I(R(\xxi))\leq \frac{e^{\mu R(\xxi)}}{M_0(1+R(\xxi)^2)}\|f\|_\mu \|g\|_\mu.
	\end{align*} To prove (3), it is sufficient to do it for $N=1$. If $f\in\mathcal{A}^1_{\mu_0}(S)$ and $\mu_0<\mu$, then it follows from the definition that $\|f\|_{\mu}\leq \|f\|_{\mu_0}$. To show that $\|f\|_{\mu}\rightarrow0$ as $\mu\rightarrow+\infty$, let $\epsilon>0$. We can find $\rho>0$ small enough such that $|(1+R(\xxi)^2)f(\xxi,\boldsymbol{\eta})|\leq \epsilon/2M_0$ if $(\xxi,\boldsymbol{\eta})\in \overline{D}$, $D=D_\rho^n\times D_{R'^{1/\a_1'}}\times\cdots\times D_{R'^{1/\a_m'}}$,  since $f(\00,\boldsymbol{\eta})=0$. Then $$M_0 \sup_{(\xxi,\boldsymbol{\eta})\in D} |f(\xxi,\boldsymbol{\eta})|(1+R(\xxi)^2)e^{-\mu R(\xxi)}\leq \frac{\epsilon}{2}.$$ 	If $(\xxi,\boldsymbol{\eta})\in S\setminus D$, then $R(\xxi)\geq \rho'=\min_{1\leq j\leq n} \rho^{s_j/\a_j}>0$ and $$M_0 \sup_{(\xxi,\boldsymbol{\eta})\in  S\setminus D} |f(\xxi,\boldsymbol{\eta})|(1+R(\xxi)^2)e^{-\mu R(\xxi)}\leq e^{-\rho'(\mu-\mu_0)} \|f\|_{\mu_0}.$$ Taking a large $\mu$ such that $e^{-\rho'(\mu-\mu_0)} \|f\|_{\mu_0}<\epsilon/2$, we obtain  $\|f\|_{\mu}<\epsilon$. This proves the claim.	
\end{proof}

\begin{lema}\label{Borel and convolution}Let $P\in \mathcal{O}_b(D_r^{n}\times D_R^m,\C^N)$ be a map such that $P(\00,\ee)=\00$, for all $\ee\in D_R^m$. Then for any $0<\rho<r$ we have $$\|f\ast \mathcal{B}(P)\|_{N,\mu}\leq
C_{\mu,\rho}\|P\|_{\rho,R}\|f\|_{\mu},\quad f\in \mathcal{A}^1_{\mu}(S),$$ where $a:=\min_{1\leq j\leq n}\,s_j/\a_j$, $C_{\mu,\rho}:=3(\left(1-2/\mu^{a}\rho\right)^{-n}-1)$  and  $\mu>\max\{4\sqrt{2},(2/\rho)^{1/a}\}$. The same inequality is valid for $P\in \mathcal{O}_b(D_r^{n}\times D_R^m,\text{Mat}(N\times N,\C))$ and $\boldsymbol{f}\in \mathcal{A}^N_{\mu}(S)$.
\end{lema}

\begin{proof}Let us write $P(\xx,\ee)=\sum_{\bb\in\N^n\setminus\{\00\}} P_{\bb}(\ee) \xx^{\bb}$ as a convergent power series at the origin with coefficients in $\mathcal{O}_b(D_R^m,\C^N)$ and also  $\mathcal{B}(P)(\xxi,\boldsymbol{\eta})=\sum_{\bb\in\N^n\setminus\{\00\}} \frac{P_{\bb}(\boldsymbol{\eta})}{\Gamma(\left<\bb,\boldsymbol{\lambda}\right>)} \xxi^{\bb-\aa}$.
	
If $\bb\neq\00$, we have that \begin{align*}
\left|\xxi^{\bb-\aa}\ast f(\xxi,\boldsymbol{\eta})\right|\leq |\xxi^\bb|\frac{e^{\mu R(\xxi)}}{M_0}\|f\|_{\mu} \int_0^1   \frac{ \tau^{\left<\bb,\boldsymbol{\lambda}\right>-1}e^{-\mu R(\xxi)\tau}}{1+R(\xxi)^2(1-\tau)^2}d\tau.
\end{align*}To properly bound this integral expression we split it from $0$ to $1/2$ and from $1/2$ to $1$. In the first case, the integral is bounded by $$\int_{0}^{1/2}\frac{ \tau^{\left<\bb,\boldsymbol{\lambda}\right>-1}e^{-\mu R(\xxi)\tau}}{1+R(\xxi)^2/4}d\tau \leq  \frac{4}{1+R(\xxi)^2} \frac{\Gamma(\left<\bb,\boldsymbol{\lambda}\right>)}{(\mu R(\xxi))^{\left<\bb,\boldsymbol{\lambda}\right>}},$$ where we have used the integral representation of the Gamma function, which is possible since $\left<\bb,\boldsymbol{\lambda}\right>>0$. In the second case, the integral is bounded by \begin{align*}
\int_{1/2}^{1} \tau^{\left<\bb,\boldsymbol{\lambda}\right>-1}e^{-\mu R(\xxi)\tau}d\tau &\leq  e^{-\mu R(\xxi)/4} \int_{1/2}^{1} \tau^{\left<\bb,\boldsymbol{\lambda}\right>-1} e^{-\mu R(\xxi)\tau/2}d\tau\\
&\leq \left(\frac{2}{\mu R(\xxi)}\right)^{\left<\bb,\boldsymbol{\lambda}\right>} \Gamma(\left<\bb,\boldsymbol{\lambda}\right>)e^{-\mu R(\xxi)/4}\\
&\leq \left(\frac{2}{\mu R(\xxi)}\right)^{\left<\bb,\boldsymbol{\lambda}\right>} \frac{\Gamma(\left<\bb,\boldsymbol{\lambda}\right>)}{1+\mu^2/32 R(\xxi)^2}\leq \left(\frac{2}{\mu R(\xxi)}\right)^{\left<\bb,\boldsymbol{\lambda}\right>} \frac{\Gamma(\left<\bb,\boldsymbol{\lambda}\right>)}{1+R(\xxi)^2},
\end{align*} as long as $\mu^2/32>1$. Now, if $\mu>1$, by using the definition of $a$, $a|\bb|\leq \left<\bb,\boldsymbol{\lambda}\right>\leq |\bb|$, that  $|\xxi^\bb|\leq R(\xxi)^{\left<\bb,\boldsymbol{\lambda}\right>}$, and $4+2^{|\bb|}\leq 3\cdot 2^{|\bb|}$, for $|\bb|\geq1$, we can conclude that $$\left|\xxi^{\bb-\aa}\ast f(\xxi,\boldsymbol{\eta})\right|\leq  \frac{3\cdot 2^{|\bb|}\Gamma(\left<\bb,\boldsymbol{\lambda}\right>)}{\mu^{a|\bb|}}\frac{e^{\mu R(\xxi)}}{M_0(1+R(\xxi)^2)}\|f\|_{\mu}.$$

By Cauchy's inequalities, if $0<\rho<r$, then $|P_\bb(\boldsymbol{\eta})|\leq \rho^{-|\bb|} \|P\|_{\rho,R}$ for any $\boldsymbol{\eta}\in D_R^m$. Therefore
\begin{align*}
|f\ast \mathcal{B}(P)(\xxi,\boldsymbol{\eta})|&\leq \left(\sum_{\bb\neq \00} \frac{3\cdot 2^{|\bb|}\|P\|_{\rho,R}}{(\mu^a\rho)^{|\bb|}} \right) \frac{e^{\mu R(\xxi)}}{M_0(1+R(\xxi)^2)}\|f\|_{\mu}\\
&=\left(\left(1-2/\mu^{a}\rho\right)^{-n}-1\right)  \frac{3\|P\|_{\rho,R}e^{\mu R(\xxi)}}{M_0(1+R(\xxi)^2)}\|f\|_{\mu},
\end{align*} as long as $\mu^a>2/\rho$. This estimate allows us to conclude the proof.
\end{proof}

\begin{lema}\label{Diference and convolution} For all $N\in\N^+$, $\II=(i_1,\dots,i_N)\in\N^N$ and $\YY, \boldsymbol{h}\in\mathcal{A}^N_{\mu}(S)$, we have
$$\|(\YY+\boldsymbol{h})^{\ast \II}-\YY^{\ast \II}\|_{\mu}\leq |\II|  (\|\YY\|_{N,\mu}+\|\boldsymbol{h}\|_{N,\mu})^{|\II|-1} \|\boldsymbol{h}\|_{N,\mu}.$$
\end{lema}

The proof can be done by induction on $\II$, see e.g.,  \cite[p.19]{Costin}.

\textbf{4. The application of the fixed point theorem.} Let $M=M(S_{r,R'})>0$ such that $|(\xxi^\aa\boldsymbol{\eta}^{\aa'}I_N-A_0)^{-1}|< M$ on $S_{r,R'}$. Fix $0<\rho<r$ and using the continuity of $A(\00,\cdot)$ take  $0<\rho'\leq R'$ small enough such that $\|A(\00,\cdot)-A_0\|_{\rho,\rho'}<1/4M$.

We consider the operator $\mathcal{H}$ given by solving $\YY$ in equation (\ref{CE}). Remember that $K,\delta>0$ in (\ref{Gr AI}) are determined by $\boldsymbol{F}$ and are fixed. Let $\mu_0>0$ be such that $\mathcal{B}(c)\in \mathcal{A}^N_{\mu_0}(S_{\rho,\rho'})$ (it exists by Step 2.). If $\mu\geq \max\{\mu_0,4\sqrt{2},(2/\rho)^{1/a}\}$ and $\|\YY\|_{N,\mu}<1/\delta$, it follows using Lemmas \ref{Exp and convolucion}, \ref{Borel and convolution} and the fact that $\|\YY^{\ast\II}\|_{\mu}\leq \|\YY\|_{N,\mu}^{|\II|}$ (Lemma \ref{Exp and convolucion}, (2)) that 
\begin{align*}
M^{-1}\|\mathcal{H}(\boldsymbol{Y})\|_{N,\mu}\leq& \|\mathcal{B}(c)\|_{N,\mu}+\left(C_{\mu,\rho}\|A-A(\00,\cdot)\|_{\rho,\rho'} +\|A(\00,\cdot)-A_0\|_{\rho,\rho'}\right)\|\boldsymbol{Y}\|_{N,\mu}\\
&+ K\left(2C_{\mu,\rho}+1\right)\left(\left(1-\delta \|\boldsymbol{Y}\|_{N,\mu}\right)^{-N}-1\right).
\end{align*} We may conclude that $\mathcal{H}$ maps $\bigcup_{\mu\geq \mu_0} B_{\mu}(1/2\delta)$ to $\mathcal{A}^N(S_{\rho,\rho'})$, where $B_{\mu}(1/2\delta)$ is the ball $$B_{\mu}(1/2\delta)=\{ \YY\in \mathcal{A}^N_{\mu}(S_{\rho,\rho'}) \,|\, \|\YY\|_{N,\mu}\leq 1/2\delta \}.$$

Now choose $0<\epsilon<1/\delta$ such that $\left(1-\delta \epsilon\right)^{-N-1}-1<(2K\delta NM)^{-1}$. Using Lemma \ref{Diference and convolution} we can conclude that if $\|\YY\|_{N,\mu}+\|\boldsymbol{h}\|_{N,\mu}\leq \epsilon$, then \begin{align*}
M^{-1}\|\mathcal{H}(\YY+\boldsymbol{h})-\mathcal{H}(\YY)\|_{N,\mu}\leq&
\left(C_{\mu,\rho}\|A-A(\00,\cdot)\|_{\rho,\rho'} +\|A(\00,\cdot)-A_0\|_{\rho,\rho'}\right)\|\boldsymbol{h}\|_{N,\mu}\\
&K\delta N(2C_{\mu,\rho}+1)\left(\left(1-\delta (\|\boldsymbol{Y}\|_{N,\mu}+\|\boldsymbol{h}\|_{N,\mu})\right)^{-N-1}-1\right) \|\boldsymbol{h}\|_{N,\mu},\\
\leq&\left(C_{\mu,\rho}\|A-A(\00,\cdot)\|_{\rho,\rho'} +(4M)^{-1}+C_{\mu,\rho}M^{-1}+(2M)^{-1}\right)\|\boldsymbol{h}\|_{N,\mu}
\end{align*} 

\noindent where we have used the identity $\sum_{|\II|\geq 2} |\II|\tau^{|\II|-1}=N\left((1-\tau)^{-N-1}-1\right)$, valid for all $|\tau|<1$. In conclusion, we have obtained the inequality $$\|\mathcal{H}(\YY+\boldsymbol{h})-\mathcal{H}(\YY)\|_{N,\mu}\leq \left(C_{\mu,\rho}M\|A-A(\00,\cdot)\|_{\rho,\rho'} +C_{\mu,\rho}+\frac{3}{4}\right)\|\boldsymbol{h}\|_{N,\mu}.$$ Since $\rho$ has been fixed and $C_{\mu,\rho}\rightarrow 0$ as $\mu\rightarrow+\infty$, taking $\mu\geq \mu_0$ large enough, we conclude that $\mathcal{H}$ is eventually a contraction, say 
$$\|\mathcal{H}(\YY+\boldsymbol{h})-\mathcal{H}(\YY)\|_{N,\mu}\leq \frac{7}{8}\|\boldsymbol{h}\|_{N,\mu}.$$ If we also take $\mu$ large such that $\|\mathcal{B}(c)\|_{N,\mu}<\epsilon/8M$ (Lemma \ref{Exp and convolucion}, (3)), then $\|\mathcal{H}(\00)\|_{N,\mu}<\epsilon/8$ and the previous inequality shows that $\mathcal{H}$ maps the ball $B_{\mu}(\epsilon)$ to itself and being a contraction, it has a unique fixed point $\YY_0\in \mathcal{O}(S,\C^N)$.

Since $S$ contains a neighborhood of the origin and equation (\ref{CE}) has $\mathcal{B}(\widehat{\yy})$ as unique analytic solution there, then it coincides with the Taylor series expansion of $\YY_0$ at the origin. This means that $\mathcal{B}(\widehat{\yy})$ can be analytically continued to $S$ with exponential growth of order at most $(\a_1/s_1,\dots,\a_n/s_n,\00)$ there. Since this can be done for all $\theta$, up to $\theta_1,\dots,\theta_N$, we conclude that $\widehat{\yy}$ is $\xx^\aa\ee^{\aa'}$--$1$--summable. Thus the possible singular directions of $\widehat{\yy}$ for $\xx^\aa\ee^{\aa'}$--$1$--summability are determined by the equation $\det\big(\xxi^{\aa}\boldsymbol{\eta}^{\aa'}I_N-A_0\big)=0$.
\end{proof}

\begin{coro}Assuming the same hypotheses of the previous theorem, consider the vector field $X=\xx^{\aa}\left( \mu_1 x_1 \frac{\d}{\d x_1}+\cdots+ \mu_n x_n \frac{\d}{\d x_n}\right)$. If $b_1,\dots,b_{l-1}\in\C$, the system of singularly perturbed PDEs $$\ee^{l\aa'} X^l(\yy)+b_{l-1}\ee^{(l-1)\aa'} X^{l-1}(\yy)+\cdots+b_1 \ee^{\aa'} X(\yy)=\boldsymbol{G}(\boldsymbol{x},\boldsymbol{\varepsilon},\boldsymbol{y}),$$ has a unique formal solution $\widehat{\yy}\in\C[[\xx,\ee]]^N$ and it is $\xx^\aa\ee^{\aa'}$--$1$--summable.
\end{coro}	
	
%Let $a_1,\dots,a_{l-1}$ be complex numbers. Under the same hypotheses of the previous theorem, the system of PDEs $$\ee^{l\aa'} X_{\boldsymbol{\lambda}}^l(\yy)+a_{l-1}\ee^{(l-1)\aa'} X_{\boldsymbol{\lambda}}^{l-1}(\yy)+\cdots+a_1 \ee^{\aa'} X_{\boldsymbol{\lambda}}(\yy)=\boldsymbol{F}(\boldsymbol{x},\boldsymbol{\varepsilon},\boldsymbol{y}),$$ has a unique formal solution $\widehat{\yy}\in\C[[\xx,\ee]]^N$ and it is $\xx^\aa\ee^{\aa'}$--$1$--summable.

\begin{proof}Dividing the given PDE by $\left<\boldsymbol{\mu},\aa\right>^l$, it is enough to prove the statement for the equation $$\ee^{l\aa'} X_{\boldsymbol{\lambda}}^l(\yy)+a_{l-1}\ee^{(l-1)\aa'} X_{\boldsymbol{\lambda}}^{l-1}(\yy)+\cdots+a_1 \ee^{\aa'} X_{\boldsymbol{\lambda}}(\yy)=\boldsymbol{F}(\boldsymbol{x},\boldsymbol{\varepsilon},\boldsymbol{y}),$$ where $a_j=\left<\boldsymbol{\mu},\aa\right>^{j-l} b_j$, $\boldsymbol{F}=\left<\boldsymbol{\mu},\aa\right>^{-l} \boldsymbol{G}$, and  $\boldsymbol{\lambda}=\left(\frac{s_1}{\a_1},\dots,\frac{s_n}{\a_n}\right)$ and $s_j/\a_j=\mu_j/\left<\boldsymbol{\mu},\aa\right>>0$ are as before. If we put $(\yy_0,\yy_1\dots,\yy_{l-1})=(\yy,\ee^{\aa'} X_{\boldsymbol{\lambda}}(\yy),\dots,\ee^{(l-1)\aa'} X_{\boldsymbol{\lambda}}^{l-1}(\yy))$, the result now follows by applying Theorem \ref{Main result sum mon for PDE} to the corresponding system of PDEs of size $lN$ given by $$\ee^{\aa'} X_{\boldsymbol{\lambda}}(\yy_0,\yy_1\dots,\yy_{l-1})=(\yy_1,\dots,\yy_{l-1},\boldsymbol{F}(\xx,\ee,\yy_0)-a_1\yy_0-\cdots-a_{l-1}\yy_{l-1}).$$
	
In this case, the possible singular directions in $t=\xxi^{\aa}\boldsymbol{\eta}^{\aa'}$ are determined by the arguments of the solutions of  $t^l+a_{l-1}t^{l-1}+\cdots+a_1 t=\nu_j$, $j=1,\dots,N$, where $\nu_1,\dots,\nu_N$ are the eigenvalues of $A_0=\frac{\partial \boldsymbol{F}}{\partial \yy}(\00,\00,\boldsymbol{0})$. Note that zero is not a solution of these polynomials since $A_0$ is invertible.	
\end{proof}

\bibliographystyle{plaindin_esp}

\end{document}